\documentclass[11pt,reqno]{amsart}
\usepackage{amsmath,amssymb,geometry,color,tikz-cd}
\usepackage{array,xcolor, tablefootnote}
\usepackage{adjustbox}
\usepackage[colorlinks=true, linkcolor=red!80!black, urlcolor=purple, citecolor=blue!70!black]{hyperref}
\usepackage{amssymb}

\usepackage{bbm,amsbsy}
\usepackage{enumitem}
\usepackage{upgreek}
\usepackage{bm}
\usepackage{stmaryrd}
\usepackage{mathrsfs}
\usepackage{comment}
\usepackage{appendix}
\usepackage{changepage}
\usepackage{makecell}
\usepackage{longtable}
\allowdisplaybreaks

\usetikzlibrary { decorations.pathmorphing, decorations.pathreplacing, decorations.shapes, }

\newcommand{\cB}{\mathcal{B}}
\newcommand{\cC}{\mathcal{C}}

\newcommand{\cE}{\mathcal{E}}
\newcommand{\cF}{\mathcal{F}}
\newcommand{\cG}{\mathcal{G}}
\newcommand{\cH}{\mathcal{H}}
\newcommand{\cJ}{\mathcal{J}}
\newcommand{\cL}{\mathcal{L}}
\newcommand{\cM}{\mathcal{M}}
\newcommand{\cN}{\mathcal{N}}
\newcommand{\cO}{\mathcal{O}}

\newcommand{\cQ}{\mathcal{Q}}

\newcommand{\cS}{\mathcal{S}}

\newcommand{\cU}{\mathcal{U}}

\newcommand{\cX}{\mathcal{X}}
\newcommand{\cY}{\mathcal{Y}}
\newcommand{\cW}{\mathcal{W}}

\newcommand{\hcS}{\widehat{\mathcal{S}}}
\newcommand{\tcS}{\widetilde{\mathcal{S}}}

\newcommand{\fm}{\mathfrak{m}}

\newcommand{\fX}{\mathfrak{X}}
\newcommand{\fY}{\mathfrak{Y}}

\newcommand{\fU}{\mathfrak{U}}
\newcommand{\fV}{\mathfrak{V}}

\newcommand{\bP}{\mathbb{P}}

\newcommand{\bR}{\mathbb{R}}
\newcommand{\bA}{\mathbb{A}}
\newcommand{\bQ}{\mathbb{Q}}
\newcommand{\bZ}{\mathbb{Z}}

\newcommand{\bG}{\mathbb{G}}
\newcommand{\bF}{\mathbb{F}}

\newcommand{\bN}{\mathbb{N}}
\newcommand{\bk}{\mathbbm{k}}

\newcommand{\pr}{\mathrm{pr}}
\newcommand{\red}{\mathrm{red}}

\newcommand{\cI}{\mathcal{I}}

\newcommand{\rmB}{\mathrm{B}}

\newcommand{\fH}{\mathfrak{H}}
\newcommand{\osXg}{\overline{\mathscr{X}}_{\!g}}
\newcommand{\QCoh}{\mathrm{QCoh}}

\DeclareMathOperator{\Aut}{Aut}

\DeclareMathOperator{\Hilb}{\mathbf{Hilb}}

\DeclareMathOperator{\Spec}{Spec}

\DeclareMathOperator{\id}{id}

\DeclareMathOperator{\Proj}{Proj}

\DeclareMathOperator{\Supp}{Supp}

\DeclareMathOperator{\rank}{rank}
\DeclareMathOperator{\diag}{diag}
\DeclareMathOperator{\rib}{rib}

\newcommand{\Chow}{\mathbf{Chow}}

\newcommand{\PGL}{\mathrm{PGL}}
\newcommand{\GL}{\mathrm{GL}}

\newcommand{\Hom}{\mathrm{Hom}}

\newcommand{\tPhi}{\widetilde{\Phi}}

\newcommand{\bmu}{\boldsymbol{\mu}}

\newcommand{\rmT}{\mathrm{T}}
\newcommand{\hyp}{\mathrm{hyp}}
\newcommand{\Stab}{\mathrm{Stab}}
\newcommand{\coker}{\mathrm{coker}}

\newcommand{\tC}{\widetilde{C}}

\newcommand{\tU}{\widetilde{U}}
\newcommand{\tV}{\widetilde{V}}

\newcommand{\K}{\mathrm{K}}

\newcommand{\oS}{\overline{S}}

\newcommand{\oM}{\overline{M}}
\newcommand{\oX}{\overline{X}}

\newcommand{\sX}{\mathscr{X}}
\newcommand{\sY}{\mathscr{Y}}
\newcommand{\sH}{\mathscr{H}}
\newcommand{\sR}{\mathscr{R}}
\newcommand{\sU}{\mathscr{U}}
\newcommand{\sV}{\mathscr{V}}
\newcommand{\sW}{\mathscr{W}}

\newcommand{\tsR}{\widetilde{\mathscr{R}}}
\newcommand{\tsH}{\widetilde{\mathscr{H}}}

\newcommand{\STR}{\overline{\mathrm{ST}}_R}
\newcommand{\osX}{\overline{\mathscr{X}}}
\newcommand{\Curves}{\mathcal{C}urves}
\newcommand{\Ext}{\mathrm{Ext}}
\newcommand{\ocM}{\overline{\mathcal{M}}}
\newcommand{\Def}{\mathrm{Def}}
\newcommand{\Sym}{\mathrm{Sym}}
\newcommand{\KSBA}{\mathrm{KSBA}}
\newcommand{\CY}{\mathrm{CY}}

\newcommand{\fL}{\mathfrak{L}}
\newcommand{\Gor}{\mathrm{Gor}}
\newcommand{\lci}{\mathrm{lci}}
\newcommand{\ccirc}{{\rm c},\circ}
\newcommand{\bfP}{\mathbf{P}}
\newcommand{\bfL}{\mathbf{L}}
\newcommand{\PuP}{\mathbb{P}^2_{x_0} \cup \mathbb{P}^2_{x_3}}

\newcommand{\ofM}{\overline{\mathfrak{M}}}

\makeatletter

\newcommand*{\@rowstyle}{}

\newcommand*{\rowstyle}[1]{
  \gdef\@rowstyle{#1}%
  \@rowstyle\ignorespaces%
}

\newcolumntype{=}{
  >{\gdef\@rowstyle{}}%
}

\newcolumntype{+}{
  >{\@rowstyle}%
}

\makeatother

\numberwithin{equation}{section}
\newcommand\numberthis{\addtocounter{equation}{1}\tag{\theequation}}

\newtheorem{prop} {Proposition} [section]
\newtheorem{thm}[prop] {Theorem} 
\newtheorem{defn-prop}[prop] {Definition-Proposition}

\newtheorem{lem}[prop] {Lemma}
\newtheorem{cor}[prop]{Corollary}
\newtheorem{prop-def}[prop]{Proposition-Definition}

\newtheorem{theorem}[prop]{Theorem}
\newtheorem{lemma}[prop]{Lemma}
\newtheorem{corollary}[prop]{Corollary}

\newtheorem{thm-defn}[prop]{Theorem-Definition}

\theoremstyle{definition}
 
\newtheorem{rem}[prop] {Remark} 

\newtheorem{defn}[prop]{Definition}

\newtheorem{expl}[prop] {Example}

\newtheorem{remark}[prop]{Remark}

\newtheorem{definition}[prop]{Definition}

\title{Hassett--Keel Program in genus four}

\date{\today} 
\author{Kenneth Ascher}
\address{Department of Mathematics, University of California, Irvine, CA, 92697, USA}
\email{kascher@uci.edu}
\author{Kristin DeVleming}
\address{Department of Mathematics,
University of California San Diego, La Jolla, CA, 92093, USA}
\email{kedevleming@ucsd.edu}
\author{Yuchen Liu}
\address{Department of Mathematics, Northwestern University, Evanston, IL 60208, USA}
\email{yuchenl@northwestern.edu}
\author{Xiaowei Wang}
\address{Department of Mathematics and Computer Science, Rutgers University, Newark, NJ 07102, USA}
\email{xiaowwan@rutgers.edu}

\begin{document}
\begin{abstract}
The Hassett--Keel program seeks to give a modular interpretation to the steps of the log minimal model program of $\overline{\mathcal{M}}_g$. The goal of this paper is to complete the Hassett--Keel program in genus four, supplementing earlier results of Casalaina-Martin--Jensen--Laza and Alper--Fedorchuk--Smyth--van der Wyck. The main tools we use are wall crossing for moduli spaces of pairs in the sense of K-stability and KSBA stability, and the recently constructed moduli spaces of boundary polarized Calabi--Yau surface pairs. We also give a construction of the hyperelliptic flip using a stackified Chow quotient which is expected to generalize to higher genus.
\end{abstract}

\maketitle
\tableofcontents

\section{Introduction}

The \textit{Hassett--Keel program} asks whether the steps of the log minimal model program on $\ocM_g$, the Deligne--Mumford moduli space of stable curves of genus $g \geq 2$, have natural modular interpretations. More specifically, let $\ocM_g$ denote the moduli stack and consider the divisor $K_{\ocM_g} + \alpha \delta$, where  $\delta = \delta_0 + \delta_1 + \dots + \delta_{\lfloor g/2 \rfloor }$
is the sum of the boundary divisors on $\ocM_g$ and $\alpha \in (0,1) \cap \bQ$.  If $\phi: \ocM_g \to \oM_g$ is the coarse moduli map and $\Delta_i$ the associated boundary divisor to $\delta_i$ on $\oM_g$, then 
\[ K_{\ocM_g} + \alpha \delta = \phi^*\left(K_{\oM_g} + \alpha(\Delta_0 + \Delta_2 + \dots + \Delta_{\lfloor g/2 \rfloor}) + \tfrac{1+ \alpha}{2} \Delta_1 \right). \]

For any rational number $\alpha \in (0,1)$ such that $K_{\ocM_g} + \alpha \delta$ is big, by \cite{BCHM} the log canonical ring is finitely generated and we may form the log canonical model 
\[ \oM_g(\alpha) = \Proj\left(\oplus_{m \ge 0} H^0( \ocM_g, \lfloor m(K_{\ocM_g} + \alpha \delta) \rfloor )\right).\]

For $\alpha = 1$, the divisor $K_{\ocM_g} + \delta$ is known to be ample (\cite{MumfordStability}, \cite[Theorem 5.7]{Has05}), so $\oM_g(1) = \oM_g$.  Hassett asked if $\oM_g(\alpha)$ admits a modular interpretation for all $\alpha \in (0,1) \cap \bQ$. \\

In general, the Hassett--Keel program is understood when $\alpha > \frac{2}{3} -\epsilon$
(see Section \ref{sec:HK} for more details, differences for $g \in \{2,3\}$ and discussion of terminology). Indeed, Hassett--Hyeon \cite{HH09, HH13} show that if $\alpha \in (\frac{9}{11}, 1)$, then $\oM_g(\alpha) = \oM_g$.  At $\alpha = \frac{9}{11}$, there is a non-trivial birational modification, a divisorial contraction,  contracting \textit{elliptic tails} to cuspidal ($A_2$) singularities. For $\alpha \in (\frac{7}{10}, \frac{9}{11}]$, there is an isomorphism $\oM_g(\alpha) \cong \oM_g^{\mathrm{ps}}$, where $\oM_g^{\mathrm{ps}}$ denotes the moduli space of Schubert pseudostable curves \cite{Sch91}. When $g \geq 3$, the authors show that the moduli space undergoes a flip at $\alpha = \frac{7}{10}$ corresponding to the replacement of \textit{elliptic bridges} by tacnodal ($A_3$) singularities.  When $g \geq 4$, Alper, Fedorchuk, Smyth, and van der Wyck \cite{AFSvdW, AFS17, AFS17b} showed that there is a flip at $\alpha = \frac{2}{3}$, corresponding to the replacement of Weierstrass genus two tails by ramphoid cuspidal ($A_4$) singularities. Note that Hyeon and Lee in \cite{HL14} showed that there is no wall crossing for $g=4$ when $\alpha\in (\frac{2}{3}, \frac{7}{10})$, which was generalized by \cite{AFSvdW, AFS17, AFS17b} to any $g\geq 4$.

When $g=4$, the range $\alpha\leq \frac{5}{9}$ is also understood: see \cite{CMJL14}  and \cite{Fed12, CMJL12} for partial earlier results. By VGIT for $(2,3)$-complete intersection curves in $\bP^3$, \cite{CMJL14} shows there are three critical values of $\alpha$ (and one additional terminal value) in the range $\alpha<\frac{5}{9}$ for which the moduli spaces change, corresponding to the last four rows of Table \ref{table:singularities}.

The purpose of this paper is to complete the Hassett--Keel program for genus four curves by proving the following.  Below, the notation $\ocM_g(b,a)$ (respectively, $\oM_g(b,a)$) for rational numbers $b > a$ represents $\ocM_g(\alpha)$ (respectively, $\oM_g(\alpha)$) for any $\alpha \in (a,b)$.

\begin{theorem}\label{introthm:59to23}
    For each rational number $\alpha \in  [\frac{5}{9},\frac{2}{3})$, there exists a smooth algebraic stack $\ocM_4(\alpha)$ of finite type with affine diagonal that parameterizes $\alpha$-stable curves of genus four, which are projective curves of genus four with lci singularities and ample canonical bundle satisfying certain  conditions (see Definition \ref{def:newmodels}). 
    The stack $\ocM_4(\alpha)$  admits a projective good moduli space isomorphic to $\oM_4(\alpha)$, the log canonical model of $\ocM_4$.  Furthermore, these stacks and good moduli spaces satisfy the following wall crossing diagram:
    \[   
    \begin{tikzcd}
    \ocM_4(\frac{2}{3},\frac{19}{29})\arrow{d}{}\arrow[hookrightarrow]{r}{} &
    \ocM_4(\frac{19}{29})\arrow{d}{}\arrow[hookleftarrow]{r}{} & \ocM_4(\frac{19}{29}, \frac{5}{9})\arrow{d}{}\arrow[hookrightarrow]{r}{} &\ocM_4(\frac{5}{9})\arrow{d}{} \arrow[hookleftarrow]{r}{} &\ocM_4(\frac{5}{9}, \frac{23}{44})\arrow{d}{}  \\
    \oM_4(\frac{2}{3},\frac{19}{29})\arrow{r}{} &
    \oM_4(\frac{19}{29}) & \oM_4(\frac{19}{29}, \frac{5}{9})\arrow[swap]{l}{\cong}{}\arrow{r}{} &\oM_4(\frac{5}{9}) & \oM_4(\frac{5}{9}, \frac{23}{44}) \arrow{l}{}
    \end{tikzcd}
    \]

    where the arrows in the top row are open immersions, the vertical arrows are good moduli space morphisms, and the arrows in the bottom row are projective birational morphisms. 

    \begin{itemize}
        \item As $\alpha$ decreases from $\alpha = \frac{19}{29} + \epsilon$ to $\alpha = \frac{19}{29}$, there is a contraction of the divisor $\Delta_2$ parametrizing the union of two genus two curves intersecting at a non-Weierstrass point of each curve, with image parametrizing curves with a separating $A_5$ singularity.
        \item  As $\alpha$ decreases from $\alpha = \frac{5}{9} + \epsilon$ to $\alpha = \frac{5}{9}$, the locus of curves with elliptic triboroughs (two genus one curves meeting at three points) is contracted to the locus of curves with $D_4$ singularities; the locus of hyperelliptic curves is contracted to a point parametrizing the unique hyperelliptic ribbon, and the locus of curves on quadric cones with a tacnodal singularity at the vertex of the cone is contracted to a curve $\Gamma \subset \oM_4(\frac{5}{9})$.
    \end{itemize} 
\end{theorem}

Combining this result with \cite{CMJL14}, which verifies the Hassett--Keel program in genus four for $\alpha \le \frac{5}{9}$ and \cite{HH09, HH13, AFSvdW, AFS17, AFS17b}, which verify the Hassett--Keel program in all genus for $\alpha > \frac{2}{3} - \epsilon$, we obtain the full Hassett--Keel program for $\ocM_4$.  

\begin{theorem}\label{introthm:fullHKprogram}
  The full Hassett--Keel program for $g = 4$ has the following description.
        \begin{enumerate}
            \item The critical values of $\alpha$ (called \emph{walls}) at which the moduli stacks change are 
            \[ (\alpha_1, \dots, \alpha_9) = \left( \frac{9}{11}, \frac{7}{10}, \frac{2}{3}, \frac{19}{29}, \frac{5}{9}, \frac{23}{44}, \frac{1}{2}, \frac{29}{60},  \frac{8}{17} \right). \] 
            For each $\alpha_i$ for $1 \le i  \le 9$, there is a sequence of stacks and good moduli spaces satisfying the following diagram:

\[
    \begin{tikzcd}
    \ocM_4(\alpha_{i-1},\alpha_i)\arrow{d}{}\arrow[hookrightarrow]{r}{} & \ocM_4(\alpha_i)\arrow{d}{}\arrow[hookleftarrow]{r}{} & \ocM_4(\alpha_i,\alpha_{i+1})\arrow{d}{} \\
    \oM_4(\alpha_{i-1},\alpha_i) \arrow[rightarrow]{r}{} & \oM_4(\alpha_i) \arrow[leftarrow]{r}{} & \oM_4(\alpha_i,\alpha_{i+1}) 
    \end{tikzcd}
\]
If $i = 1$, we take $\alpha_0 = 1$, and if $i = 9$, we take $\ocM_4(\alpha_9,\alpha_{10}) = \oM_4(\alpha_9,\alpha_{10}) = \emptyset$. For $i = 9$, the moduli space $\oM_4(\alpha_9) = \oM_4(\frac{8}{17}) = \{ *\}$ is the terminal model of $\oM_4$.
            \item The curve singularities parameterized by $\oM_4(\alpha)$ are as predicted in Table \ref{table:singularities}.
        \end{enumerate}
\end{theorem}

\begin{table}[htbp!]\renewcommand{\arraystretch}{1.5}
\caption{Hassett--Keel Program in genus four}\label{table:singularities}
\begin{tabular}{|=c|+c|+l|+l|}
\hline 
$i$ & $\alpha_i$   & \textbf{Locus removed} & \textbf{Singularity type introduced} \\ \hline \hline 
1 & $\frac{9}{11}$  &  elliptic tails & $A_2$        \\ 
2 & $\frac{7}{10}$  & elliptic bridges & $A_3$    \\ 
3 & $\frac{2}{3}$  &   \makecell[l]{genus two tails attached at a \\Weierstrass point}   &     $A_4$ \\ 
\rowstyle{\bfseries} 4 & $\mathbf{\frac{19}{29}}$  &   $\delta_2$, i.e. general genus two tails   &   $\boldsymbol{A_5^{\rm sep}}$           \\ 
\rowstyle{\bfseries} 5 & $\mathbf{\frac{5}{9}}$ & \makecell[l]{tacnodal curves glued at \\conjugate points} 
&     $\boldsymbol{A_5^{\rm nsep}}$    \\ 
\rowstyle{\bfseries} 5 & $\mathbf{\frac{5}{9}}$ & hyperelliptic curves & $\boldsymbol{A_6}$, $\boldsymbol{A_7^{\rm nsep}}$, $\boldsymbol{A_8}$, $\boldsymbol{A_9}$  \\ 
\rowstyle{\bfseries} 5 & $\mathbf{\frac{5}{9}}$ &  elliptic triboroughs  &     $\boldsymbol{D_4}$     \\ 
6 & $\frac{23}{44}$  &   \makecell[l]{cuspidal curves with hyperelliptic \\normalization}  & $A_7^{\rm sep}$   \\ 
7 & $\frac{1}{2}$ &  \makecell[l]{nodal curves with hyperelliptic \\normalization}  &  $\textrm{double conic} + \textrm{transversal conic}$  \\ 
8 & $\frac{29}{60}$ &   Gieseker-Petri divisor   &  \textrm{triple conic}\\ 
9 & $\frac{8}{17}$ & everything & terminal model \\
\hline
\end{tabular}
\end{table}

\begin{remark}[Previous expectations]\label{rmk:separating}
 In \cite[Section 3]{AFS16}, the authors predict values corresponding to walls for $\alpha > \frac{5}{9}$ for general genus $g$ and in \cite{CMJL14}, the authors predict values for the range $(\frac{2}{3}, \frac{5}{9})$ for $g = 4$.  The known walls and predictions are described in Table \ref{table:singularities}, with the previously unconfirmed walls emphasized in bold.  In the table, the notation $A_{2n+1}^{\rm sep}$ (resp. $A_{2n+1}^{\rm nsep}$) is used to denote curves with \textit{separating} $A_{2n+1}$ singularities, where normalization of the singularity disconnects the curve (resp. \textit{non-separating} $A_{2n+1}$ singularities, where the curve remains connected after normalization of the singularity).  

    Our results are mostly in line with previous predictions (see \cite{Fed12, AFS16,CMJL14}), aside from the case of $A_6$ singularities, which were predicted to be replaced at $\alpha = \frac{49}{83}$.
\end{remark}

\subsection{Method of proof}
To prove Theorem \ref{introthm:59to23} and Theorem \ref{introthm:fullHKprogram}, we first compare $\oM_4(\frac{2}{3} -\epsilon)$ and $\oM_4(\frac{5}{9})$, known by the earlier works \cite{AFSvdW, AFS17, AFS17b} and \cite{CMJL14}. Comparing these two known moduli spaces, one observes that they are isomorphic away from three disjoint loci in $\oM_4(\frac{5}{9})$: a point $[C_{2A_5}]$ parametrizing the $(2,3)$-complete intersection curve $C_{2A_5}$ in $\bP^3$ with two separating $A_5$ singularities, a point $[C_D]$ parametrizing the $(2,3)$-complete intersection curve $C_D$ with two $D_4$ singularities and three $A_1$ singularities, and a curve $\Gamma$ parametrizing certain singular $(2,3)$-complete intersection curves $C_{A,B}$ (which generically have both $A_5$ and $A_3$ singularities) on singular quadric surfaces (see \cite[Theorems 3.1 \& 4.1]{CMJL12}).  

We verify the predictions of the Hassett--Keel program by studying each of these three loci separately.  More precisely, in an open neighborhood of each point, we construct stacks and good moduli spaces satisfying a wall crossing diagram that will be glued together to form the global wall crossings for $\ocM_4(\alpha)$.  For the isolated points $[C_{2A_5}]$ and $[C_D]$, we use the recent theory of wall crossing for boundary polarized Calabi--Yau pairs from \cite{ABB+, BL24} (see Section \ref{sec:bpcy} and Section \ref{sec:logCYwallcrossing}).  For the curve $\Gamma$ introduced in Theorem \ref{introthm:59to23}, we construct a stack $\sX = \ocM_4(\frac{5}{9})$ directly and a suitable open substack $ \sU^+ \subset \sX$ of curves that admit degenerations to curves parameterized by $\Gamma$ that will serve as a local model of the wall crossing $\ocM_4(\frac{5}{9} + \epsilon) \to \ocM_4(\frac{5}{9} )$ (see Section \ref{sec:hyp-flip}). We prove that the stacks admit good moduli spaces, and then glue together the relevant substacks in a neighborhood of each locus to form the stacks $\ocM_4(\alpha)$ as $\alpha$ varies from $\frac{5}{9}$ to $\frac{2}{3} - \epsilon$ (see Theorem \ref{thm:new-models}).  Finally, by verifying that the log canonical divisor $K_{\ocM_4} + \alpha \delta$ descends to an ample line bundle on the good moduli space (see Theorem \ref{thm:projectivity}), we conclude that the good moduli space is the log canonical model of $\ocM_4$.

\subsection{Wall crossing for pairs} As mentioned above, we construct the wall crossing that replaces the curves $C_{2A_5}$ and $C_D$  using the theory of wall crossing for \textit{boundary polarized Calabi--Yau pairs} (bpCYs) as developed in \cite{ABB+,BL24}. 

If $X$ is a Fano variety and $D$ is an effective $\bQ$-divisor on $X$ such that $(X,D)$ is slc and $K_X + D \sim_{\bQ} 0$, we may consider moduli of spaces of pairs $(X,cD)$ (and their $\bQ$-Gorenstein degenerations) for all rational coefficients $c > 0$ such that $(X,cD)$ is slc.  For $c < 1$, the pair $(X,cD)$ is log Fano and the associated moduli spaces can be constructed using K-stability \cite{Xu25}.  These moduli spaces admit a wall crossing framework with respect to the coefficient $c$ by \cite{ADL19, Zho23}.  For $c > 1$, the pair $(X,cD)$ is log canonically polarized and the associated moduli spaces can be constructed using KSBA stability \cite{Kol23} and also admit a wall crossing framework with respect to $c$ by \cite{ABIP, MZ23}.  Finally, for $c =1$, in certain settings, the works \cite{ABB+,BL24} connect the K-moduli and KSBA moduli spaces with an additional boundary polarized Calabi--Yau moduli space.  We note that connections between wall crossing and the Hassett--Keel program have already been studied in the case of genus $g=3$; see for example \cite[Section 9.3.1]{ADL19} and \cite{HL10}.

In this paper, we apply this theory to the case where $X \subset \bP^3$ is a quadric surface and $D = \frac{2}{3} C$ for $C$ a $(2,3)$-complete intersection curve, and consider bpCY degenerations of $(X, D)$.  The curves $C_{2A_5}$ and $C_{D}$ each can be viewed in this context as bpCY pairs $(\bP^1 \times \bP^1, \frac{2}{3}C_{2A_5})$ and $(\bP^2 \cup \bP^2, \frac{2}{3}C_{D})$, and the bpCY wall crossing in an open neighborhood of each point provides a local model for the Hassett--Keel program.  More precisely, let $\ocM_4^{\CY}$ denote the bpCY moduli stack of all pairs $(X,D)$ (see Definition \ref{def:CY4stack}). Then it follows from \cite{ABB+, BL24} that $\ocM_4^{\CY}$ admits a projective good moduli space $\oM_4^{\CY}$. Note that in our setting, the stack $\ocM_4^{\CY}$ is already bounded as it does not contain type III bpCY pairs by  \cite[Theorem 16.11]{ABB+}.
Then we prove that there is an open neighborhood $\sW^{\CY} \subset \ocM_4^{\CY}$ of the above two bpCY pairs associated to $C_{2A_5}$ and $C_D$ such that the forgetful map $[(X,D)] \mapsto [D]$ is an isomorphism onto its image inside the stack of Gorenstein curves (Theorem \ref{thm:CY-forget-open}).  In a neighborhood of $[C_{2A_5}]$ and $[C_D]$, we use open substacks of $\sW^{\K} \subset \sW^{\CY}$ and $\sW^{\KSBA} \subset \sW^{\CY}$ and the associated bpCY wall crossing to realize the Hassett--Keel program.  By analysis of the pairs $(X,D)$ parameterized by $\sW^{\K}, \sW^{\KSBA},$ and $\sW^{\CY}$, one proves directly that $D$ is an $\alpha$-stable curve for the appropriate $\alpha$ in the Hassett--Keel program (Propositions \ref{prop:C2A5-replacement} and \ref{prop:D4replacement}) and therefore suitable open substacks give local models of the Hassett--Keel program near $[C_{2A_5]}$ and $[C_D]$. 

We note that this bpCY wall crossing does \textit{not} give the appropriate wall crossing in the Hassett--Keel program in a neighborhood of the third locus $\Gamma$, which must be replaced using a separate method.  Indeed, the KSBA moduli space of $(X, (\frac{2}{3}+\epsilon)C)$  on one side of the bpCY wall crossing compactifies the \textit{blowup} of $M_4$ along the hyperelliptic locus by \cite[Theorem 2]{DH21}, and thus cannot be a log canonical model of $\oM_4$.  In the K-moduli space of $(X, (\frac{2}{3}-\epsilon)C)$ on the other side of the bpCY wall crossing, by \cite[Remark 6.10]{ADL20}, the walls in the Hassett--Keel program for $\alpha < \frac{5}{9}$ obtained by VGIT are a strict proper subset of the K-moduli walls and thus this K-moduli space $\oM_4^{\K}$ does not agree with any of the Hassett--Keel models. 

We also note that the wall crossing for K-moduli spaces has applications to the Hassett--Keel program in genus six, see \cite{Zha24, Zha23}. In addition, one of the Hassett--Keel models in genus four appears naturally as a K-moduli space of Fano threefolds, see \cite{LZ24}.

Finally, we prove the following result of independent interest relating moduli spaces of genus four curves and K3 surfaces which is a slight improvement of \cite[Theorem 7.11]{BL24}.

\begin{thm}[see Theorem \ref{thm:bpCY-is-smooth}]
The moduli stack $\ocM_4^{\CY}$ is smooth. Moreover, its good moduli space $\oM_4^{\CY}$ is a normal projective variety and isomorphic to the Baily--Borel compactification of Kond\={o}'s ball quotient from \cite{Kondog4}. 
\end{thm}

\subsection{The hyperelliptic flip and the stackified Chow quotient}
The remaining locus to study in $\oM_4(\frac{5}{9})$ is a curve $\Gamma$ parameterizing a pencil $C_{A,B}$ of singular curves on a quadric cones, containing a unique point representing the hyperelliptic ribbon (see Definition \ref{def:GIT-strictly-semistable-curves}), all of which are $\alpha$-unstable for $\alpha > \frac{5}{9}$.  Thus, we wish to construct the wall crossing for the Hassett--Keel program in a neighborhood of this curve $\Gamma$.  By standard predictions in the Hassett--Keel program, the hyperelliptic ribbon must be replaced (as $\alpha$ increases) by all hyperelliptic genus four curves, and the wall crossing interchanging the hyperelliptic curves and ribbons is called the \textit{hyperelliptic flip}.   

In \cite{CMJL12, CMJL14}, the authors prove that $\oM_4(\frac{5}{9})$ is isomorphic to the Chow quotient $\fX^{\rm c}_4$ of cycles associated to $(2,3)$-complete intersection curves (see Theorem \ref{thm:cmjl-vgit}).  The authors conjectured (see, e.g. \cite[Last paragraph of Section 4]{CMJL12}) that this Chow quotient would provide a solution to the open problem of constructing the hyperelliptic flip.   We ultimately verify this conjecture based on work in Sections \ref{sec:hyp-and-rib}, \ref{sec:chow-ss-can}, and \ref{sec:hyp-flip}.  To do so, for our wall crossing framework, we first need a stack $\ocM_4(\frac{5}{9})$ of curves with open immersions $\ocM_4(\frac{5}{9} \pm \epsilon) \hookrightarrow \ocM_4(\frac{5}{9})$ admitting the Chow quotient as a good moduli space. The Chow scheme only parameterizes \emph{cycles}, and so we construct a new stack $\sX = \sX_4$ of curves which presents a ``stackification'' of the Chow quotient and prove that it admits a good moduli space isomorphic to the Chow quotient (see Remark \ref{rmk:cmjl-cycles}) which enables us to set $\ocM_4(\frac{5}{9})=\sX_4$. We also note that our stack $\sX_4$ is similar in spirit to the moduli stack of branchvarieties from \cite{HLFHJR}.

We first consider a stack $\sX_g$ parameterizing smoothable Gorenstein genus $g$ curves $C$ such that $\omega_C$ is ample and basepoint free, and that the Hilbert--Chow image of $C$ induced by $|\omega_C|$ is a Chow-semistable $1$-cycle in $\bP^{g-1}$ (see Definition \ref{def:sX_g} for a precise definition). 

\begin{thm}[see Theorem \ref{thm:sX_g-smooth}]\label{introthm:hyperelliptic}
Let $g\geq 3$ be an integer. Then $\sX_g$ is an irreducible algebraic stack of finite type with affine diagonal that contains $\cM_g$ as a dense open substack. Moreover, $\sX_g$ contains a smooth open neighborhood of the point corresponding to the unique hyperelliptic ribbon of genus $g$.
\end{thm}

The main purpose of the stack $\sX_g$ is to (potentially) provide a description of the hyperelliptic flip in the Hassett--Keel program for all $g$.  Toward this goal, we define stacks of hyperelliptic curves $\sH_g$ and ribbons $\sR_g$ (see Definition \ref{def:hyp-rib-stacks}) and give explicit presentations of these stacks as quotient stacks in Theorems \ref{thm:hyp-stack} and \ref{thm:rib-stack}.  From Theorem \ref{introthm:hyperelliptic}, we prove that both $\sH_g$ and $\sR_g$ are closed substacks of $\sX_g$ and find a specific smooth open substack $\sX_g^\circ$ containing $\sH_g$ and $\sR_g$ (see Definition-Proposition \ref{dp:Xgcirc}). This allows us to prove:

\begin{theorem}[see Theorem \ref{thm:sX_g-gms}]
    The open substack $\sX_g^\circ$ admits a good moduli space isomorphic to an open subscheme of the Chow quotient $\fX^{\rm c}_g$.  The image of $\sH_g \cup \sR_g$ is the unique Chow polystable point representing the doubled rational normal curve.
\end{theorem}

When $g=4$, we strengthen this result by proving $\sX_4^{\circ} = \sX_4$ and the following.

\begin{theorem}[see Theorem \ref{thm:sX-gms}]
    For $g = 4$, the stack $\sX_4$ admits a projective good moduli space isomorphic to the Chow quotient $\fX^{\rm c}_4$.
\end{theorem}

In Section \ref{sec:hyp-flip}, this result is used to prove the existence of the hyperelliptic flip for $g=4$.  The stack $\sX_4$ admits $\fX^{\rm c}_4 \cong \oM_4(\frac{5}{9})$ as its good moduli space, and we define an open substack $\sX_4^+ \subset \sX_4$ which admits a good moduli space $\fX_4^+$ with a proper morphism to $\oM_4(\frac{5}{9})$.  This serves as the final local model of the Hassett--Keel program for $\oM_4(\frac{5}{9}+\epsilon) \to \oM_4(\frac{5}{9})$ in a neighborhood of the curve $\Gamma$.  Furthermore, by defining a suitable local VGIT on $\sX^\circ_g$, we expect that this construction yields the hyperelliptic flip in every genus (see Remark \ref{rmk:highergenus-hyperellipticflip}).

\subsection*{Acknowledgement} We would like to thank Jarod Alper, Harold Blum, Sebastian Casalaina-Martin, Nathan Chen, Maksym Fedorchuk, Daniel Halpern-Leistner, Changho Han, Radu Laza, Yongnam Lee, Zhiyuan Li, and Junyan Zhao for helpful discussions. The authors were supported in part by the American Insitute of Mathematics as part of the AIM SQuaREs program. Research of KA was supported in part by NSF grant DMS-2302550 and a UCI Chancellor's Fellowship.  Research of KD was supported in part by NSF grant DMS-2302163.  Research of YL was supported in part by the NSF CAREER grant DMS-2237139 and an AT\&T Research Fellowship from Northwestern University.  Research of XW was supported in part by Simons Foundation Grant 631318 and MPS-TSM-00006636.

\section{Preliminaries}

Throughout this paper, we work over a fixed algebraically closed field $\bk$ of characteristic zero. 

\subsection{Stacks of curves}\label{sec:stack-curves}

Following \cite[\href{https://stacks.math.columbia.edu/tag/0D4Z}{Tag 0D4Y} and \href{https://stacks.math.columbia.edu/tag/0DMJ}{Tag 0DMJ}]{stacksproject}, let $\Curves$ be the functor from $\bk$-schemes to groupoids defined as 
\[
\Curves(S) := \left\{f: X \to S\left| \begin{array}{l} \textrm{$X$ is an algebraic space, $f$ is a flat proper morphism} \\
\textrm{of finite presentation, and $f$ has relative dimension $\leq 1$}
\end{array}
\right.\right\}.
\]
We note that our notation of $\Curves$ refers to $\bk\textrm{-}\Curves$ in \cite{stacksproject}. By \cite[\href{https://stacks.math.columbia.edu/tag/0D5A}{Tag 0D5A} and \href{https://stacks.math.columbia.edu/tag/0DSS}{Tag 0DSS}]{stacksproject} (see also \cite{dJHS11, Smy13}), we know that $\Curves$ is a quasi-separated algebraic stack locally of finite type over $\bk$. 

A \emph{curve} in this paper stands for a geometrically connected equidimensional scheme of dimension $1$ and of finite type over some field extension $\bk'$ of $\bk$. Note that a $\bk'$-point in $\Curves$ may not parameterize a curve in our sense, although all relevant open substacks below will satisfy this.

Let $g$ be a non-negative integer. 
By \cite[\href{https://stacks.math.columbia.edu/tag/0E1L}{Tag 0E1L} and \href{https://stacks.math.columbia.edu/tag/0E6K}{Tag 0E6K}]{stacksproject}, there is an open substack $\Curves_g^{\Gor}$ of $\Curves$ such that $\Curves_g^{\Gor}(S)$ consists of $(f:X\to S)\in \Curves(S)$ satisfying the following conditions:
\begin{enumerate}
    \item $f$ is Gorenstein whose fibers are equidimensional of dimension $1$;
    \item $f_*\cO_X = \cO_S$, and this holds after arbitrary base change;
    \item $R^1 f_*\cO_X$ is a locally free $\cO_S$-module of rank $g$.
\end{enumerate}
In particular, a $\bk'$-point $[C]\in \Curves_g^{\Gor}(\bk')$ for a field extension $\bk'/\bk$ parameterizes a Gorenstein proper curve $C$ over $\bk'$ such that $H^0(C, \cO_C)\cong \bk'$ and $H^1(C, \cO_C)\cong \bk'^g$. In this case, we say that $g$ is the \emph{(arithmetic) genus} of $C$.
We also define two open substacks of $\Curves_g^{\Gor}$ as 
\[
\Curves_g^{\lci +}\hookrightarrow \Curves_g^{\lci} \hookrightarrow \Curves_g^{\Gor},
\]
where $\Curves_g^{\lci}$ parameterizes curves in $\Curves_g^{\Gor}$ with local complete intersection singularities, and $\Curves_g^{\lci +}$ parameterizes curves in  $\Curves_g^{\lci}$ whose singular locus is finite; see \cite[\href{https://stacks.math.columbia.edu/tag/0E0J}{Tag 0E0J} and \href{https://stacks.math.columbia.edu/tag/0DZT}{Tag 0DZT}]{stacksproject}. Moreover, the algebraic stack $\Curves_g^{\lci +}$ is smooth by \cite[\href{https://stacks.math.columbia.edu/tag/0DZT}{Tag 0DZT}]{stacksproject}.

Let $\cM_g$ and $\ocM_g$ denote open substacks of $\Curves_g^{\Gor}$ parameterizing smooth curves of genus $g$ and Deligne--Mumford stable curves of genus $g$, respectively; see \cite[\href{https://stacks.math.columbia.edu/tag/0E82}{Tag 0E82} and \href{https://stacks.math.columbia.edu/tag/0E77}{Tag 0E77}]{stacksproject}. Then we have open immersions 
\[
\cM_g \hookrightarrow \ocM_g \hookrightarrow \Curves_g^{\lci +}\hookrightarrow\Curves_g^{\Gor}.
\]
By \cite[\href{https://stacks.math.columbia.edu/tag/0E86}{Tag 0E86}]{stacksproject}, we know that $\cM_g$ is open and dense in $\Curves_g^{\lci +}$. Since $\cM_g$ is irreducible by \cite{DM69}, so are  $\ocM_g$ and $\Curves_g^{\lci +}$.

\subsection{Good moduli spaces for algebraic stacks}

We recall the definition of good moduli spaces for algebraic stacks and their properties here.

\begin{defn}[\cite{alper}]
A quasi-compact and quasi-separated morphism $\pi: \cX \to \cY$ between algebraic stacks $\cX$ and $\cY$  is  a \emph{good moduli space morphism} if the following conditions hold.
\begin{enumerate}
    \item $\cO_{\cY} \to \pi_* \cO_{\cX}$ is an isomorphism;
    \item $\pi$ is \emph{cohomologically affine}, i.e. $\pi_*: \QCoh(\cX) \to \QCoh(\cY)$ is exact.
\end{enumerate}
Moreover, if $\cY $ is representable by an algebraic space $X$, then we say that $X$ is a \emph{good moduli space} of $\cX$. 
\end{defn}

\begin{defn}
Let $\cX$ be an algebraic stack of finite type over $\bk$. Let $x \in \cX(\bk)$ be a closed point. We say that $f: \cW \to \cX$ is a \emph{local quotient presentation around $x$} if:
\begin{enumerate}
    \item the stabilizer $G_x$ of $x$ is linearly reductive;
    \item there is an isomorphism $\cW \cong \left[ \Spec A / G_x \right]$, where $A$ is a finite type $\bk$-algebra; 
    \item the morphism $f$ is \'etale and affine; and
    \item there exists a point $w \in \cW(\bk)$ such that $f(w) = x$ and $f$ induces an isomorphism $G_w \cong G_x$. 
\end{enumerate}
The stack $\cX$ admits local quotient presentations if there exist local quotient presentations around all closed points $x \in \cX(\bk)$.
\end{defn}

The following result guarantees the existence of local quotient presentations.

\begin{thm}[\cite{AHR20}]\label{thm:ahr}
Let $\cX$ be an algebraic  stack of finite type 
over $\bk$ with affine diagonal. If $x \in \cX(\bk)$ is a closed point with linearly reductive stabilizer, then there exists a local quotient presentation $f: \cW \to \cX$ around $x$.
\end{thm}

Moreover, we have the following criterion that guarantees the existence of a good moduli space.

\begin{thm}[{\cite[Theorem 1.2]{AFS17}}]\label{thm:AFS-gms}
Let $\cX$ be an algebraic stack of finite type over $\bk$. Suppose the following conditions hold.
\begin{enumerate}
    \item For every closed point $x\in \cX(\bk)$, there exists a local quotient presentation $f: \cW\to \cX$ around $x$ such that $f$ is stabilizer preserving at closed points of $\cW$, and $f$ sends closed points to closed points.
    \item For every point $x\in \cX(\bk)$, the closed substack $\overline{\{x\}}$ admits a good moduli space.
\end{enumerate}
Then, $\cX$ admits a good moduli space.
\end{thm}

Next, we introduce the concept of $\Theta$-reductivity and S-completeness which are closely related to the existence of good moduli spaces.

\begin{defn}[\cite{AHLH18}]
For a DVR $R$ with a uniformizer $r$, we define two quotient stacks 
\[
\Theta_R:= [\bA^1_R /\bG_m]\quad\textrm{and}\quad\STR := [\Spec( R[s,t]/(st-r) )/\bG_m],
\]
where the $\bG_m$-action on $\bA_R^1$ is the standard multiplication action, and the $\bG_m$-action on $\STR$ has weights $1,-1$ in $s,t$. We denote by $0$ the unique closed point in $\Theta_R$ and $\STR$.

An algebraic stack $\cX$ is
\begin{enumerate}
    \item \emph{$\Theta$-reductive} if, for every DVR $R$, every morphism $\Theta_R \setminus 0 \to \cX$ admits a unique extension $\Theta_R\to \cX$; 
    \item \emph{S-complete} if, for every DVR $R$, every morphism $\STR \setminus 0 \to \cX$ admits a unique extension $\STR \to \cX$. 
\end{enumerate}
\end{defn}

\begin{thm}[\cite{AHLH18}]\label{thm:ahlh}
Let $\cX$ be an algebraic stack of finite type over $\bk$ with affine diagonal. Then $\cX$ admits a separated good moduli space $X$ if and only if $\cX$ is $\Theta$-reductive and S-complete. Moreover, the good moduli space $X$ is proper over $\bk$ if and only if $\cX$ satisfies the existence part of the valuative criterion for properness.
\end{thm}

Let us remark that it suffices to check the lifting criteria of Theorem \ref{thm:ahlh} for DVRs that are essentially of finite type over $\bk$; see \cite[Propositions 3.17, 3.41, Theorem 5.4, and Lemma A.11]{AHLH18}.

We collect some useful properties of good moduli spaces.

\begin{thm}[{\cite[Theorem 6.6]{alper}}]\label{thm:gms-universal}
Let $\cX$ be an algebraic stack of finite type over $\bk$. Assume that $\phi: \cX \to X$ is a good moduli space for $\cX$. Then, $\phi$ is universal for maps from $\cX$ to algebraic spaces.
\end{thm}

\begin{thm}[{\cite[Theorem 4.16(viii)]{alper}}]\label{thm:gms-normal}
Let $\cX$ be an algebraic stack of finite type over $\bk$ that admits a good moduli space $X$. If $\cX$ is reduced (resp. irreducible, normal), then so is $X$. 
\end{thm}

\begin{defn}[{\cite[Definition 6.1 and Remark 6.2]{alper}}]
Let $\cX$ be an algebraic stack of finite type over $\bk$ that admits a good moduli space $\phi: \cX \to X$. An open substack $\cU\hookrightarrow \cX$ is \emph{saturated} if $\cU = \phi^{-1}(\phi(\cU))$. If $\cU$ is saturated in $\cX$, then $\phi(\cU)\subset X$ is open and $\phi|_{\cU}:\cU \to \phi(\cU)$ is a good moduli space.
\end{defn}

\begin{thm}[{\cite[Proposition 7.9]{alper}}]\label{thm:gms-glue}
Let $\cX$ be an algebraic stack of finite type over $\bk$. Let $\{\cU_i\}_{i=1}^m$ be open substacks of $\cX$ such that $\cX = \cup_{i=1}^m \cU_i$ and for each $1\leq i\leq m$, there exists a good moduli space $\phi_i: \cU_i \to U_i$. Then there exists a good moduli space $\phi: \cX \to X$ and open sub-algebraic spaces $\tU_i \subset X$ such that $\tU_i \cong U_i$ and $\phi^{-1}(\tU_i) = \cU_i$ if and only if for each $1\leq i, j\leq m$, the intersection $\cU_i\cap \cU_j$ is saturated in $\cU_i$. 
\end{thm}

\subsection{Stacks from first three steps of Hassett--Keel program}\label{sec:HK}
We recall the notation \[\oM_g(\alpha) = \Proj\big(\oplus_{m \geq 0} H^0(\ocM_g, \lfloor m(K_{\ocM_g} + \alpha \delta)\rfloor)\big),\] the log canonical model of $(\ocM_g, \alpha \delta)$, from the introduction. The Hassett--Keel program predicts that as $\alpha$ decreases from one to zero there are certain ``critical'' values of $\alpha$ where birational contractions occur to produce new modular compactifications of $\cM_g$ with increasingly worse singularities.  In this section, we briefly review previously known results. 

First, we recall that an \emph{elliptic tail} is a connected genus $1$ component of a stable curve meeting the rest of the curve at a single point. When $\alpha > \frac{7}{10}$ and $\alpha\neq \frac{9}{11}$, the stack $\ocM_g(\alpha)$ is Deligne--Mumford for $g \geq 2$.  Indeed:
\begin{itemize}
    \item When $\alpha > \frac{9}{11}$, there is an isomorphism $\ocM_g(\alpha) \cong \ocM_g$ (see e.g. \cite[Section 4]{Has05} or \cite[Section 1]{HH09}; the point is that $K_{\ocM_g} + \alpha \delta$ is ample provided $\frac{9}{11} < \alpha \leq 1$).
    \item  The first critical $\alpha$ value occurs at $\alpha = \frac{9}{11}$. Indeed, by the work of \cite{HH09} ($g \geq 4)$, \cite{HL10} ($g = 3$), and \cite{Has05} ($g=2$), we have that $\oM_g(\frac{9}{11})$ is the coarse space of $\ocM_g^{\rm ps}$, where $\ocM_g^{\rm ps}$ denotes Schubert's moduli stack of pseudo-stable curves. There is a divisorial contraction $\oM_g \to \oM_g(\frac{9}{11})$ induced by $\ocM_g \to \ocM_g^{\rm ps}$ which corresponds to replacing elliptic tails with cusps. The curves parameterized by $\ocM_g^{\rm ps}$ have ample canonical sheaf, and at worst $A_2$-singularities not containing $A_1$-attached elliptic curves (i.e. elliptic tails). Pseudostable curves have finite automorphisms for $g > 2$ by \cite[Proof of Lemma 5.3]{Sch91}, and so the stack $\ocM_g^{\rm ps}$ is Deligne--Mumford.  Moreover, in the above mentioned works, the authors show that $\oM_g(\alpha) \cong \oM_g(\frac{9}{11})$ whenever $\frac{7}{10} < \alpha \leq \frac{9}{11}$.
\end{itemize}

Before stating the next lemma, we summarize the curves $C$ by parameterized by $\ocM_g(\alpha)$ for $\alpha \in (\frac{2}{3} - \epsilon, 1)$. We use the notation $A_{\leq n}$ to denote the class of $A_1, A_2, \dots, A_n$ singularities. For the definition of \emph{elliptic chains} and \emph{Weierstrass chains}, we refer the reader to \cite[Definitions 2.3 \& 2.4]{AFSvdW}.

\begin{defn}
\cite[Definition 2.5 \& Theorem 2.7]{AFSvdW}\label{def:HKsing}
The stacks $\ocM_g(\alpha)$ parameterize proper reduced Gorenstein  curves $C$ of arithmetic genus $g \ge 4$ with $\omega_C$ ample satisfying the following constraints on their singularities and irreducible components depending on the value $\alpha \in (\frac{2}{3}-\epsilon, 1]$.
    \begin{itemize}
        \item For $\alpha \in (\frac{9}{11}, 1]$: $C$ has only $A_1$ singularities.
        \item For $\alpha = \frac{9}{11}$: $C$ has only $A_1, A_2$ singularities. 
        \item For $\alpha \in (\frac{7}{10}, \frac{9}{11})$: $C$ has only $A_1$ and $A_2$ singularities, and does not contain $A_1$-attached elliptic tails.
        \item For $\alpha = \frac{7}{10}$: $C$ has only $A_{\leq 3}$ singularities, and does not contain $A_1, A_3$-attached elliptic tails. 
        \item For $\alpha \in (\frac{2}{3}, \frac{7}{10})$: $C$ has only $A_{\leq 3}$ singularities, and does not contain: 
        \begin{itemize}
            \item $A_1, A_3$-attached elliptic tails, 
            \item $A_1/A_1$-attached elliptic chains. 
        \end{itemize}
        \item For $\alpha = \frac{2}{3}$: $C$ has only $A_{\leq 4}$ singularities, and does not contain:
        \begin{itemize}
            \item $A_1, A_3, A_4$-attached elliptic tails, 
            \item $A_1/A_1, A_1/A_4, A_4/A_4$-attached elliptic chains.
        \end{itemize}
       \item For $\alpha \in (\frac{2}{3}-\epsilon, \frac{2}{3})$: $C$ has only $A_{\leq 4}$-singularities, and does not contain:
\begin{itemize}
 \item $A_1,A_3, A_4$-attached elliptic tails,
 \item $A_1/A_1, A_1/A_4, A_4/A_4$-attached elliptic chains, or
 \item $A_1$-attached Weierstrass chains.
    \end{itemize}
\end{itemize}
Moreover, each $\ocM_g(\alpha)$ is an algebraic stack of finite type over $\bk$ with affine diagonal. Indeed, each $\ocM_g(\alpha)$ is an open substack of $\Curves_g^{\lci +}$.
\end{defn}

The following summarizes the first three steps of the Hassett--Keel program in any genus $g\geq 2$. We note that there are minor differences when $g \in \{2,3\}$, so we state those cases separately.

\begin{thm}[Hassett--Keel program for $g \geq 4$; {\cite{HH09, HH13, AFSvdW, AFS17, AFS17b}}]
For every $\alpha\in (\frac{2}{3}-\epsilon, 1]$ and every $g\geq 4$, the algebraic stack $\ocM_g(\alpha)$ admits a projective good moduli space isomorphic to $\oM_g(\alpha)$. Moreover, for each critical value $\alpha_c\in \{\frac{9}{11}, \frac{7}{10}, \frac{2}{3}\}$, we have a wall crossing diagram 
\[   
    \begin{tikzcd}
    \ocM_g(\alpha_c+\epsilon)\arrow[hookrightarrow]{r}{}\arrow{d}{} &
    \ocM_g(\alpha_c)\arrow{d}{}\arrow[hookleftarrow]{r}{} & \ocM_g(\alpha_c-\epsilon)\arrow{d}{}\\
    \oM_g(\alpha_c+\epsilon)\arrow[rightarrow]{r}{}&
    \oM_g(\alpha_c)\arrow[leftarrow]{r}{}& \oM_g(\alpha_c-\epsilon)
    \end{tikzcd}
 \]
 where the top arrows are open immersions of algebraic stacks, the bottom arrows are projective morphisms, and the vertical arrows are good moduli space morphisms.
\end{thm}

\begin{thm}[Hassett--Keel program for $g=2$ {\cite{Has05, HL07}} and $g=3$ \cite{HL10}]

Let $g \in \{2,3\}$.
    For every $\alpha \in (\frac{7}{10}, 1]$, the algebraic stack $\ocM_g(\alpha)$ admits a projective coarse moduli space isomorphic to $\oM_g(\alpha)$. In particular:
    \begin{itemize}
        \item if $\alpha \in (\frac{7}{10}, \frac{9}{11}]$ then $\oM_g(\alpha)$ is the coarse space of the stack of pseudostable curves, and when $\alpha > \frac{9}{11}$, the space $\oM_g(\alpha)$ is the coarse space of the Deligne--Mumford compactification. 
        \item When $g=2$, the space $\oM_2(\alpha)$ becomes a point for $\alpha = \frac{7}{10}$.
        \item When $g=3$, the space $\oM_g(\alpha)$ admits a projective good moduli space for all $\alpha \in (\frac{5}{9}, \frac{7}{10})$ with explicit descriptions given in \cite{HL10}. When $\alpha = \frac{5}{9}$, the space $\oM_3(\alpha)$ becomes a point. 
    \end{itemize} 
\end{thm}

The next lemma is special for $g=4$.

\begin{lem}\label{lem:2/3+DM}
When $\alpha\in (\frac{2}{3}, \frac{7}{10})$, the stack $\ocM_4(\alpha)$ is a proper Deligne--Mumford stack.
\end{lem}

\begin{proof}
By Definition \ref{def:HKsing}, we know that $\ocM_g(\alpha)$ parameterizes curves with $A_{\leq 3}$-singularities, not containing $A_1$, $A_3$-attached elliptic tails or $A_1/A_1$-attached elliptic chains. Since $\ocM_4(\alpha)$ admits a proper good moduli space by \cite{HH13, AFSvdW, AFS17}, it suffices to show that every closed point in the stack has finite stabilizer.

Assume to the contrary that there exists a closed point $[C]\in \ocM_4(\alpha)$ with infinite stabilizer. By \cite[Corollary 8.8]{HH13} (see also \cite[Remark 2.28]{AFSvdW}), $C$ admits an $A_1/A_1$-attached rosary of length $3$. Since the rosary has arithmetic genus $2$,  the remaining curve has to be either an $A_1/A_1$-attached elliptic bridge or two $A_1$-attached elliptic tails, which are not allowed in $\ocM_4(\alpha)$. This is a contradiction. The proof is finished.
\end{proof}

\subsection{Chow stability and VGIT}\label{sec:chow-vgit}

\subsubsection{Chow stability}
We first introduce the relevant Chow schemes and stacks for curves of genus $g\geq 3$. 
Let $\boldsymbol{C}_{1,2g-2}(\bP^{g-1})$ be the Chow scheme of $\bP^{g-1}$ parameterizing $1$-cycles of degree $2g-2$. Denote by $\Chow_{g,1}$ the seminormalization of the closure of the locus parameterizing the $1$-cycles of  smooth canonical curves of genus $g$ in $\boldsymbol{C}_{1,2g-2}(\bP^{g-1})$.

Let $L_\infty$ be the restriction of the Chow polarization on the Chow scheme $\boldsymbol{C}_{1,2g-2}(\bP^{g-1})$ to $\Chow_{g,1}$. 

\begin{defn}\label{def:chowss}
 We say that a cycle in $\Chow_{g,1}$ is \emph{Chow (poly/semi)stable} if it is GIT (poly/semi)stable with respect to $L_{\infty}$.
Let $\Chow_{g,1}^{\rm ss}$ be the open subscheme of $\Chow_{g,1}$ parametrizing Chow semistable cycles. Then we have the Chow quotient stack
\[
\sX_g^{\rm c} := [\Chow_{g,1}^{\rm ss}/ \PGL(g)].
\]
with a good moduli space $\phi_g^{\rm c}:\sX_g^{\rm c}\to \fX_g^{\rm c}=\Chow_{g,1}^{\rm ss}\sslash \PGL(g)$.
\end{defn}

\begin{lem}\label{lem:sm-chow}
A smooth canonical curve of genus $g$ in $\bP^{g-1}$ is Chow stable. A doubled rational normal curve in $\bP^{g-1}$ is Chow polystable.
\end{lem}

\begin{proof}
The first part follows from \cite[Corollary 2.8]{PazLopez} (see also \cite[Section XIV.3]{ACGH2}). The second part follows from the facts that the rational normal curve is always Chow polystable (see \cite[Corollary 5.3]{Kempf1978}) and any multiple of a Chow polystable cycle remains Chow polystable.  Alternatively, one may deduce that a doubled rational normal curve in $\bP^{g-1}$ is Chow polystable using \cite[Corollary 5.1]{Kempf1978} together with the fact that the centralizer of  the automorphism group of a doubled rational normal curve in $\mathrm{PGL}(g)$ is trivial.
\end{proof}

\subsubsection{Chow stability and VGIT in genus $4$}

Let us specialize to the genus $g=4$ case. For simplicity, we omit the subscript $4$ from $\sX_4^{\rm c}$, $\fX_4^{\rm c}$, etc. We also denote by $\sX^{\rm s}\hookrightarrow \sX^{\rm c}$ (resp. $\fX^{\rm s}\hookrightarrow\fX^{\rm c}$) the open substack (resp. open subscheme) parameterizing Chow stable cycles. From the theory of GIT we know that $\sX^{\rm s}$ is a Deligne--Mumford stack with a coarse moduli space $\fX^{\rm s}$.

Let $\bP^9 = \bP(H^0(\bP^3, \cO_{\bP^3}(2)))$ be the parameter space of quadric surfaces in $\bP^3$. Let $\cQ\subset \bP^3\times \bP^9$ with second projection $\pr_2: \cQ\to \bP^9$ be the universal family of quadric surfaces. Let $E_3:=(\pr_2)_* \cO_{\cQ}(3,0)$ be the vector bundle over $\bP^9$ whose fiber over $[Q]\in \bP^9$ is given by $H^0(Q, \cO_Q(3))$. Let $\bP(E_3):=\Proj_{\bP^9} \Sym(E_3^\vee)\xrightarrow{\pi} \bP^9$ be the projective bundle associated to $E_3$.  Denote by 
\[
\eta:=\pi^* \cO_{\bP^9}(1) \quad \textrm{and}\quad \xi:=\cO_{\bP(E)}(1).
\]

Let $U_{(2,3)}\subset \bP(E_3)$ be the open subscheme parameterizing $(2,3)$-complete intersections in $\bP^3$. We have a Hilbert--Chow morphism $\varphi_{(2,3)}: U_{(2,3)} \to \boldsymbol{C}_{1,6}(\bP^3)$. Since a general curve in $U_{(2,3)}$ is a smooth canonical curve of genus $4$, as abuse of notation we have a factorization \[
\varphi_{(2,3)}: U_{(2,3)} \to \Chow_{4,1}.
\]

Let $W$ denote the closure of the graph of $\varphi_{(2,3)}$ in $\bP(E_3)\times \Chow_{4,1}$ with reduced scheme structure. Denote by $p_1: W\to \bP(E_3)$ and $p_2:W\to \Chow_{4,1}$ the first and second projections respectively.

\begin{defn}\label{def:23VGIT}
 For each $t\in (0, \frac{2}{3})\cap \bQ$, we choose $\delta \in (0, \min\{t, \frac{2}{9}\})\cap \bQ$ and define the VGIT quotient stack $\sX_t$ and space $\fX_t$ of slope $t$ as
\[
\sX_t:=[W^{\rm ss}(N_t)/\PGL(4)]\quad \textrm{and}\quad\fX_t:=W^{\rm ss}(N_t)\sslash\PGL(4),
\]
where 
\[
N_t:=\tfrac{2-3t}{2-3\delta}p_1^*(\eta+\delta \xi) + \tfrac{t-\delta}{2-3\delta}p_2^* L_\infty. 
\]
These VGIT quotients are independent of the choice of $\delta$, see e.g. \cite{CMJL14, LO21, ADL20}.
\end{defn}

In \cite{CMJL12}, the authors gave a complete description of cycles in $\fX^{\rm c}$ and $\sX^{\rm c}$.  To state their result, we first define the following curves. 

\begin{defn}[\cite{CMJL12}]\label{def:GIT-strictly-semistable-curves}
    Denote by $C_{2A_5}$ the $(2,3)$ complete intersection curve $V(x_0x_3 - x_1x_2, x_0x_2^2 + x_1^2 x_3) \subset \bP^3$, depicted in Figure \ref{fig:C2A5andCD}. On the quadric surface $\bP^1 \times \bP^1 = V(x_0 x_3 -x_1 x_2)$, this curve is the union of two sections of $\cO(1,0)$, each meeting a section of $\cO(1,3)$ in an $A_5$-singularity. In particular, $C_{2A_5}$ has two separating $A_5$-singularities.

    Denote by $\PuP = V(x_0x_3) \subset \bP^3$ and by $C_{D}$ the $(2,3)$ complete intersection curve $V(x_0x_3, x_1^3 + x_2^3) \subset \PuP$, depicted in Figure \ref{fig:C2A5andCD}.  This curve is the union of two sets of three lines meeting at a point. In particular, $C_D$ has two $D_4$-singularities and three $A_1$-singularities.

    For $(A,B) \in \bk^2\setminus \{(0,0)\}$, denote by $C_{A,B}$ the pencil of $(2,3)$ complete intersection curves 
    \[ V(x_2^2 - x_1 x_3, Ax_1^3 + Bx_0 x_1 x_2 + x_0^2x_3) \subset \bP^3\]
    contained in the singular quadric surface $V(x_2^2 - x_1 x_3)$.  If $4A /B^2 \ne 0,1$, these curves have an $A_5$ singularity at a smooth point of the quadric surface and an $A_3$ singularity at the cone point.  If $4A/B^2 = 0$, there is an additional $A_1$ singularity at a smooth point of the quadric surface.  If $4A/B^2 = 1$, this curve is a canonical ribbon of genus four (see section \ref{sec:ribbons}).
\end{defn}

    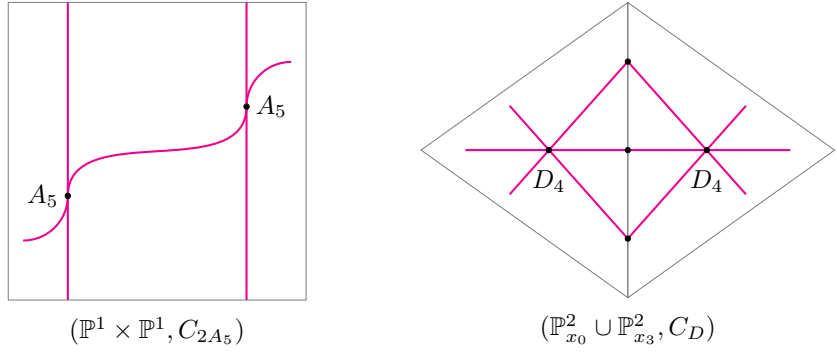
\begin{figure}[h]
    \resizebox{.3\textwidth}{!}{
    \begin{tabular}{c}
    {\scalebox{4}{\begin{tabular}{c}
\begin{tikzpicture}[gren0/.style = {draw, circle,fill=greener!80,scale=.7},gren/.style ={draw, circle, fill=greener!80,scale=.4},blk/.style ={draw, circle, fill=black!,scale=.08},plc/.style ={draw, circle, color=white!100,fill=white!100,scale=0.02},smt/.style ={draw, circle, color=gray!100,fill=gray!100,scale=0.02},lbl/.style ={scale=.2}] 

\draw [-,color=gray] (0,0) to (0,4) to (4,4) to (4,0) to (0,0);

\draw [color=magenta,thick] (.8,0) -- (.8,4);
\draw [color=magenta,thick] (3.2,0) -- (3.2,4);
\draw [-,color=magenta,thick] (.2,.8) to [out=0, in=-90] (.8,1.4) to [out=90,in=-90] (3.2,2.6) to [out=90, in=180] (3.8,3.2);
\filldraw (.8,1.4) circle (1pt);
\filldraw (3.2,2.6) circle (1pt);

\node[below, node font=\small] at (2,-.1) {$(\bP^1 \times \bP^1, C_{2A_5})$};
\node[left, node font=\small] at (.8,1.4) {$A_5$};
\node[right, node font=\small] at (3.2,2.6) {$A_5$};

\end{tikzpicture}
\end{tabular}}}
    \end{tabular}
    } \qquad
    \resizebox{.4\textwidth}{!}{
    \begin{tabular}{c}
    {\scalebox{4}{\begin{tabular}{c}
\begin{tikzpicture}[gren0/.style = {draw, circle,fill=greener!80,scale=.7},gren/.style ={draw, circle, fill=greener!80,scale=.4},blk/.style ={draw, circle, fill=black!,scale=.08},plc/.style ={draw, circle, color=white!100,fill=white!100,scale=0.02},smt/.style ={draw, circle, color=gray!100,fill=gray!100,scale=0.02},lbl/.style ={scale=.2}] 

\draw [-,color=gray] (0,-2) -- (0,2) -- (2.8,0) -- (0, -2);
\draw [-,color=gray] (0,-2) -- (0,2) -- (-2.8,0) -- (0, -2);

\draw [color=magenta,thick] (-2.2,0) -- (2.2,0);
\draw [color=magenta,thick] (-1.6,-.6) -- (0,1.2);
\draw [color=magenta,thick] (-1.6,.6) -- (0,-1.2);
\draw [color=magenta,thick] (1.6,.6) -- (0,-1.2);
\draw [color=magenta,thick] (1.6,-.6) -- (0,1.2);

\filldraw (-1.07,0) circle (1pt);
\filldraw (1.07,0) circle (1pt);
\filldraw (0,0) circle (1pt);
\filldraw (0,1.2) circle (1pt);
\filldraw (0,-1.2) circle (1pt);

\node[below, node font=\small] at (0,-2.1) {$(\PuP, C_D)$};
\node[below, node font=\small] at (-1.07,-.15) {$D_4$};
\node[below, node font=\small] at (1.07,-.15) {$D_4$};
\end{tikzpicture}
\end{tabular}}}
    \end{tabular}
    }
    \caption{A depiction of the curves $C_{2A_5}$ and $C_{D}$.}
    \label{fig:C2A5andCD}
    \end{figure}

	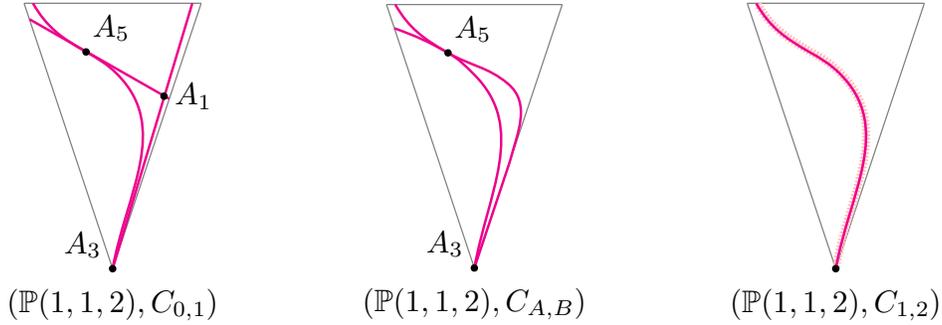
\begin{figure}[h]
		\resizebox{.25\textwidth}{!}{
			\begin{tabular}{c}
				{\scalebox{4}{\begin{tabular}{c}
	\begin{tikzpicture}[gren0/.style = {draw, circle,fill=greener!80,scale=.7},gren/.style ={draw, circle, fill=greener!80,scale=.4},blk/.style ={draw, circle, fill=black!,scale=.08},plc/.style ={draw, circle, color=white!100,fill=white!100,scale=0.02},smt/.style ={draw, circle, color=gray!100,fill=gray!100,scale=0.02},lbl/.style ={scale=.2}]

		\draw [-,color=gray] (0,0) to (1, 3);
		\draw [-,color=gray] (0,0) to (-1, 3);

		\draw [-,color=gray] (-1, 3) to (1,3);
		
		\draw [-,color=magenta,thick] (0,0) to (0.9,3);
		\draw [-,color=magenta,thick] (0,0) to [out=80,in=-45] (.05,2.2) to [out=135,in=-30] (-.3,2.45) to [out=150,in=-60] (-.9,3);
		\draw [-,color=magenta,thick] (-.94,2.82) to (0.64,1.92);

		\filldraw (0,0) circle (1pt);
		\filldraw (-.3,2.45) circle (1pt);
		\filldraw (.58,1.95) circle (1pt);
		
		\node[below, node font=\small] at (0,-.1) {$(\bP(1,1,2), C_{0,1})$};
		\node[above left, node font=\small] at (0,0) {$A_3$};
		\node[above right, node font=\small] at (-.35,2.45) {$A_5$};
		\node[right, node font=\small] at (.58,1.95) {$A_1$};
	\end{tikzpicture}
\end{tabular}}}
			\end{tabular}
		} \qquad
		\resizebox{.26\textwidth}{!}{
			\begin{tabular}{c}
				{\scalebox{4}{\begin{tabular}{c}
	\begin{tikzpicture}[gren0/.style = {draw, circle,fill=greener!80,scale=.7},gren/.style ={draw, circle, fill=greener!80,scale=.4},blk/.style ={draw, circle, fill=black!,scale=.08},plc/.style ={draw, circle, color=white!100,fill=white!100,scale=0.02},smt/.style ={draw, circle, color=gray!100,fill=gray!100,scale=0.02},lbl/.style ={scale=.2}]

		\draw [-,color=gray] (0,0) to (1, 3);
		\draw [-,color=gray] (0,0) to (-1, 3);

		\draw [-,color=gray] (-1, 3) to (1,3);

		\draw [-,color=magenta,thick] (0,0) to [out=80,in=-45] (0,2.2) to [out=135,in=-30] (-.3,2.45) to [out=150,in=-60] (-.9,3);
		\draw [-,color=magenta,thick] (0,0) to (.4,1.2) to [out=75,in=-40] (.38,2.1) to [out=140,in=-30] (-.3,2.45) to [out=150,in=-20] (-.91,2.73);

		\filldraw (0,0) circle (1pt);
		\filldraw (-.3,2.45) circle (1pt);
		
		\node[below, node font=\small] at (0,-.1) {$(\bP(1,1,2), C_{A,B})$};
		\node[above left, node font=\small] at (0,0) {$A_3$};
		\node[above right, node font=\small] at (-.35,2.45) {$A_5$};
	\end{tikzpicture}
\end{tabular}}}
			\end{tabular}
		}
		 \qquad
		\resizebox{.25\textwidth}{!}{
			\begin{tabular}{c}
				{\scalebox{4}{\begin{tabular}{c}
	\begin{tikzpicture}[gren0/.style = {draw, circle,fill=greener!80,scale=.7},gren/.style ={draw, circle, fill=greener!80,scale=.4},blk/.style ={draw, circle, fill=black!,scale=.08},plc/.style ={draw, circle, color=white!100,fill=white!100,scale=0.02},smt/.style ={draw, circle, color=gray!100,fill=gray!100,scale=0.02},lbl/.style ={scale=.2}]

		\draw [-,color=gray] (0,0) to (1, 3);
		\draw [-,color=gray] (0,0) to (-1, 3);

		\draw [-,color=gray] (-1, 3) to (1,3);

        \draw [-,color=pink, decorate,decoration={snake, amplitude=.4mm,segment length=.4mm}] (0,0) to [out=80,in=-45] (.05,2.2) to [out=135,in=-30] (-.3,2.45) to [out=150,in=-60] (-.9,3);
		\draw [-,color=magenta,thick] (0,0) to [out=80,in=-45] (.05,2.2) to [out=135,in=-30] (-.3,2.45) to [out=150,in=-60] (-.9,3);

		\filldraw (0,0) circle (1pt);

		\node[below, node font=\small] at (0,-.1) {$(\bP(1,1,2), C_{1,2})$};

	\end{tikzpicture}
\end{tabular}}}
			\end{tabular}
		}
		\caption{A depiction of the curves $C_{A,B}$.}
		\label{fig:CAB}
	\end{figure}

\begin{thm}[{\cite[Theorem 3.1]{CMJL12}}]\label{thm:CMJL-Chow-GIT}
The Chow stable locus $\sX^{\rm s}$ precisely parameterizes curves $C\in |\cO_Q(3)|$ such that $Q\subset \bP^3$ is a quadric surface with $\rank Q \geq 3$, $C$ has at worst $A_4$-singularities at smooth points of $Q$, and at worst $A_2$-singularities at the vertex of $Q$ if $\rank Q = 3$.

The Chow strictly semistable locus $\sX^{\rm c} \setminus \sX^{\rm s}$ precisely parameterizes curves $C \in |\cO_Q(3)|$ such that $Q \subset \bP^3$ is a quadric surface with $\rank Q \ge 2$, and, if $C$ is reduced, $Q$ is uniquely determined and $(Q,C)$ satisfies:
    \begin{enumerate}
        \item if $\rank Q = 4$, then either (a) $C$ contains a singularity of type $D_4$ or $A_5$, or (b) $C$ does not contain an irreducible component of degree $\le 2$ and contains a singularity of type $A_k$, $k \ge 6$.
        \item if $\rank Q = 3$, then $C$ has at worst an $A_k$ singularity at the vertex of $Q$ and either (a) $C$ contains a $D_4$ or $A_5$ singularity at a smooth point of $Q$ or an $A_3$ singularity at the vertex of $Q$, or (b) $C$ does not contain an irreducible component of degree $1$ and contains a singularity of type $A_k$, $k \ge 6$ at a smooth point of $Q$ or of type $A_k$, $k \ge 4$ at the vertex of $Q$. 
        \item if $\rank Q = 2$, then $C$ meets the singular locus of $Q$ in 3 distinct points.  These curves are called \textit{elliptic triboroughs}.
    \end{enumerate}
If $C$ is not reduced, then $C$ is a genus $4$ ribbon.

Moreover, the strictly Chow polystable locus $\fX^{\rm c}\setminus \fX^{\rm s}$ is precisely the union of two isolated points $C_{2A_5}$, $C_{D}$ and a disjoint rational curve $\Gamma$ parameterizing $C_{A,B}$ (which contains a unique point representing the cycle of a ribbon).
\end{thm}

\begin{remark}\label{rmk:GIT-stable-properties}
    From the description above, it is clear that any Chow stable curve $C$ lies on a normal quadric surface $Q$ such that $(Q, \frac{2}{3}C)$ is klt.  Furthermore, any Chow polystable point $C$ other than the ribbon lies on a quadric surface such that $(Q, \frac{2}{3}C)$ is slc.

    Furthermore, it will follow from Proposition \ref{prop:U+open} that all Chow stable curves are contained in $\ocM_4(\frac{2}{3} - \epsilon)$ as in Definition \ref{def:HKsing}. 
\end{remark}

The following lemma was used implicitly in \cite{CMJL12}. We provide a proof for the readers' convenience.

\begin{lem}\label{lem:Chow-normal}
The Chow semistable open subscheme $\Chow_{4,1}^{\rm ss}$ is normal. In particular, the stack $\sX^{\rm c}$ and its good moduli space $\fX^{\rm c}$ are both normal.
\end{lem}

\begin{proof}
We look at the second projection map $p_2: W\to \Chow_{4,1}$. Let $W^{\rm ss}(L_\infty):=p_2^{-1}(\Chow_{4,1}^{\rm ss})$. Since $p_2$ is birational and proper, we know that its restriction $W^{\rm ss}(L_\infty)\to \Chow_{4,1}^{\rm ss}$ is birational and proper. By taking the Stein factorization, it suffices to show that  $p_2$ has connected fibers over any point $[C]\in \Chow_{4,1}^{\rm ss}$. 

We first assume that $[C]$ is reduced. By \cite[Theorem 3.1]{CMJL12} there exists a unique quadric surface $Q= (q=0)$ and a non-zero section $s\in H^0(Q, \cO_Q(3))$ unique up to scaling such that $C=V(q,s)$. We claim that $p_2^{-1}([C])$ consists of only one point $([s],[C])$. Assume that $([s'], [C])\in p^{-1}([C])$ where $0\neq s'\in H^0(Q', \cO_{Q'}(3))$ and $Q'=V(q')$ is a quadric surface in $\bP^3$. From the construction we know that $C\subset V(q',s')\subset Q'$ which implies that $Q'=Q$ is the unique quadric surface containing $C$. If $V(q',s')$ is not a complete intersection, then under  suitable coordinates we may write
\[
q' = x_0x_1, \quad s = x_0 \cdot f_2(x_0, x_2, x_3),
\]
where $f_2$ is a non-zero homogeneous polynomial of degree $2$. Thus $V(q',s') \cap V(x_1)$ is a conic curve $V(x_1, f_2)$ union the line $V(x_0,x_1)$. On the other hand, by \cite[Theorem 3.1(2)iii)]{CMJL12} we know that $C\cap V(x_1)$ is a reduced cubic curve not containing the line $V(x_0,x_1)$. Thus $C\cap V(x_1)$ is not contained in $V(q',s') \cap V(x_1)$ which contradicts $C\subset V(q',s')$.
Thus $V(q',s')$ is a complete intersection with the same degree as $C$ which implies  $[s']=[s]\in \bP(E_{3})$. 

Finally, we study the case that $[C]$ is not reduced. Suppose $p_2([s], [C]) = [C]$ where $s\in H^0(Q, \cO_Q(3))$ is a non-zero section with $Q = V(q)$ a quadric surface. Since $\Supp([C])$ is a twisted cubic curve contained in $Q$, we know that $\rank (Q) \geq 3$ which implies that $q$ is irreducible. Hence $C=V(q,s)$ is a complete intersection curve with $\phi(C) = [C]$. By \cite[Theorem 3.1(0)]{CMJL12} we know that $C$ is a canonical genus $4$ ribbon.
Since all canonical genus $4$ ribbons are projectively equivalent by \cite{BE95}, we know that $p_2^{-1}([C])$ is a single orbit for the group $\Aut([C_{\red}]) \cong \PGL(2)$ which implies its connectedness.

The last statement follows from the fact that $\sX^{\rm c} = [\Chow_{4,1}^{\rm ss}/\PGL(4)]$ and a good moduli space of a normal stack is normal (see Theorem \ref{thm:gms-normal}).
\end{proof}

\begin{thm}[\cite{CMJL12, CMJL14}]\label{thm:cmjl-vgit}
The Chow quotient $\fX^{\rm c}$ is isomorphic to $\oM_4(\frac{5}{9})$. Moreover, for every $\alpha \in (\frac{8}{17}, \frac{5}{9})\cap \bQ$ and $t(\alpha):=\frac{34\alpha - 16}{33\alpha -14}$, let $\ocM_4(\alpha):=\sX_{t(\alpha)}$. Then $\ocM_4(\alpha)$ admits a projective good moduli space isomorphic to $\oM_4(\alpha)$. 
\end{thm}

\begin{rem}\label{rmk:cmjl-cycles}
    In Theorem \ref{thm:cmjl-vgit}, we do not define $\ocM_{4}(\frac{5}{9})$ to be $\sX^{\rm c}$. The key reason is that $\sX^{\rm c}$ does not parameterize curves but rather cycles. Later, in Sections \ref{sec:hyp-flip} and \ref{sec:newmodelsforHK}, we define  $\ocM_4(\frac{5} {9})$ as a new stack $\sX$ which admits a good moduli space isomorphic to $\fX^{\rm c}$.  We will see in Proposition \ref{prop:59-ewall} that $\ocM_4(\frac{5}{9} - \epsilon)$ admits an open immersion to the new stack $\sX$. 
\end{rem}

In Section \ref{sec:newmodelsforHK}, we will describe the wall crossings for $\oM_4(\alpha)$ as $\alpha$ ranges in $[\frac{5}{9}, \frac{2}{3})$.  For completeness, we describe below the last wall crossing from \cite{CMJL14} as $\alpha$ increases from $\frac{5}{9} - \epsilon$ to $\frac{5}{9}$.

\begin{theorem}[{\cite[Theorem 6.3]{CMJL14}}]\label{thm:lastVGITstack}
    All curves parameterized by the final VGIT quotient stack $ \ocM_4(\frac{5}{9}, \frac{23}{44}) := \sX_{t \in(\frac{6}{11},\frac{2}{3})}$ are reduced $(2,3)$ complete intersection curves in $\bP^3$.  The VGIT (semi-)stable points $[(C)] \in \ocM_4(\frac{5}{9}, \frac{23}{44})$ are precisely the Chow (semi-)stable (see Theorem \ref{thm:CMJL-Chow-GIT}) that are not ribbons, elliptic triboroughs, or curves on a quadric cone with a tacnode (an $A_3$ singularity) at the vertex of the cone.  
\end{theorem}

The next result is a summary of \cite[Theorem 3.1, Theorem 6.3]{CMJL14}.

\begin{theorem}[\cite{CMJL14}]\label{thm:lastVGITwall}
    The morphism $\oM_4(\frac{5}{9} - \epsilon) \to \oM_4(\frac{5}{9}) \cong \fX^{\rm c}$ is an isomorphism above $\fX^{\rm c} - \{[C_D] \cup \Gamma\}$.  The preimage of $[C_D]$ parameterizes unions of three conics meeting in two $D_4$ singularities.  The preimage of a generic point of $\Gamma$ consists of reduced curves $C$ on normal quadrics satisfying all of the following conditions: 
    \begin{itemize}
        \item $C$ has at worst a cusp at a singular point of the quadric (including the case of not passing through a singular point); 
        \item $C$ has an $A_5$ singularity at a smooth point of the quadric; 
        \item any component of $C$ passing through the $A_5$ singularity has degree at least $2$.\footnote{For genus reasons, a curve on $\bP(1,1,2)$ satisfying these conditions cannot contain a line.  Curves on the smooth quadric containing a line through the $A_5$ singular point are precisely those with non-separating $A_5$ singularities and specialize to the point $[C_{2A_5}]$ in $\sX^{\mathrm{c}}$.}
    \end{itemize}
    The preimage of the point in $\Gamma$ corresponding to the ribbon parameterizes reduced curves $C$ on normal quadrics satisfying all of the following conditions: 
    \begin{itemize}
        \item $C$ has at worst a cusp at a singular point of the quadric (including the case of not passing through a singular point); 
        \item $C$ has an $A_6, A_7, A_8, $ or $A_9$ singularity at a smooth point of the quadric; 
        \item if $C$ lies on a smooth quadric, then $C$ does not contain an irreducible component of degree $\le 2$; and 
        \item if $C$ lies on a quadric cone, then $C$ does not contain an irreducible component of degree $1$. 
    \end{itemize}
\end{theorem}

\subsection{Moduli of pairs and log Calabi--Yau wall crossing}\label{sec:bpcy}

In this subsection, we adapt the the moduli theory of boundary polarized Calabi--Yau pairs developed in \cite{ABB+, BL24} to study the moduli of genus four curves. We also introduce the relevant K-moduli and KSBA moduli spaces and setup the log Calabi--Yau wall crossing.

\begin{defn}
A \emph{boundary polarized CY pair} is a projective slc pair $(X,cD)$ such that 
\begin{enumerate}
    \item $D$ is an effective $\bQ$-Cartier Weil divisor on $X$, and $c\in \bQ_{>0}$;
    \item $K_X+cD\sim_{\bQ} 0$. 
\end{enumerate}

\end{defn}

\begin{defn}
Let $N\in \bZ_{>0}$ and $c\in \bQ_{>0}$ satisfying $Nc\in \bZ$. 
A \emph{family of boundary polarized CY pairs} $(X,cD) \to B$ over a Noetherian scheme $B$ with coefficients $c$ and index dividing $N$ consists of 
a flat projective morphism $X\to B$ and a relative K-flat Mumford divisor $D$,
satisfying the following conditions:
\begin{enumerate}
    \item $(X_b,cD_b)$ is a boundary polarized CY pair for each $b\in B$;
    \item $\omega_{X/B}^{[N]}(NcD) \cong_B \cO_X$;
    \item $\omega_{X/B}^{[m]}(nD)$ commutes with base change for all $m,n\in \bZ$.
\end{enumerate}
\end{defn}

\begin{thm}[{\cite[Theorem 3.7]{ABB+}}]
Let $N\in \bZ_{>0}$ and $c\in \bQ_{>0}$ satisfying $Nc\in \bZ$. 
Let $\chi:\bN\to \bZ$ be a function. 
There exists an algebraic stack $\cM(\chi, N, c)$ locally of finite type with affine diagonal that represents the moduli functor that send a scheme $B$ over $\bk$ to 
\[
\left\{\begin{array}{l}\textrm{families of boundary polarized CY pairs }(X,cD)\to B \textrm{ with coefficient $c$ and }\\ \textrm{index dividing $N$ satisfying }\chi(X_b, \omega_{X_b}^{[-m]}) = \chi(m) \textrm{ for any $b\in B$ and $m\in \bN$.}
\end{array}\right\}
\]
\end{thm}

We are interested in the case of compactifications of moduli spaces of $(\bP^1\times\bP^1, \frac{2}{3}D)$ where $D\in |\cO_{\bP^1\times\bP^1}(3,3)|$. 

\begin{defn}\label{def:CY4stack}
    Let $\chi(m) = \chi(\bP^1\times\bP^1, \omega_{\bP^1\times\bP^1}) = (2m+1)^2$.
    Let $\cM_4^{\circ}$ denote the open substack of $\cM(\chi, 3,\frac{2}{3})$ whose $\bk$-points are pairs $(\bP^1\times\bP^1, \frac{2}{3}D)$ where $D$ is a smooth curve of bidegree $(3,3)$.
    Let $\ocM_4^{\CY}$ denote the stack theoretic closure of $\cM_4^{\circ}$ in $\cM(\chi, 3,\frac{2}{3})$. We call $\ocM_4^{\CY}$ the \emph{boundary polarized Calabi--Yau moduli stack} for genus four curves.
\end{defn}

\begin{thm}[{\cite[Theorem 16.11]{ABB+} and \cite[Theorem 7.5]{BL24}}]\label{thm:CY-gms}
    The algebraic stack $\ocM_4^{\CY}$ is irreducible, of finite type with affine diagonal. Moreover, $\ocM_4^{\CY}$ is S-complete, $\Theta$-reductive,  satisfies the existence part of valuative criterion for properness, and admits a good moduli space $\oM_4^{\CY}$ as an irreducible projective scheme.
\end{thm}

We will see in Theorem \ref{thm:bpCY-is-smooth} that $\ocM_4^{\CY}$ is smooth, and that its good moduli space $\oM_4^{\CY}$ is a normal projective variety isomorphic to the Baily--Borel compactification of Kond\={o} \cite{Kondog4} of non-hyperelliptic genus four curves.

\begin{defn}\label{def:KAndKSBA}
We define the \emph{K-moduli stack} $\ocM_4^{\K}$ and the \emph{KSBA moduli stack} $\ocM_4^{\KSBA}$ for genus four curves as follows.
\begin{enumerate}
    \item Let $\ocM_4^{\K}$ be the substack of $\ocM_4^{\CY}$ consisting of families $(X,\frac{2}{3}D)\to B$ such that for each $b\in B$, $(X_b, (\frac{2}{3}-\epsilon)D_b)$ is a K-semistable log Fano pair for $0<\epsilon\ll 1$. 
    \item Let $\ocM_4^{\KSBA}$ be the substack of $\ocM_4^{\CY}$ consisting of families $(X,\frac{2}{3}D)\to B$ such that for each $b\in B$, $(X_b, (\frac{2}{3}+\epsilon)D_b)$ is slc for $0<\epsilon\ll 1$. 
\end{enumerate}

\end{defn}

Before stating the log Calabi--Yau wall crossing theorem, we recall some deformation theory results from \cite{ABB+}.

\begin{prop}\label{prop:deformations}
If $[(X_0, \frac{2}{3}D_0)]$ in $\ocM_4^{\CY}$ and the normalization $(\overline{X}_0, \overline{G}_0)$ of $(X_0, 0)$ is plt, then $\ocM_4^{\CY}$ is smooth at $[(X_0, \frac{2}{3}D_0)]$. 
\end{prop}

\begin{proof}
This follows from the proof of \cite[Proposition 11.5]{ABB+}, using \cite[Proposition 3.1]{HP10} and \cite[Lemma 6]{ACC+} (in addition to \cite[Theorems 8.2, 9.1]{Hac04} as used in loc. cit.) to obtain the canonical covering family.
\end{proof}

\begin{thm}[\cite{ABB+, BL24}]\label{thm:bpCYwallcrossing}
    Both $\ocM_4^{\K}$ and $\ocM_4^{\KSBA}$ are open substacks of $\ocM_4^{\CY}$. Furthermore, the following hold.
    \begin{enumerate}
        \item $\ocM_4^{\K}$ is an algebraic stack of finite type with affine diagonal. Moreover, it admits a good moduli space $\oM_4^{\K}$ as a normal, irreducible projective scheme.
        \item $\ocM_4^{\KSBA}$ is a separated Deligne--Mumford stack of finite type. Moreover, it admits a coarse moduli space $\oM_4^{\KSBA}$ as an irreducible projective scheme.
        \item There is a commutative wall crossing diagram
        \[
\begin{tikzcd}
\ocM_4^{\K} \arrow[d] \arrow[r,hook]  &
\ocM_{4}^{\CY} \arrow[d] & 
\ocM_{4}^{\KSBA} \arrow[d] \arrow[l,hook']  \\
\oM_4^{\K} \arrow[r]  & 
\oM_4^{\CY}  & 
\oM_4^{\KSBA} \arrow[l]
\end{tikzcd}
,\]
where the top arrows are open immersions, the vertical arrows are good moduli space morphisms, and the bottom arrows are projective birational morphisms. 
    \end{enumerate}
\end{thm}

\begin{proof}

This follows from \cite[Theorems 1.1 and 7.5]{BL24} (see also \cite[Theorem 16.11]{ABB+}). The normality of $\oM_4^{\K}$ follows from \cite[Theorem 2.21]{ADL20}.
\end{proof}

The locus of klt pairs $(X,\frac{2}{3}D)$ is contained in each of the three stacks.  

\begin{prop}\label{prop:klt-is-stable}
    If $[(X,\frac{2}{3}D)] \in \ocM_4^{\CY}$ is a klt pair, then $[(X,\frac{2}{3}D)] \in  \ocM_4^{\KSBA}$ and $[(X,\frac{2}{3}D)] \in \ocM_4^{\K}$.  In fact, $(X,(\frac{2}{3}- \epsilon)D)$ is K-stable for all $0<\epsilon \ll 1$. 
\end{prop}

\begin{proof}
    The KSBA stability is straightforward: because $(X,\frac{2}{3}D)$ is klt, $(X, (\frac{2}{3} + \epsilon)D)$ is lc for $0<\epsilon \ll 1$ and hence $[(X,\frac{2}{3}D)] \in \ocM_4^{\KSBA}$. The K-semistability follows from \cite[Lemma 5.3]{Zho23} because there are no lc places for the klt pair, and by \cite[Proposition 6.13, Theorem 6.15]{ABB+}, the pairs are in fact K-stable.  \end{proof}

\subsection{Curves on quadric cones}

In this subsection, we collect some useful results on curves of degree $6$ in $\bP(1,1,2)$.

\begin{prop}\label{prop:A4-cone}
Let $C$ be a curve on $\bP(1,1,2)$ of degree $6$ with an $A_4$-singularity at the cone vertex $[0,0,1]$. Then $C$ admits a $\bG_m$-equivariant degeneration to the strictly GIT polystable hyperelliptic curve $(x_2^2 = x_0^5 x_1^5)\subset \bP(1,1,5)_{[x_0, x_1, x_2]}$ (see Definition \ref{def:hyp-GIT}). 
\end{prop}

\begin{proof}
Let $[x,y,z]$ be the projective coordinates of $\bP(1,1,2)$. 
Denote by $C = (f(x,y) = 0)$ in the affine coordinate chart $z=1$. 
Since $C$ has an $A_4$-singularity at the cone vertex, we know that the covering curve $\tC = (f(x,y)=0)\subset \bk^2_{(x,y)}$ has an $A_7$-singularity at the origin. Thus after change of coordinates we may assume that the tangent line at the origin is $(y=0)$, and there exists analytic coordinates $(x, y')$ such that 
\begin{equation}\label{eq:A4-1}
f(x,y) = y'^2 + a x^8 + \textrm{higher order terms},
\end{equation}
where $(x,y')$ has weight $(1,4)$ and $a \neq 0$. Moreover, since the tangent line is $(y=0)$, we may choose $y'$ to be a $4$-jet of the form $y'= y - u x^3$ since it is semi-invariant under the $\bmu_2$-action $(x,y)\mapsto (-x,-y)$. If $u=0$ then we must have $a = 0$ which is contradiction to the assumption of an $A_4$ singularity. Thus we may assume $u = 1$ after a rescaling of the coordinates. Then, we may write 
\begin{equation}\label{eq:A4-2}
f(x,y) = y^2 - 2x^3 y + g_6(x,y) = y'^2 + (g_6(x, y'+x^3) - x^6),
\end{equation}
where $g_6(x,y)= \sum_{i =0}^6 b_i x^i y^{6-i}$ is a degree $6$ homogeneous polynomial in $(x,y)$. Combining \eqref{eq:A4-1} and \eqref{eq:A4-2}, we get $b_6 = 1$ and $b_5=a\neq 0$. Again after a rescaling of $(x,y)$ with weight $(1,3)$, we may assume that $b_5 = -1$. Thus we can write
\[
f(x, y) = y^2 - 2x^3 y + x^6 - x^5y + y^2 h_4(x,y),
\]
where $h_4(x,y) = \sum_{i=0}^4 b_i x^i y^{4-i}$ is a degree $4$ homogeneous polynomial in $(x,y)$.

Next we construct the test configuration degenerating $C$ to the given hyperelliptic curve. Let 
\begin{align*}
\cX & = V(sz_3 - z_1^3 + z_0 z_2) \subset \bP(1,1,2,3)_{[z_0, z_1, z_2, z_3]} \times \bA_s^1,\\
\cC & = V(z_3^2 - z_0^2 z_1^2 z_2 - s z_0 z_1^2 z_3+ \sum_{i=0}^4 b_i s^{10-2i} z_0^{6-i} z_1^{i})|_{\cX}
\end{align*}
We consider the $\bG_m$-action on $\bP(1,1,2,3)\times \bA^1$ by 
\[
t\cdot ([z_0, z_1, z_2, z_3], s) = ([z_0, t^2 z_1, t^6 z_2, t^5 z_3], t s).
\]
It is straightforward to check that $(\cX, \cC)$ is $\bG_m$-invariant which makes it a test configuration.

Finally, we analyze the fibers of $(\cX,\cC)\to \bA^1$. When $s=1$, we can identify $\cX_1$ with $\bP(1,1,2)_{[x,y,z]}$ under the closed immersion $\iota:\bP(1,1,2)\hookrightarrow \bP(1,1,2,3)$ where $\iota([x,y,z]) = [y, x, z, x^3 - yz]$. Then the equation of $\cC_1$ in $\bP(1,1,2)$ becomes
\begin{align*}
0 & = (x^3 - yz)^2 - x^2 y^2 z - x^2 y (x^3-yz) + \sum_{i=0}^4 b_i x^i y^{6-i} \\ & = y^2z^2 - 2x^3 y z + x^6 - x^5y + y^2 h_4(x,y).
\end{align*}
Thus in the affine coordinate $z=1$ we get precisely the equation $f(x,y)=0$, which implies that $(\cX_1, \cC_1)\cong (\bP(1,1,2), C)$. 

When $s=0$, we can identify $\cX_0$ with $\bP(1,2,9)_{[y_0, y_1, y_2]}$ under the closed immersion $\iota':\bP(1,2,9)\hookrightarrow \bP(1,1,2,3)$ where $\iota'([y_0,y_1,y_2]) = [y_0^3, y_0y_1, y_1^3, y_2]$. Then the equation of $\cC_0$ in $\bP(1,2,9)$ becomes $0 = y_2^2 - y_0^8 y_1^5$. We claim that $\cC_0$ is isomorphic to the strictly GIT polystable hyperelliptic curve. Indeed, we can take another closed immersion $\iota'':\bP(1,2,9)\hookrightarrow \bP(1,1,5,9)_{[x_0,x_1,x_2,x_3]}$ where $\iota''([y_0,y_1,y_2]) = [y_0^2, y_1, y_0y_2, y_2^2]$. Then $(\cX_0,\cC_0)$ is identified with $(V(x_2^2 - x_0 x_3), V(x_2^2 - x_0 x_3, x_3 - x_0^4 x_1^5))$ inside $\bP(1,1,5,9)$. Thus $\cC_0$ is isomorphic to $V(x_2^2 - x_0^5 x_1^5)$ inside $\bP(1,1,5)_{[x_0,x_1,x_2]}$ after substituting $x_3$ by $x_0^4 x_1^5$. The proof is finished.
\end{proof}

\begin{prop}\label{prop:A3-cone}
Let $C$ be a curve on $\bP(1,1,2)_{[x,y,z]}$ of degree $6$ with an $A_3$-singularity at the cone vertex $[0,0,1]$. Then after change of coordinates we may write 
\begin{equation}\label{eq:A3-1}
C = (y^2 z^2 + B x^3 y z + A x^6 + y^2 h_4(x,y) = 0)
\end{equation}
where $4A\neq B^2$. Moreover, $C$ admits a $\bG_m$-equivariant degeneration to $C_{A,B}$. If in addition $h_4\neq 0$, then $C$ has at worst $A_4$-singularities in the smooth locus of $\bP(1,1,2)$. 
\end{prop}

\begin{proof}
Let $[x,y,z]$ be the projective coordinates of $\bP(1,1,2)$. 
Denote by $C = (f(x,y) = 0)$ in the affine coordinate chart $z=1$. 
Since $C$ has an $A_3$-singularity at the cone vertex, we know that the covering curve $\tC = (f(x,y)=0)\subset \bA^2_{(x,y)}$ has an $A_5$-singularity at the origin. Thus after change of coordinates we may assume that the tangent line at the origin is $(y=0)$, and there exists   an analytic coordinate $(x, y')$ such that 
\[
f(x,y) = y'^2 + a x^6 + \textrm{higher order terms},
\]
where $(x,y'= y- ux^3)$ has weight $(1,3)$ and $a\neq 0$. Thus we may write 
\begin{equation}\label{eq:A3-2}
f(x,y) = y^2 + B x^3y  + Ax^6 + \textrm{higher order terms},
\end{equation}
where $(x,y)$ has weight $(1,3)$ and $(A,B) = (a + u^2, -2u)$ satisfies $4A - B^2 = 4a \neq 0$.

Next we claim that there exists a unique unipotent automorphism $\sigma:[x,y,z]\mapsto [x+ vy, y , z+q(x,y)]$ that transforms $C$ into the standard form \eqref{eq:A3-1} where $v\in \bk$ and $q(x,y)$ is a degree $2$ homogeneous polynomial. According to \eqref{eq:A3-2} we have $C=(F(x,y,z)=0)$ where 
\[
F(x,y,z) = y^2 z^2 + g_4(x,y) z + g_6(x,y),
\]
where $g_4(x,y) = B x^3 y + \sum_{i=0}^2 a_i x^i y^{4-i}$ and $g_6(x,y) = A x^6+\sum_{i=0}^5 b_i x^i y^{6-i}$. Thus $\sigma^* C=(F\circ\sigma(x,y,z)) =0)$ where 
\begin{align*}
F\circ\sigma(x,y,z) & = y^2(z+q)^2 + g_4(x+vy, y) (z+q) + g_6(x+vy, y)\\
& = y^2 z^2 + (2y^2 q + g_4(x+vy, y)) z + (y^2 q^2 + g_4(x+vy, y) q + g_6(x+vy, y))\\
& =: y^2 z^2 + g_{4}^\sigma (x, y) z + g_6^{\sigma}(x,y).
\end{align*}
For every $v\in \bk$ there exists a unique $q = \frac{Bx^3 y - g_4(x+vy, y)}{2y^2}$ such that $g_4^\sigma = Bx^3 y$. Then we have
\begin{align*}
    g_6^\sigma(x,y) & = g_6(x+vy,y) +  B x^3y q - y^2 q^2 \\
    & = Ax^6 + (b_5 +6A v + B\cdot \mathrm{coeff}_{x^2}(q) ) x^5 y +  \textrm{lower order terms}
\end{align*}
From the above description we have $\mathrm{coeff}_{x^2}(q) = -\frac{1}{2}(a_2+3Bv)$. Thus we have 
\[
\mathrm{coeff}_{x^5 y}(g_6^\sigma) = b_5 +6A v -\frac{B}{2}(a_2+3Bv) = b_5 - \frac{B a_2}{2}  + \frac{3}{2}(4A-B^2) v.
\]
Since $4A-B^2\neq 0$,  there exists a unique $v\in \bk$ such that $\mathrm{coeff}_{x^5 y}(g_6^\sigma)  = 0$. As a result, we have 
\[
F\circ \sigma (x,y,z) = y^2 z^2 + Bx^3yz + Ax^6 +  y^2 h_4(x,y),
\]
where $h_4(x,y) := \frac{g_6^\sigma - Ax^6}{y^2}$,
and the claim is proved.

Next we show the statement on degenerations. Suppose $C$ has the standard form \eqref{eq:A3-1}. Let $\rho: \bG_m\to \Aut(\bP(1,1,2))$ be a $1$-PS such that $\rho(t)\cdot [x,y,z ] = [x, t^{-1} y, t z]$. Then clearly $\rho(t)_* C = (y^2z^2 + Bx^3yz + Ax^6 + t^2 y^2 h_4(x, ty) = 0)$ which implies that $\lim_{t\to 0} \rho(t)_* C = C_{A,B}$. 

Finally we address the singularities of $C$ in the smooth locus of $\bP(1,1,2)$. Assume that $h_4\neq 0$. We first look at the affine chart $y = 1$, and consider the new coordinates $(x, z' = z+\frac{B}{2} x)$. Then $C$ is defined by the equation
\[
z'^2 + \frac{4A-B^2}{4} x^6 + \sum_{i=0}^4 b_i x^i = 0.
\]
Since $4A-B^2\neq 0$ and $h_4\neq 0$, we know that the polynomial $\frac{4A-B^2}{4} x^6 + \sum_{i=0}^4 b_i x^i$ has no root of multiplicity $6$. Thus $C$ has at worst $A_4$-singularities in the affine chart $y=1$. Next, we look at the affine chart $x=1$, where $C$ is defined by the equation
\[
y^2 z^2 + By z + A + \sum_{i=0}^4 b_i y^{6-i} = 0.
\]
We only need to look at the singularities on the line $(y=0)$, which is non-empty if and only if $A= 0$. By analyzing the partial derivatives in $y$ and $z$, we conclude that $C$ has at worst an $A_1$-singularity at $[1,0,0]$, and is smooth elsewhere on the line $y=0$ and $x\neq 0$. Thus the proof is finished.
\end{proof}

\section{Hyperelliptic curves and ribbons}\label{sec:hyp-and-rib}

In this section, we revisit the moduli theory of hyperelliptic curves and ribbons of genus $g\geq 2$ after \cite{AV04, BE95}. We define moduli stacks $\sH_g$ of hyperelliptic curves and $\sR_g$ of ribbons (Definition \ref{def:hyp-rib-stacks}) and give explicit quotient stack presentations (Theorems \ref{thm:hyp-stack} and \ref{thm:rib-stack}). We also show that the deformation of the hyperelliptic ribbon (the unique curve in $\sH_g\cap \sR_g$) is unobstructed (Theorem \ref{thm:hyp-ribbon-unobstructed}).

\subsection{Hyperelliptic curves}

\begin{defn}\label{def:hyperelliptic}
A \emph{hyperelliptic curve} of (arithmetic) genus $g$ is a proper curve $C$ over $\bk$ satisfying the following conditions.
\begin{enumerate}
    \item $H^0(C, \cO_C) \cong \bk$ and $H^1(C, \cO_C) \cong \bk^{g}$;
    \item there exists a finite morphism $\varphi: C \to \bP^1$ such that $\varphi_* \cO_C$ is a sheaf of rank two. 
\end{enumerate}
\end{defn}

\begin{defn}
A \emph{Weierstrass curve} of (arithmetic) genus $g$ is a Cartier divisor 
 $C$ of $\bP(1,1,g+1)_{[x,y,z]}$ defined by the equation $C=(z^2 = f(x,y))$ where $f$ is a (possibly zero) binary form of degree $2g+2$.
\end{defn}

\begin{thm}\label{thm:hyp=Weierstrass}
Every hyperelliptic curve is isomorphic to a Weierstrass curve. Conversely, every Weierstrass curve is hyperelliptic. In particular, if $C$ is a hyperelliptic curve, then $\varphi_* \cO_C$ is a locally free sheaf of rank $2$ on $\bP^1$.
\end{thm}

\begin{proof}
We start by showing the first statement.
Let $C$ be a hyperelliptic curve of genus $g$. Since $\varphi$ is surjective,  $\varphi^{\#}: \cO_{\bP^1} \to \varphi_* \cO_C$ is non-zero at the generic point of $\bP^1$. 
Then $H^0(\bP^1, \varphi_* \cO_C) \cong H^0(C, \cO_C) \cong \bk$ implies that $\varphi_* \cO_C$ is torsion free as an $\cO_{\bP^1}$-module. Hence $\varphi_* \cO_C$ is locally free of rank $2$. By the Grothendieck splitting theorem, we have $\varphi_* \cO_C\cong \cO_{\bP^1}(a) \oplus \cO_{\bP^1}(b)$ for some integers $a\geq b$. Then $h^0(\bP^1, \varphi_* \cO_C) = 1$ and $h^1(\bP^1, \varphi_* \cO_C) = g$ implies that $a =0$ and $b = -g - 1$. Denote by $L:=\cO_{\bP^1}(-g-1)$ so we have $\varphi_* \cO_C \cong \cO_{\bP^1}\oplus L$. Moreover, we may choose the isomorphism so that $\varphi^{\#}(s) = (s,0)$. Then we have $C \cong \Spec_{\bP^1} \varphi_* \cO_C$.

Next, we analyze the $\cO_{\bP^1}$-algebra structure on $\cO_{\bP^1} \oplus L$. It is clear that the algebra structure is determined by a $\cO_{\bP^1}$-module homomorphism $L^{\otimes 2}\to \cO_{\bP^1} \oplus L$, which then corresponds to a section $(h_2,h_1)\in H^0(\bP^1, L^{\otimes -2} \oplus L^{\otimes -1})$. We claim that $C$ is isomorphic to the Cartier divisor $C'$ on $\bP(1,1,g+1)$ defined by $(z^2 - h_1(x,y)z - h_2(x,y) =0)$. Indeed, let $S:=\bP(1,1,g+1)\setminus \{[0,0,1]\}$. Then we have the projection map $\varphi_S: S\to \bP^1$ defined by $\varphi_S([x,y,z]) = [x,y]$. Moreover, $\varphi_S$ realizes $ S$ as the total space of the line bundle $L^\vee$ over $\bP^1$, so we have
\[
S \cong \Spec_{\bP^1} \oplus_{m=0}^\infty L^{\otimes m}.
\]
Since $(z=0)$ corresponds to the zero section of $\varphi_S$, we have that $z$ corresponds to a nowhere vanishing section $z_S\in H^0(\bP^1, L\otimes \cO_{\bP^1}(g+1))$, such that 
\[
C' \cong \Spec_{\bP^1} \oplus_{m=0}^{\infty} L^{\otimes m}/(z_S^2 -h_1(x,y) z_S - h_2(x,y)).
\]
Since $z_S^m$ generates $L^{\otimes m}\otimes \cO_{\bP^1}(m(g+1))$ for any $m\in \bN$, we know that 
\[
C' \cong  \Spec_{\bP^1} \cO_{\bP^1}\oplus L,
\]
where the algebra structure is determined by 
\[
L^{\otimes 2}\to \cO_{\bP^1}\oplus L, \quad s\mapsto (s\otimes h_2, s\otimes h_1) \quad\textrm{for any local section $s$ of $L$}.
\]
Thus we have $C\cong C'$. After a change of coordinate $[x,y,z]\mapsto [x,y,z-\frac{h_1}{2}]$, we see that $C'$ is isomorphic to the curve $C'':= (z^2 = f(x,y))$ in $\bP(1,1,g+1)$ where $f = \frac{h_1^2}{4} +h_2$. Thus we have that $C\cong C''$, and $\varphi$ corresponds to $\varphi_S|_{C''}$ under this isomorphism. 

Conversely, we shall show that a Weierstrass curve is always hyperelliptic. If $C = (z^2 = f(x,y))$ is a Weierstrass curve in $\bP(1,1,g+1)$, then the above argument implies that $\varphi=\varphi_S|_{C}: C\to \bP^1$ is affine and $\varphi_* \cO_C \cong \cO_{\bP^1} \oplus L$, and hence $H^i(C, \cO_C) \cong H^0(\bP^1, \varphi_* \cO_C) $. This implies that $\varphi$ is finite and hence $C$ is hyperelliptic. The proof is finished.
\end{proof}

\begin{prop}\label{prop:hyp-canonical-map}
Let $C$ be a hyperelliptic curve of genus $g$. Then $C$ has local complete intersection singularities, $|\omega_C|$ is base point free and induces a map $f:C\to \bP^{g-1}$. Moreover, $f$ can be factored into $f = \iota\circ \varphi$ where $\iota:\bP^1\hookrightarrow \bP^{g-1}$ is a closed embedding whose image is a rational normal curve. In particular, $\varphi$ is unique up to automorphisms of $\bP^1$.
\end{prop}

\begin{proof}
By Theorem \ref{thm:hyp=Weierstrass}, we may assume that $C$ is a Weierstrass curve in $\bP(1,1,g+1)$. Since $C$ is a Cartier divisor in the smooth surface $S = \bP(1,1,g+1)\setminus \{[0,0,1]\}$, we know that $C$ has local complete intersection singularities. By adjunction, we have $\omega_C \cong \omega_S(C)|_C$. Moreover, 
\[
\omega_S(C)\cong \cO_{\bP(1,1,g+1)}(K_{\bP(1,1,g+1)} + C)|_S \cong \cO_{\bP(1,1,g+1)}(g-1)|_S\cong \varphi_S^* \cO_{\bP^1}(g-1). 
\]
Thus we have $\omega_C \cong \varphi^* \cO_{\bP^1}(g-1)$ which implies that $|\omega_C|$ is base point free. Moreover, we have 
\[
H^0(C, \omega_C) \cong H^0(\bP^1, \varphi_* \omega_C) \cong H^0(\bP^1, \cO_{\bP^1}(g-1)\oplus (L\otimes \cO_{\bP^1}(g-1))) \cong H^0(\bP^1, \cO_{\bP^1}(g-1)). 
\]
Here we used the fact that $L = \cO_{\bP^1}(-g-1)$ which implies $H^0(\bP^1, L\otimes \cO_{\bP^1}(g-1))=0$. Hence $f$ factors into $f = \iota\circ\varphi$ where $\iota:\bP^1\to \bP^{g-1}$ is induced by the complete linear system $|\cO_{\bP^1}(g-1)|$. The last statement on $\varphi$ follows easily from this as it is determined by the canonical map $f$ up to projective equivalence. 
\end{proof}

\begin{cor}\label{cor:equiv-hyp}
    Let $C, C'$ be hyperelliptic curves of genus $g$.  Then, $C \cong C'$ if and only if $C$ and $C'$ are equivalent as Weierstrass curves, i.e. by an automorphism of $\bP(1,1,g+1)$.  
\end{cor}

\begin{proof}
    The backwards implication is clear.  For the forward implication, suppose there is an isomorphism $\alpha: C \to C'$ and denote the map $C \to \bP^1$ by $\varphi$ and $C' \to \bP^1$ by $\varphi'$.  By Theorem~\ref{thm:hyp=Weierstrass}, we may assume each curve is a Weierstrass curve with $C$ given by $(z^2 = f(x,y))$ and $C'$ by $(z^2 = f'(x,y))$.
    If both $C$ and $C'$ are non-reduced, then we must have $f = f'=0$ and hence the Weierstrass curves are the same. Let us assume that both $C$ and $C'$ are reduced, i.e. both $f$ and $f'$ are non-zero.
    By Proposition~\ref{prop:hyp-canonical-map}, $\varphi'\circ \alpha$ is a hyperelliptic map $C \to \bP^1$, so $\varphi' \circ \alpha = \beta \circ \varphi$ where $\beta$ is an automorphism of $\bP^1$.  Therefore, the ramification divisors of $C$ and $C'$ must coincide via $\beta$, but the ramification divisors are precisely $(f(x,y) = 0)$ and $(f'(x,y) = 0) \subset \bP^1$.  Therefore, taking a suitable lifting of $\beta \in \Aut(\bP^1)$ as an element of $\Aut(\bP(1,1,g+1)$ acting on the first two coordinates, we see that $C$ and $C'$ differ by an automorphism of $\bP(1,1,g+1)$. 
\end{proof}

\begin{defn}\label{def:hyp-GIT}
A Weierstrass curve $C= (z^2 = f(x,y))\subset \bP(1,1,g+1)$ is called \emph{GIT (poly/semi)stable} if the binary form $f$ is non-zero and GIT (poly/semi)stable. A hyperelliptic curve is called  \emph{GIT (poly/semi)stable} if its isomorphic Weierstrass curve is GIT (poly/semi)stable. By Corollary \ref{cor:equiv-hyp}, this is independent of the choice of the isomorphic Weierstrass curve.  We denote by $\frak{H}_g^+$ the GIT moduli space of hyperelliptic curves of genus $g$. 
\end{defn}

\subsection{Ribbons}\label{sec:ribbons}

We briefly recall the concept of ribbons from \cite{BE95}. 

\begin{defn}\label{def:ribbon}
    A \emph{(rational) ribbon}  is a scheme $C$ of finite type over $\bk$ equipped with an isomorphism $\bP^1_{\bk} \xrightarrow{\cong} C_{\red}$, such that the nilradical sheaf $\cN\subset \cO_C$ satisfies $\cN^2 = 0$, and that $\cN$ is a locally free $\cO_{C_{\red}}$-module of rank $1$. We often denote by $L$ the sheaf $\cN$ considered as an  $\cO_{\bP^1}$-module, and call it the \emph{conormal bundle} of $\bP^1$ in $C$.

    We say that a ribbon $C$ has \emph{(arithmetic) genus} $g$ if $\chi(C, \cO_C) = 1-g$. By Riemann--Roch, this is equivalent to saying that $L \cong \cO_{\bP^1}(-g-1)$. In particular, we have $H^0(C, \cO_C) \cong \bk$ and $H^1(C, \cO_C) \cong \bk^g$.

    A ribbon $C$ is called \emph{split} if the morphism $C_{\red}\to C$ admits a retraction $C \to C_{\red}$. By \cite[Proposition 1]{BE95} there exists a unique split ribbon of each genus. 
\end{defn}

\begin{expl}
Let $R_{\hyp}:=(z^2=0) \subset \bP(1,1,g+1)$ be a specific hyperelliptic curve. Then there is a finite morphism $\varphi: R_{\hyp}\to \bP^1$ given by $\varphi([x,y,z]) =[x,y]$. Moreover, $R_{\hyp} \cong \Spec_{\bP^1} \cO_{\bP^1} \oplus L$ with $L = \cO_{\bP^1}(-g-1)$, where the $\cO_{\bP^1}$-algebra structure satisfies $L^2 = 0$. Thus $R_{\hyp}$ is a split ribbon of genus $g$. We call $R_{\hyp}$ the \emph{hyperelliptic ribbon of genus $g$}. Since split ribbons are unique, we know that every split ribbon is isomorphic to the hyperelliptic ribbon.
\end{expl}

For any ribbon $C$ with conormal bundle of $\bP^1$ denoted $L$ as above, there is a short exact sequence 
\[ 0 \to L \to \Omega_C \vert_{\bP^1} \to \Omega_{\bP^1} \to 0 \] from restriction of $\Omega_C$ to $\bP^1$.  By the standard interpretation of Ext groups (see, for example \cite[Exercise III.6.1]{Har77}), $\Omega_C \vert_{\bP^1}$ corresponds to an element of $\Ext^1(\Omega_{\bP^1}, L)$, denoted by $e_C$.  This extension class classifies the ribbon as given by the following theorem. 

\begin{thm}[{\cite[Theorem 1.2]{BE95}}]\label{thm:ribbon-ext}
Given any class $e\in \Ext^1(\Omega_{\bP^1}, L)$, there is a unique ribbon $C$ with $e = e_C$. Moreover, two ribbons $C$ and $C'$ are isomorphic over $\bP^1$ if and only if $e_C= a e_{C'}$ for some $a\in \bk^{\times}$. In particular, a ribbon $C$ is split if and only if $e_C = 0$.
\end{thm}

\begin{prop}[{\cite[Theorem 5.3]{BE95}}]\label{prop:ribbon-canonical-map}
Let $C$ be a ribbon of genus $g$. Then $|\omega_C|$ is base point free and induces a map $f: C\to \bP^{g-1}$ whose reduced image is a rational normal curve. Moreover, $f$ is a closed embedding if and only if $C$ is not a hyperelliptic ribbon.
\end{prop}

\subsection{Stacks of hyperelliptic curves and ribbons}

\begin{defn}
Let $S$ be a scheme. A \emph{$\bP^1$-fibration} $P/S$ is a morphism $P\to S$ from a scheme $P$ that is a $\bP^1$-bundle in the \'etale topology. In other words, a $\bP^1$-fibration $P/S$ is precisely a Severi--Brauer $S$-scheme $P$ of relative dimension $1$.
\end{defn}

\begin{defn}\label{def:hyp-rib-stacks}
We define the functors $\sH_g$ and $\sR_g$ from locally Noetherian $\bk$-schemes to groupoids as 
\[
\sH_g(S) := \left\{(f: X\to S )\in \Curves_g^{\Gor}(S)\left|\begin{array}{l}
\textrm{there exists a $\bP^1$-fibration $\varpi:P\to S$ and a }\\
\textrm{finite morphism $\varphi: X\to P$ such that $\varphi_*\cO_{X}$}\\
\textrm{is locally free of rank $2$ and $f = \varpi\circ \varphi$.}
\end{array}\right.\right\}.
\]
\[
\sR_g(S):= \left\{(f: X\to S )\in \Curves_g^{\Gor}(S)\left|\begin{array}{l}
\textrm{there exists a $\bP^1$-fibration $\varpi:P\to S$ and a}\\
\textrm{closed immersion $\psi:P\hookrightarrow X$ with ideal sheaf}\\
\textrm{$\cN$ such that $\cN^2 = 0$, $\cN$ is a locally free}\\ \textrm{$\cO_P$-module of rank $1$, and $\varpi = f\circ \psi$.}
\end{array}\right.\right\}.
\]
An object in $\sH_g(S)$ (resp. $\sR_g(S)$) is called a \emph{family of hyperelliptic curves} (resp. a \emph{family of ribbons}) of genus $g$ over $S$. 
\end{defn}

\begin{rem}
\begin{enumerate}
    \item By Theorem \ref{thm:hyp=Weierstrass} (resp. Definition \ref{def:ribbon}), we know that $\sH_g(\bk)$ (resp. $\sR_g(\bk)$) is precisely the groupoid of all hyperelliptic curves (resp. all ribbons) of genus $g$. 
    \item Note that we assume the $\bk$-schemes $S$ in the above definition to be locally Noetherian, as $\Curves_g^{\Gor}$ is a locally Noetherian algebraic stack. One can define the functors over arbitrary $\bk$-schemes by bootstrapping from the locally Noethrian case, see e.g. \cite[Definition 3.4]{ABB+}.
    \item In the above definition, we only assume the existence of a $\bP^1$-fibration $P/S$, while no specific choice of $P/S$ is taken. Nevertheless, as we shall see in proofs of Theorems \ref{thm:hyp-stack} and \ref{thm:rib-stack}, the $\bP^1$-fibration $P/S$ is indeed unique up to unique isomorphism once exists. The advantage of our definition is that there are natural monomorphisms of stacks from $\sH_g$ and $\sR_g$ to $\Curves_g^{\Gor}$, which turn out to be locally closed immersions (see Proposition  \ref{prop:hyp-rib-closed-substack}). 
\end{enumerate}

\end{rem}

\begin{thm}\label{thm:hyp-stack}
The functor $\sH_g$ is represented by an algebraic stack of finite type over $\bk$. Moreover, we have an isomorphism of  stacks
\[
\sH_g \cong [\bA(H^0(\bP^1, \cO_{\bP^1}(2g+2)))/(\GL(2)/\bmu_{g+1})],
\]
where the $\GL(2)/\bmu_{g+1}$-action is induced by the standard isomorphism of $\GL(2)$ with the automorphism group $\Aut(\bP^1, \cO_{\bP^1}(1))$. 
\end{thm}

\begin{proof}

First of all, we show that $\sH_g$ is a prestack, i.e. a category fibered in groupoids. Suppose $(f:X\to S)\in \sH_g(S)$. Let $g: S'\to S$ be a morphism of locally Noetherian $\bk$-schemes. Then we have $f': X'= X\times_S S' \to S'$ as the base change of $f$ under $g$. Clearly, $f'$ is the pull-back of $f$ under $g$ in the stack $\Curves_g^{\Gor}$. Thus it suffices to show that $f'\in \sH_g(S')$. Let $\varpi': P'=P\times_S S'\to S'$ and $\varphi': X'\to P'$ be the base change of $\varpi$ and $\varphi$ under $g$ respectively. Then clearly $P'/S'$ is a $\bP^1$-fibration, and $\varphi'$ is a finite morphism such that $f'=\varpi'\circ\varphi'$. Denote by $g_P: P'\to P$ the base change of $g$ under $\varpi$. Since $\varphi$ is affine, by \cite[\href{https://stacks.math.columbia.edu/tag/02KG}{Tag 02KG}]{stacksproject} we have that $\varphi'_* \cO_{X'}  = g_P^*\varphi_* \cO_X$ is locally free of rank $2$. Thus $f'\in\sH_g(S')$. 

Next, we define a functor $\sH_g'$ from locally Noetherian $\bk$-schemes to groupoids as
\[
\sH_g'(S):=\left\{(P/S, \cL, i: \cL^{\otimes 2}\to \cO_{P})\left|\begin{array}{l}
P/S\textrm{ is a $\bP^1$-fibration, $\cL$ is a line bundle on $P$}\\
\textrm{whose restriction to every geometric fiber has}\\
\textrm{degree $-g-1$, and $i$ is a homomorphism of}\\
\textrm{$\cO_P$-modules.}
\end{array}\right.\right\}.
\]
Note that our definition of $\sH_g'$ is very similar to $\cH'(1,2, g+1)$ from \cite[Remark 3.3]{AV04}, except that we allow $i$ to be zero on fibers. By the same argument as \cite[Proof of Theorem 4.1]{AV04}, we have that $\sH_g'$ is a stack and
\begin{equation}\label{eq:hyp-quotient-stack}
\sH_g'\cong [\bA(H^0(\bP^1, \cO_{\bP^1}(2g+2)))/(\GL(2)/\bmu_{g+1})].
\end{equation}

Next, we show that $\sH_g$ and $\sH_g'$ are equivalent as fibered categories. Let $F:\sH_g'\to \sH_g$ be a functor that sends $(P/S, \cL, i)\in \sH_g'(S)$ to $(X\to S)\in \sH_g(S)$, where 
\[
X :=\Spec_{P} \cO_P \oplus \cL,
\]
and the algebra structure on  $\cO_P \oplus \cL$ is determined by $i:\cL^{\otimes 2}\to \cO_P$. It is not hard to see that $X\to S$ is a family of hyperelliptic curves. 

To show that $F$ is an equivalence of fibered categories, it suffices to show that $F$ is fully faithful and essentially surjective. 
We start from showing the essential surjectivity of $F$.
Let $(\pi:X\to S)\in \sH_g(S)$ be a family of hyperelliptic curves. 

 Let $\cF:=\ker(\cO_P\to \varphi_* \cO_{X})$ and $\cL:= \coker(\cO_P\to \varphi_* \cO_{X})$. For every point $s\in S$, by Theorem \ref{thm:hyp=Weierstrass} we know that there is an exact sequence 
\[
0 \to \cO_{P_s}\to \varphi_{s,*} \cO_{X_s}\to \cL_s\to 0
\]
where $\cL_s$ is a line bundle on $P_s$. Since $\varphi_{s,*} \cO_{X_s}$ is the pull-back of $\varphi_* \cO_{X}$ by \cite[\href{https://stacks.math.columbia.edu/tag/02KG}{Tag 02KG}]{stacksproject}, we know that $\cL_s$ is also the pull-back of $\cL$ by right exactness of pull-back. Thus we know that $\cL\otimes \kappa(p)$ has dimension $1$ at every point $p\in S$, which implies that $\cL$ is a line bundle on $P$. Thus we have a short exact sequence 
\[
0 \to \cO_P/\cF \to \varphi_* \cO_{X} \to \cL \to 0,
\]
where $\cF\subset \cO_P$ is some ideal sheaf.
Since both $\varphi_* \cO_{X}$ and  $\cL$ are locally free, they are both flat which implies that $\cO_P/\cF$ is also flat, and hence locally free. As a result, we have $\cF = 0$ and hence a short exact sequence
\begin{equation}\label{eq:hyp-SB}
    0\to \cO_P \to \varphi_* \cO_{X}\to \cL \to 0,
\end{equation}
where $\cL$ is an invertible sheaf on $P$ whose restriction to every geometric fiber has degree $-g-1$. 

Next, we show that there is a fiberwise $\bmu_2$-action on $X/S$ such that $X\to P$ is the $\bmu_2$-quotient map. Let $\sqcup_j S_j\to S$ be an \'etale affine covering such that $(P_j:=P\times_S S_j, \cL_j)$ is isomorphic to $(\bP^1_{S_j}, \cO_{\bP^1_{S_j}}(-g-1))$. Here $\pi_j:X_j\to S_j$, $\varpi_j: P_j\to S_j$, and $\cL_j$ denote the base change of $\pi$, $\varpi$, and $\cL$ under $S_j/S$. Then we get an exact sequence by the base change of \eqref{eq:hyp-SB}:
\begin{equation}\label{eq:hyp-SB-cover}
     0\to \cO_{P_j} \to \varphi_{j,*} \cO_{X_j}\to \cL_j \to 0.
\end{equation}
Since 
\[
\Ext^1(\cL_j, \cO_{P_j}) \cong H^1(P_j, \cL_j^{\vee}) \cong H^1( \bP^1_{S_j}, \cO_{\bP^1_{S_j}}(g+1))= 0,
\]
the exact sequence \eqref{eq:hyp-SB-cover} splits. Thus we have $\varphi_{j,*}\cO_{X_j} \cong \cO_{P_j}\oplus \cL_j$, and the algebra structure is determined by $\cL_j^{\otimes 2} \to \cO_{P_j}\oplus \cL_j$, which corresponds to $(h_{2}, h_{1})\in H^0(P_j, \cL_j^{\otimes -2}\oplus \cL_j^{\otimes -1})$. Similar to the proof of Theorem \ref{thm:hyp=Weierstrass}, under an automorphism of $\cO_{P_j}\oplus \cL_j$ given by $(t_1, t_2) \mapsto (t_1 - t_2\otimes \frac{h_1}{2}, t_2)$, we have a new isomorphism $\varphi_{j,*}\cO_{X_j} \cong \cO_{P_j}\oplus \cL_j$ where the algebra structure on  the latter is given by $\cL_j^{\otimes 2}\to \cO_{P_j}$ corresponding to $h_2 +\frac{h_1^2}{4}$. Thus we have an involution $\tau_j: X_j \to X_j$ where $\tau_j$ acts as $\diag(1,-1)$ on $\cO_{P_j}\oplus \cL_j$.

Next, we show that $\{\tau_j\}$ satisfies \'etale descent and hence gives a non-trivial involution $\tau: X \to X$ leaving $P$ fixed. Indeed, if $T\to S_j$ is a morphism, and $\tau_T: X_T\to X_T$ with $X_T:=X_j\times_{S_j} T$ is a non-trivial involution that leaves $P_T:=P_j\times_{S_j} T$ fixed, then we claim that $\tau_T$ is the pull-back of $\tau_j$. Let $\cF_+$ and $\cF_-$ be the $\tau_T$-eigensheaves of $\varphi_{T,*} \cO_{X_T}$ corresponding to eigenvalues $1$ and $-1$. Then we have  $\varphi_{T,*} \cO_{X_T}\cong \cF_+\oplus \cF_-$ which implies that both $\cF_+$ and $\cF_-$ are locally free of rank $1$ as they are non-zero. 
Moreover, since $\cO_{P_T}$ is contained in $\cF_+$, we have a surjection $\cL \twoheadrightarrow  \cF_-$ which implies that $\cL \cong \cF_-$ and hence $\cO_{P_T}\cong \cF_+$. Therefore, $\tau_T$ and the pull-back of $\tau_j$ have the same eigensheaves which implies that they are the same. This shows that $\{\tau_j\}$ satisfies \'etale descent and gives a construction of $\tau$. Thus \eqref{eq:hyp-SB} splits and we have $\varphi_* \cO_{X}\cong \cO_P \oplus \cL$ as the decomposition of $\tau$-eigensheaves corresponding to eigenvalues $1$ and $-1$. Therefore, the algebra structure on $\varphi_* \cO_{X}$ is determined by $i: \cL^{\otimes 2} \to \cO_P$. As a result, for each family $\pi:X\to S$ of hyperelliptic curves we found a triple $(P/S, \cL, i)\in \sH_g'(S)$ such that $(X\to S)$ is isomorphic to $F(P/S, \cL, i)$. Thus $F$ is essentially surjective. 

Finally, we show that $F$ is fully faithful. This is equivalent to showing that for any $(X\to S)\in \sH_g(S)$, the triple $(P/S, \cL, i)\in \sH_g'(S)$ whose image under $F$ is $(X\to S)$ is unique up to unique isomorphism. 
This can be reduced to showing that the factorization $X\xrightarrow{\varphi} P \xrightarrow{\varpi} S$ of $f: X\to S$ is unique up to unique isomorphism. Indeed, once this is shown, then $\cL=\coker(\cO_P\to \varphi_* \cO_X)$, and there is a unique splitting $\cL\hookrightarrow \varphi_*\cO_X$ such that the multiplication map sends $\cL^{\otimes 2}$ to $\cO_P$, which implies that both $\cL$ and $i$ are unique up to unique isomorphism. 

Since $\varphi$ is finite and Gorenstein, by Grothendieck duality we have 
\[
\varphi_* \omega_{X/P} \cong (\varphi_* \cO_X)^{\vee} \cong \cO_P \oplus \cL^{\vee}\cong \cL^{\vee} \otimes \varphi_* \cO_X. 
\]
Thus we have $\omega_{X/P}\cong \varphi^* \cL^{\vee}$ and hence 
\[
\omega_{X/S} \cong \omega_{X/P}\otimes \varphi^* \omega_{P/S} \cong \varphi^* (\cL^{\vee}\otimes  \omega_{P/S} ). 
\]
In particular, $\omega_{X/S}$ is $f$-globally generated as $\cL^{\vee}\otimes \omega_{P/S}$ restricts to $\cO_{\bP^1}(g-1)$ on every geometric fiber of $\varpi$. 
Hence 
\[
f_* \omega_{X/S}\cong f_*\varphi^* (\cL^{\vee}\otimes  \omega_{P/S} ) \cong \varpi_* (\cL^{\vee}\otimes \omega_{P/S}\otimes \varphi_* \cO_X) \cong \varpi_*(\cL^{\vee}\otimes \omega_{P/S}) \oplus \varpi_*  \omega_{P/S}. 
\]
Since $\omega_{P/S}$ restricts to $\cO_{\bP^1}(-2)$ on every geometric fiber of $\varpi$, we know that $\varpi_* \omega_{P/S} =0$. As a result, we have $f_* \omega_{X/S}\cong \varpi_*(\cL^{\vee}\otimes \omega_{P/S})$. Therefore, the relative canonical map $X\to \bfP:=\Proj_S \Sym f_* \omega_{X/S}$ factors through the closed embedding of $P \hookrightarrow \bfP$ induced by the $\varpi$-very ample line bundle $\cL^{\vee}\otimes \omega_{P/S}$. Since $\cO_P$ injects into $\varphi_* \cO_X$, we know that $X\to P$ is isomorphic to the scheme theoretic image map of $X\to \bfP$ up to a unique isomorphism. Thus we have shown that $F$ is fully faithful. 

To summarize, we have shown that $\sH_g$ and $\sH_g'$ are equivalent as fibered categories. Thus  the proof is finished by \eqref{eq:hyp-quotient-stack}. 
\end{proof}

\begin{thm}\label{thm:rib-stack}
The functor $\sR_g$ is represented by an algebraic stack of finite type over $\bk$. Moreover, we have an isomorphism of stacks
\[
\sR_g  \cong [\bA(\Ext^1(\Omega_{\bP^1}, \cO_{\bP^1}(-g-1)))/(\GL(2)/\bmu_{g+1})],
\]
where $\GL(2)/\bmu_{g+1}$-action is induced by the standard isomorphism $\GL(2)\cong \Aut(\bP^1, \cO_{\bP^1}(1))$. 
\end{thm}

\begin{proof}
First of all, we show that $\sR_g$ is a prestack, i.e. a category fibered in groupoids. Suppose $(f:X\to S)\in \sR_g(S)$. Let $g: S'\to S$ be a morphism of locally Noetherian $\bk$-schemes. Then we have $f': X'= X\times_S S' \to S'$ as the base change of $f$ under $g$. Clearly, $f'$ is the pull-back of $f$ under $g$ in the stack $\Curves_g^{\Gor}$. Thus it suffices to show that $f'\in \sR_g(S')$. Let $\varpi': P'=P\times_S S'\to S'$ and $\psi': P'\to X'$ be the base change of $\varpi$ and $\psi$ under $g$ respectively. Then clearly $P'/S'$ is a $\bP^1$-fibration, and $\psi'$ is a closed immersion such that $f'=\psi'\circ\varpi'$. Denote by $g_X: X'\to X$ the base change of $g$ under $f$.  Since $P$ is flat over $S$, we know that $g_X^* \cN$ is the ideal sheaf $\cN'$ of $P'$ in $X'$, which implies that $\cN'^2=0$ and $\cN'$ is a locally free $\cO_{P'}$-module of rank $1$. Thus $f'\in \sR_g(S')$. 

Next, we define a functor $\sR_g'$ from locally Noetherian $\bk$-schemes to groupoids as
\[
\sR_g'(S):=\left\{(P/S, \cL, \cE)\left|\begin{array}{l}
P/S\textrm{ is a $\bP^1$-fibration, $\cL$ is a line bundle on $P$ whose}\\
\textrm{restriction to every geometric fiber has degree $-g-1$,}\\
\textrm{and $\cE$ is an extension of $\Omega_{P/S}$ by $\cL$.}
\end{array}\right.\right\}.
\]
Here the morphisms of $\sR_g'(S)$ are defined in the way that $(P/S,\cL,\cE) \to (P'/S, \cL', \cE')$ consists of an isomorphism $h: P\xrightarrow{\cong} P'$ of $S$-schemes and an isomorphism between exact sequences $(0\to \cL \to \cE\to \Omega_{P/S}\to 0)$ and $(0\to h^*\cL' \to h^*\cE'\to h^*\Omega_{P'/S}\to 0)$ extending the natural identification $\Omega_{P/S} = h^*\Omega_{P'/S}$. It is clear that $\sR_g'$ is a prestack by flatness of $P/S$.
We claim that $\sR_g'$ is a stack and
\begin{equation}\label{eq:rib-quotient-stack}
\sR_g'\cong [\bA(\Ext^1(\Omega_{\bP^1}, \cO_{\bP^1}(-g-1)))/(\GL(2)/\bmu_{g+1})].
\end{equation}

Similar to \cite[Proof of Theorem 4.1]{AV04}, we consider the auxiliary fibered category $\tsR_g'$ whose objects over $S$ are pairs consisting of $(P/S, \cL, \cE)\in \sR_g(S)$ plus an isomorphism $p: (P, \cL) \to (\bP_S^1, \cO_{\bP_S^1}(-g-1))$ over $S$. The arrows in $\tsR_g'$ are arrows in $\sR_g'$ preserving $p$. From the standard theory of extensions, we know that  $\tsR_g'$ is indeed equivalent to  a category fibered in sets where $\tsR_g'(S)$ is equivalent to the set $\Ext^1_{\bP_S^1}(\Omega_{\bP_S^1/S},\cO_{\bP_S^1}(-g-1))$. By Leray's spectral sequence and flat base change $S\to \Spec(\bk)$, one can see that 
\[
\Ext^1_{\bP_S^1}(\Omega_{\bP_S^1/S},\cO_{\bP_S^1}(-g-1)) \cong \Ext^1_{\bP^1}(\Omega_{\bP^1}, \cO_{\bP^1}(-g-1))\otimes_{\bk} H^0(S, \cO_S).
\]
Moreover, the above isomorphism commutes with arbitrary base change $S'\to S$. 
As a result, the functor $\tsR'_g$ is represented by the $\bk$-scheme $\bA(\Ext^1(\Omega_{\bP^1}, \cO_{\bP^1}(-g-1)))$. Then following  similar arguments to \cite[Proof of Theorem 4.1]{AV04}, we have the equivalence \eqref{eq:rib-quotient-stack} as $\GL(2)/\bmu_{g+1}$ is precisely $\Aut(\bP^1, \cO_{\bP^1}(-g-1))$. 

Next, we show that $\sR_g$ and $\sR_g'$ are equivalent as fibered categories. We construct a functor $F: \sR_g'\to \sR_g$ following \cite[Proof of Theorem 1.2]{BE95}. Given $(P/S, \cL, \cE)\in \sR_g'(S)$, let $\cO_X$ be a sheaf of $\varpi^{-1} \cO_S$-modules obtained as the pullback of $\cO_{P}\xrightarrow{d}\Omega_{P/S}$ under $\cE \to \Omega_{P/S}$, so that we have a commutative diagram
\[
\begin{tikzcd}
    0 \arrow{r}{}& \cL \arrow{r}{} \arrow{d}{\id}& \cO_X \arrow{r}{}\arrow{d}{d'} & \cO_P\arrow{r}{}\arrow{d}{d} & 0\\
    0 \arrow{r}{}& \cL \arrow{r}{} & \cE \arrow{r}{\gamma} & \Omega_{P/S}\arrow{r}{}& 0
\end{tikzcd}
\]
We define a $\varpi^{-1} \cO_S$-algebra structure on $\cO_X$ in the following way. Let $a_1, a_2 \in \cO_P(U)$ and $x_1, x_2 \in \cE(U)$ where $U\subset P$ is open, so that $d(a_i) = \gamma(x_i)$ and hence $(a_i, x_i)\in \cO_X(U)$. Then we define
\[
(a_1, x_1) (a_2, x_2):= (a_1 a_2, a_1 x_2 +a_2x_1).
\]
The product is in $\cO_X(U)$ as $d$ is a derivation. Then we define $X:=\Spec \cO_X$ with the underlying topological space $|X|=|P|$.
Following \cite[Proof of Theorem 1.2]{BE95}, it is not hard to check that $X\to S$ is a family of ribbons and $\cE \cong \Omega_{X/S}|_{P}$.

To show that $F$ is an equivalence of fibered categories, it suffices to show that $F$ is fully faithful and essentially surjective.  For each family of ribbons $X\to S$, by smoothness of $P/S$ we have a short exact sequence 
\[
0\to \cN \to \Omega_{X/P}|_{P}\to \Omega_{P/S}\to 0.
\]
Thus following \cite[Proof of Theorem 1.2]{BE95} it is not hard to check that $(X\to S)$ is isomorphic to $F(P/S, \cN, \Omega_{X/P}|_P)$. This shows that $F$ is essentially surjective.

Finally, we show that $F$ is fully faithful. Let $(X\to S)\in \sR_g(S)$ be a family of ribbons, which is isomorphic to an image of $F(S)$ by essential surjectivity of $F$. Since both $\cE$ and $\cL$ are uniquely determined (up to unique isomorphism) by the ideal sheaf $\cN$ of $P$ in $X$, it suffices to show that $\cN$ is unique. This is clear if $S$ is reduced: notice that $P$ is reduced as $P\to S$ is flat with reduced fibers. Thus the ideal $\cN$ contains the nilradical sheaf $\cN_X$ of $X$, while $\cN^2=0$ implies that $\cN$ is contained in $\cN_X$, and hence $\cN = \cN_X$ is unique. 

It remains to show that $\cN$ is unique for a general base scheme $S$. 
Assume that we have two ideal sheaves $\cN$ and $\cN'$ of $\cO_X$ defining two $\bP^1$-fibrations $P/S$ and $P'/S$ respectively. Our goal is to show  $\cN =\cN'$. By \'etale descent, we may replace $S$ by an affine \'etale cover such that $P$ is a trivial $\bP^1$-fibrations. In other words, we may assume that $S = \Spec R$ and $P = \bP_{R}^1$. Let $x\in X$ be a scheme-theoretic point. It suffices to show the localizations $\cN_{x}$ and $\cN'_{x}$ are the same. Let $(B,\fm):=(\cO_{P,x},\fm_{P,x})$, then $R\to B$ is flat. Since any extension $\cE$ of $\cN$ and $\Omega_{P/S}$ splits locally at $x$, we may identify $\cO_{X,x}$ with $B\otimes_k k[t]/(t^2) = B[t]/(t^2)$, under which  $\cN_x = (t)$. Since $\cN'$ is a locally free $\cO_P'$-module of rank $1$, it is locally a principal ideal. Thus there exist $b_0, b_1\in B$ such that $\cN'_x = (b_0 + b_1 t)$. Since $B$ is a local ring of $P=\bP_{R}^1$, we have $B_{\red}\cong B\otimes_R R_{\red}$. Moreover, from the previous discussion on the case of reduced $S$ we have that $\phi_B(\cN_{x})= \phi_B(\cN'_x)$ has the same image under the quotient ring map $\phi_B: B[t]/(t^2) \to B_{\red}[t]/(t^2)$. Let $N_{B}$ be the nilradical of $B$. Then there exists $c_0 + c_1 t\in B[t]/(t^2)$ such that $\phi_B(c_0 + c_1t)$ is invertible and  
\[
b_0 + b_1 t - t(c_0 + c_1 t) \in N_B[t]/(t^2). 
\]
In particular, we have $b_0 \in N_B$ and $b_1 - c_0 \in N_B$.
The invertibility of $\phi_B(c_0 + c_1t)$ is equivalent to  $c_0 \in B^{\times} = B\setminus \fm$. Since $N_B\subset \fm$, we know that $b_1\not\in \fm$ which implies  $b_1\in B^{\times}$. Recall that $\cN'^2 = 0$ from definition, and hence $(b_0+b_1t)^2 = 0$ in $B[t]/(t^2)$ which implies that $b_0 b_1 = 0$. Since $b_1$ is invertible, we have $b_0 = 0$ and hence $\cN_x' = (b_1 t) = (t) = \cN_x$. Thus we have shown that $F$ is fully faithful.

To summarize, we have shown that $\sR_g$ and $\sR_g'$ are equivalent as fibered categories. Thus  the proof is finished by \eqref{eq:rib-quotient-stack}. 
\end{proof}

\begin{prop}\label{prop:hyp-rib-closedpt}
The hyperelliptic ribbon $[R_{\hyp}]$ is the only closed point in $\sH_g$ and $\sR_g$. The automorphism group $\Aut(R_{\hyp})$ is isomorphic to $\GL(2)/\bmu_{g+1}$. Moreover, $[R_{\hyp}]$ is contained in $\overline{\{x\}}$ for any point $x$ in $\sH_g$ or $\sR_g$.
\end{prop}

\begin{proof}
From the construction it is straightforward to see that $[R_{\hyp}]$ corresponds to the origin in both quotient stack presentations of $\sH_g$ and $\sR_g$. Thus the claim is clear by examining the $\GL(2)/\bmu_{g+1}$-action on both affine spaces, where the last statement follows from the fact that the $\bG_m (\cong \{\diag(t,t)\mid t\in \bk^\times\}/\bmu_{g+1})$-action on both affine spaces is a scaling.
\end{proof}

\subsection{Deformation theory of the hyperelliptic ribbon}\label{sec:ribbon-deform}

\begin{thm}\label{thm:hyp-ribbon-unobstructed}
The deformation of the hyperelliptic ribbon of genus $g\geq 2$ is unobstructed.
\end{thm}

\begin{proof}

Let $C$ be the hyperelliptic ribbon of genus $g\geq 2$, i.e. $C=(z^2=0) \subset \bP(1,1,g+1)$.

Denote by $S=\bP(1,1,g+1) \setminus \{[0,0,1]\}$. We have the closed immersion $\psi: \bP^1 \hookrightarrow C$ and the projection maps $f:S\to \bP^1$ and $\varphi = f|_C: C\to \bP^1$. We can treat $S$ as the total space of $L^\vee$ over $\bP^1$ where $L = \cO_{\bP^1}(-g-1)$. Denote by $\cI$ the ideal sheaf of $C$ in $S$.

From deformation theory we know that the tangent space (resp. the obstruction space) of $\Def_C$ is given by $\Ext_C^1(\bfL_C, \cO_C)$ (resp. $\Ext_C^2(\bfL_C, \cO_C)$), where $\bfL_C$ denotes the cotangent complex of $C$ over $\bk$. Since $C$ is a local complete intersection in $S$, we know that $\bfL_C$ is quasi-isomorphic to the mapping cone $[\cI/\cI^2\to \Omega_S\otimes \cO_C]$ by \cite[\href{https://stacks.math.columbia.edu/tag/08SL}{Tag 08SL}]{stacksproject}. Thus we have  the following long exact sequence
\begin{align*}\label{eq:conormal-exact}
0 &\to \Hom_C(\bfL_C, \cO_C) \to \Hom_C(\Omega_S\otimes \cO_C, \cO_C) \to \Hom_C(\cI/\cI^2 , \cO_C) \\
&\to \Ext_C^1(\bfL_C, \cO_C) \to \Ext_C^1(\Omega_S\otimes \cO_C, \cO_C)\to  \Ext_C^1(\cI/\cI^2 , \cO_C)\numberthis\\
& \to \Ext_C^2(\bfL_C, \cO_C) \to \Ext_C^2(\Omega_S\otimes \cO_C, \cO_C).
\end{align*}

To show unobstructedness of $\Def_C$, it suffices to show that $\Ext_C^2(\bfL_C, \cO_C) = 0$. First of all, we have 
\[
\Ext_C^i(\Omega_S\otimes \cO_C, \cO_C) \cong H^i(C, T_S\otimes \cO_C),
\]
which implies $\Ext_C^2(\Omega_S\otimes \cO_C, \cO_C)=0$
by the Grothendieck vanishing theorem. Next, notice that $\cI = \cO_S(-2 C_{\red})\cong f^* L^{\otimes 2}$ and thus we have $\cI/\cI^2 \cong \varphi^* L^{\otimes 2}$. This implies that 
\begin{align*}
\Ext_C^i(\cI/\cI^2, \cO_C)& \cong \Ext_C^i(\varphi^*L^{\otimes 2}, \cO_C)\cong H^i(C, \varphi^* L^{\otimes -2})\\
&\cong H^i(\bP^1, L^{\otimes -2} \otimes \varphi_* \cO_C)\cong H^i(\bP^1, L^{\otimes -2}\oplus L^{\vee}).
\end{align*}
Here we use $\varphi_* \cO_C\cong \cO_{\bP^1}\oplus L$ in the last isomorphism. Thus we have 
\[
\Ext_C^1(\cI/\cI^2 , \cO_C)\cong H^1(\bP^1, L^{\otimes -2}\oplus L^{\vee})=H^1(\bP^1, \cO_{\bP^1}(2g+2)\oplus\cO_{\bP^1}(g+1)) = 0.
\]
Therefore, we have $\Ext_C^2(\bfL_C, \cO_C) = 0$ by \eqref{eq:conormal-exact}. The proof is finished.
\end{proof}

\begin{cor}\label{cor:hyp-rib-unobstructed}
The deformation of a hyperelliptic curve or a ribbon of genus $g\geq 2$ is unobstructed.
\end{cor}

\begin{proof}
Let $C$ be a hyperelliptic curve or a ribbon. Then by Proposition \ref{prop:hyp-rib-closedpt} we know that $[R_{\hyp}]$ is contained in the closure of $[C]$ in $\Curves_{g}^{\Gor}$. By Theorem \ref{thm:hyp-ribbon-unobstructed} we know that $\Curves_g^{\Gor}$ is smooth at $[R_{\hyp}]$.  Thus $\Curves_g^{\Gor}$ is smooth at $[C]$ as well, i.e. the deformation of $C$ is unobstructed. 
\end{proof}

\begin{prop}\label{prop:hyp-rib-transversal}
Let $C = (z^2=0)\subset \bP(1,1,g+1)$ be the hyperelliptic ribbon of genus $g\geq 2$. Let $S$ be the smooth locus of $\bP(1,1,g+1)$. Recall that $\Def_C$ represents the functor of deformations of $C$. Denote by $\Def_{C\hookrightarrow S}$ (resp. $\Def'_{C}$) the functor of embedded deformations of $C$ in $S$ (resp. locally trivial deformations of $C$).
Then the tangent space of $\Def_C$ is a direct sum of  the image of the tangent space of $\Def_{C\hookrightarrow S}$ and the tangent space of $\Def_C'$. Moreover, this can be written as 
\[
\Ext_C^1(\bfL_C, \cO_C) \cong H^0(\bP^1, L^{\otimes -2}) \oplus \Ext_{\bP^1}^1(\Omega_{\bP^1}, L),
\]
where $L = \cO_{\bP^1}(-g-1)$. 
\end{prop}

\begin{proof}
Denote by $\cN_{C/S}:=(\cI/\cI^2)^{\vee}$ the normal sheaf of $C$ in $S$. 
From deformation theory, we know that the tangent space of $\Def_C$ (resp. of $\Def'_C$, of $\Def_{C\hookrightarrow S}$) is given by $\Ext_C^1(\bfL_C, \cO_C)$ (resp. $H^1(C, T_C)$, $H^0(C, \cN_{C/S})$). Recall from \eqref{eq:conormal-exact} the long exact sequence
\[
H^0(C, T_S\otimes\cO_C) \xrightarrow{\beta}H^0(C, \cN_{C/S}) \xrightarrow{\gamma} \Ext_C^1(\bfL_C, \cO_C) \to H^1(C, T_S\otimes\cO_C)\to 0.
\]
Thus it suffices to show the following two statements.
\begin{enumerate}
    \item $\mathrm{im}(\gamma) \cong H^0(\bP^1, L^{\otimes -2})$;
    \item $H^1(C, T_S\otimes \cO_C) \cong H^1(C, T_C) \cong \Ext_{\bP^1}^1(\Omega_{\bP^1}, L)$.
\end{enumerate}

We have the normal exact sequence
\[
0 \to T_C \to T_S \otimes \cO_C \to \cN_{C/S}.
\]
From the proof of Theorem \ref{thm:hyp-ribbon-unobstructed} we know that $\cN_{C/S} \cong \varphi^* L^{\otimes -2}$ is a line bundle on $C$.
By local computation, we know that the image of $T_S \otimes \cO_C \to \cN_{C/S}$ is precisely $\cN \cdot \cN_{C/S}$ where $\cN$ is the nilradical sheaf of $\cO_C$. Thus we have a short exact sequence
\[
0 \to T_C \to T_S \otimes \cO_C \to \cN\cdot \cN_{C/S}\to 0,
\]
which gives a long exact sequence
\begin{align*}
0& \to H^0(C, T_C) \to H^0(C, T_S\otimes \cO_C) \xrightarrow{\beta'} H^0(C, \cN\cdot \cN_{C/S})\\& \to H^1 (C, T_C) \to H^1(C, T_S\otimes \cO_C) \to H^1(C, \cN\cdot \cN_{C/S}).
\end{align*}

We first show that $H^1(C, \cN\cdot \cN_{C/S})=0$. We have 
\[
\varphi_*(\cN\cdot \cN_{C/S})\cong \varphi_*(\cN \cdot \varphi^* L^{\otimes -2}) \cong (\varphi_*\cN) \otimes L^{\otimes -2}\cong L\otimes L^{\otimes -2} \cong L^{\vee}.
\]
Thus 
\begin{equation}\label{eq:normal-cohom}
H^i(C, \cN\cdot \cN_{C/S})\cong H^1(\bP^1, \varphi_*(\cN\cdot \cN_{C/S})) \cong H^i(\bP^1, L^{\vee}) \cong H^i(\bP^1, \cO_{\bP^1}(g+1)).
\end{equation}
As a result, we have $H^1(C, \cN\cdot \cN_{C/S}) \cong H^1(\bP^1, \cO_{\bP^1}(g+1))  = 0$. 

Next, we show that $\beta'$ is surjective. From \eqref{eq:normal-cohom} we have 
\[
h^0(C, \cN\cdot \cN_{C/S}) = h^0(\bP^1, \cO_{\bP^1}(g+1))  = g+2.
\]
By Proposition \ref{prop:hyp-rib-closedpt} we have $h^0(C, T_C) = \dim \Aut(C) = 4$. It is clear that $T_S$ fits into an exact sequence
\[
0 \to f^*L^\vee \to T_S \to f^* T_{\bP^1}\to 0,
\]
which implies the following exact sequence
\[
0 \to \varphi^*L^\vee \to T_S\otimes \cO_C \to \varphi^* T_{\bP^1}\to 0.
\]
Thus we have an exact sequence
\begin{align*}\label{eq:tangent-exact}
    0 &\to H^0(C, \varphi^*L^\vee)\to   H^0 (C, T_S\otimes \cO_C) \to H^0(C, \varphi^* T_{\bP^1}) \\
    &\to H^1(C, \varphi^*L^\vee) \to   H^1 (C, T_S\otimes \cO_C) \to H^1(C, \varphi^* T_{\bP^1}) \to H^2(C, \varphi^*L^\vee) =0\numberthis
\end{align*}
We have $H^1(C, \varphi^* L^\vee) \cong H^1(\bP^1, L^\vee \oplus \cO_{\bP^1}) =0$. Thus we can determine the dimension as 
\begin{align*}
h^0(C, T_S\otimes \cO_C) & = h^0(C, \varphi^* L^{\vee}) + h^0(C, \varphi^* T_{\bP^1}) \\
&= h^0(\bP^1, L^{\vee} \oplus \cO_{\bP^1}) + h^0(\bP^1, T_{\bP^1} \oplus T_{\bP^1}\otimes L) = g+6.
\end{align*}

Thus we have $h^0(C, T_S\otimes \cO_C) = h^0(C, T_C) + h^0(C, \cN\cdot \cN_{C/S})$, which implies that $\beta'$ is surjective. 

Combining $H^1(C, \cN\cdot \cN_{C/S}) = 0$ and surjectivity of $\beta'$, we conclude that $H^1(C, T_C) \cong H^1(C, T_S\otimes \cO_C)$ and that $\mathrm{im}(\beta) = H^0(C, \cN \cdot \cN_{C/S})\subset H^0(C, \cN_{C/S})$. We have
\[
\varphi_* \cN_{C/S} \cong (\varphi_* \cO_C)\otimes \cL^{\otimes -2} \cong L^{\vee} \oplus L^{\otimes -2}. 
\]
Under this isomorphism, $\varphi_*(\cN\cdot \cN_{C/S})$ is mapped to $L^{\vee}$. Thus we have 
\[
\mathrm{im}(\gamma) = \coker(\beta)\cong H^0(\bP^1, L^{\vee} \oplus L^{\otimes -2})/ H^0(\bP^1, L^{\vee}) \cong H^0(\bP^1, L^{\otimes -2}).
\]
This proves (1). 

Finally we prove (2). By \eqref{eq:tangent-exact} and the fact that $H^1(C, \varphi^*L^{\vee}) = 0$, we have 
\begin{align*}
H^1(C, T_S\otimes \cO_C)&\cong H^1(C, \varphi^* T_{\bP^1}) \cong H^1 (\bP^1, (\varphi_* \cO_C) \otimes T_{\bP^1}) \\&\cong H^1(\bP^1, T_{\bP^1} \oplus  T_{\bP^1} \otimes L) \cong  H^1(\bP^1,  T_{\bP^1} \otimes L) \cong \Ext^1(\Omega_{\bP^1}, L). 
\end{align*}
This together with $H^1(C, T_C) \cong H^1(C, T_S\otimes \cO_C)$ proves (2). 
\end{proof}

\section{Moduli of curves via canonical maps} \label{sec:chow-ss-can}

In this section, we construct a new stack $\sX_g$  parameterizing Gorenstein smoothable curves $C$ of genus $g\geq 3$ with ample and basepoint free canonical bundle such that its cycle-theoretic image under the canonical map is Chow semistable (see Definition \ref{def:sX_g}). 
The purpose of introducing $\sX_g$ is to (potentially) provide a wall crossing replacing hyperelliptic curves by ribbons, which is verified  in Section \ref{sec:hyp-flip} for $g=4$. The main result of this section (Theorem \ref{thm:sX_g-gms}) shows that there exists an open neighborhood $\sX_g^{\circ}$ of $\sH_g\cup \sR_g$ in $\sX_g$ that admits a good moduli space isomorphic to an open subscheme of the Chow quotient space $\fX_g^{\rm c}$.

\subsection{Stack of Gorenstein curves with Chow semistable cycle}

\begin{defn-prop}\label{def:curves_g,1} There exists an open substack  $\Curves_{g,1}\hookrightarrow \Curves_g^{\Gor}$ that parameterizes curves $C\in \Curves_g^{\Gor}$ whose canonical bundle $\omega_C$ is ample and basepoint free. Moreover, $\cM_g$ is an open substack of $\Curves_{g,1}$.
\end{defn-prop}

\begin{proof}
Since $\Curves_{g}^{\Gor}$ is an algebraic stack locally of finite type over $\bk$, it suffices to show that for every scheme $U$ of finite type over $\bk$ together with a smooth morphism $U\to \Curves_g^{\Gor}$, the locus 
\[
V:=\{u\in U\mid \omega_{C_u} \textrm{ is ample and basepoint free}\}
\]
is open in $U$. Here $f: C_U\to U$ denotes the pull-back of the universal curve over $ \Curves_g^{\Gor}$, and $C_u:=C_U\times_U k(u)$. Since $f$ is Gorenstein with equidimensional fibers of dimension $1$, we know that $\omega_{C_U/U}$ is an invertible sheaf that commutes with base change. Moreover,  $f_* \cO_{\cU}\cong \cO_U$ and $R^1 f_* \cO_{C_U}$ is a locally free sheaf of rank $g$. By Grothendieck duality, $R^1 f_*\omega_{C_U/U}\cong \cO_U$ and the sheaf $f_* \omega_{C_U/U}\cong (R^1 f_* \cO_{C_U})^*$ is also locally free of rank $g$.
 
Clearly, the ampleness of $\omega_{C_u}$ is an open condition; denote the locus of $u$ for which $\omega_{C_u}$ is ample by $U'\subset U$.  The restriction $f': C_{U'}= f^{-1}(U')\to U'$ of $f$ is a projective morphism of schemes of finite type over $\bk$. By cohomology and base change, we know that $f'_* \omega_{C_{U'}/U'}$ commutes with base change. Thus for $u\in U'$,  $\omega_{C_u}$ is basepoint free if and only if $\psi_{f'}:f'^*f'_* \omega_{C_{U'}/U'}\to \omega_{C_{U'}/{U'}}$ is surjective along $C_{u}$. If $u\in V$, then $\psi_{f'}$ being surjective on $C_{u}$ implies that it is surjective in a neighborhood of $C_u$ as the surjectivity of a morphism between locally free sheaves is an open condition. Thus $V$ contains an open neighborhood of $u$ in $U'$ by properness of $f'$, which implies that $V$ is open in $U$. This shows that $\Curves_{g,1}$ is an open substack of $\Curves_g^{\Gor}$. Moreover, since every smooth curve $C$ of genus $g\geq 2$ satisfies that $\omega_C$ is basepoint free, we know that $\cM_g\hookrightarrow \Curves_{g,1}$ is an open substack.
\end{proof}

\begin{defn}\label{def:sX_g} Let $\osX_{\!g}$ be the stack-theoretic closure of $\cM_g$ in $\Curves_{g,1}$. %
Since $\cM_g$ is a smooth irreducible open substack in $\Curves_g^{\Gor}$, we know that $\osX_{\!g}$ is an irreducible component of $\Curves_{g,1}$ with reduced stack structure. 

Recall from Section \ref{sec:chow-vgit} that $\Chow_{g,1}$ is the seminormalization of the irreducible component of the Chow scheme $\boldsymbol{C}_{1,2g-2}(\bP^{g-1})$ containing smooth canonical genus $g$ curves. Denote by $\osX_{\!g}^{\rm sn}$ the seminormalization of $\osX_{\!g}$ (see e.g. \cite[Definition 13.6]{ABB+} for the definition of seminormalization of algebraic stacks). 
Then we have the Hilbert--Chow morphism $\overline{\Phi}_g^{\rm sn}: \osX_{\!g}^{\rm sn} \to [\Chow_{g,1}/\PGL(g)]$.  There is an open substack
$(\overline{\Phi}_g^{\rm sn})^{-1}(\sX_g^{\rm c})\hookrightarrow \osX_{\!g}^{\rm sn}$,
where $\sX_g^{\rm c} = [\Chow_{g,1}^{\rm ss}/\PGL(g)]$ as in Definition \ref{def:chowss}.

Let $\sX_g \hookrightarrow \osX_{\!g}$ be an open substack with seminormalization $\sX_g^{\rm sn} = (\overline{\Phi}_g^{\rm sn})^{-1}(\sX_g^{\rm c})$. Since seminormalization is  a universal homeomorphism, such an open substack $\sX_g$  exists. In particular, $\sX_g$ parameterizes smoothable curves $C\in \Curves_{g}^{\Gor}$ such that $\omega_C$ is ample and basepoint free, and that the Hilbert--Chow image of $C$ induced by $|\omega_C|$ is a Chow-semistable $1$-cycle in $\bP^{g-1}$. Denote by $\Phi_g^{\rm sn}: \sX_g^{\rm sn}\to \sX_g^{\rm c}$ the restriction of $\overline{\Phi}_g^{\rm sn}$.
\end{defn}

\begin{thm}\label{thm:sX_g-smooth}
Let $g\geq 3$ be an integer. Then $\sX_g$ is an irreducible algebraic stack of finite type with affine diagonal that contains $\cM_g$ as a dense open substack. Moreover, $\sX_g$ contains a smooth open neighborhood of the point $[R_{\hyp}]$ in $\Curves_{g,1}$.
\end{thm}

To prove this theorem, we first show a more general boundedness result.

\begin{lem}\label{lem:g,1-ft}
For any $g\geq 2$, the algebraic stack $\Curves_{g,1}$ is of finite type over $\bk$.
\end{lem}

\begin{proof}
 Since $\Curves_g^{\Gor}$ is locally of finite type, it suffices to show that $\Curves_{g,1}$ is quasi-compact, or equivalently, the curves $C\in \Curves_{g,1}(\bk)$ are bounded. Let $\varphi: C\to \bP^{g-1}$ be the finite morphism induced by the basepoint free linear system $|\omega_C|$. We first show that $\omega_C^{\otimes 3}$ is very ample. Since $H^1(\bP^{g-1}, \varphi_* \cO_C (2)) \cong H^1(C, \omega_C^{\otimes 2}) \cong H^0(C, \omega_C^{\vee}) ^{\vee}= 0$, we know that $\varphi_* \cO_C$ is $3$-regular in the sense of Castelnuovo--Mumford regularity. By \cite[Theorem 1.8.3]{positivity-I}, we know that 
 \[
 H^0(\bP^{g-1}, \varphi_* \cO_C (3)) \otimes H^0(\bP^{g-1}, \cO_{\bP^{g-1}}(m)) \to H^0(\bP^{g-1}, \varphi_* \cO_C (3+m))
 \]
 is surjective for every $m\geq 0$. As a result, we know that 
 \[
 H^0(C,\omega_C^{\otimes 3})\otimes H^0(C,\omega_C^{\otimes m})\to H^0(C, \omega_C^{\otimes(3+m)})
 \]
 is surjective for every $m\geq 0$, which implies that $\omega_C^{\otimes 3}$ is very ample. By the Riemann--Roch formula for Gorenstein curves \cite[\href{https://stacks.math.columbia.edu/tag/0BS6}{Tag 0BS6}]{stacksproject}, we know that 
 \[
 h^0(C,\omega_C^{\otimes 3m}) =  \chi(C, \omega_C^{\otimes 3m}) = 3m(2g-2) - g +1. 
 \]
 Thus the Hilbert polynomial of $(C, \omega_C^{\otimes 3})$ is independent of the choice of $C$. This together with very ampleness of $\omega_C^{\otimes 3}$ implies the boundedness  of $C\in \Curves_{g,1}(\bk)$. 
\end{proof}

\begin{proof}[Proof of Theorem \ref{thm:sX_g-smooth}]
First we show that $\sX_g$ is irreducible and contains a dense open substack $\cM_g$. By Lemma \ref{lem:sm-chow} we know that the $1$-cycle of a smooth canonical curve of genus $g$ is Chow stable. The $1$-cycle as the Hilbert--Chow image of a smooth hyperelliptic curve of genus $g$ is the doubled rational normal curve in $\bP^{g-1}$, which is Chow polystable by Lemma \ref{lem:sm-chow}. Thus we have $\cM_g\subset \sX_g$. 
Since $\cM_g$ is open and dense in $\osXg$, it is also open and dense in $\sX_g$ which implies the irreducibility of $\sX_g$.

Next, since $\sX_g$ is a locally closed substack of $\Curves_{g,1}$, by Lemma \ref{lem:g,1-ft} we know that $\sX_g$ is of finite type. In addition, since every curve $C$ in $\Curves_{g,1}$ satisfies that $\omega_C$ is ample, we conclude that $\Curves_{g,1}$ and hence $\sX_g$ has affine diagonal.

Finally, we show that $\sX_g$ contains a smooth open neighborhood of $[R_{\hyp}]$ in $\Curves_{g,1}$. By Theorem \ref{thm:hyp-ribbon-unobstructed} we know that the deformations of $R_{\rm hyp}$ are unobstructed, which implies that $\Curves_{g,1}$ is smooth and hence irreducible in a neighborhood of $R_{\rm hyp}$. Clearly, $\osXg$ is an irreducible component of $\Curves_{g,1}$. Since $R_{\hyp}$ is smoothable as it is a degeneration of smooth hyperelliptic curves, the irreducibility of $\Curves_{g,1}$ at $[R_{\hyp}]$ implies that $\osXg$ contains a smooth open neighborhood of  $R_{\hyp}$ in $\Curves_{g,1}$. Moreover, by openness of Chow-semistability we know that $\sX_g$ also contains a smooth open neighborhood of  $R_{\hyp}$ in $\Curves_{g,1}$. 
\end{proof}

\subsection{Closed substacks parameterizing hyperelliptic curves and ribbons}

\begin{prop}\label{prop:hyp-rib-closed-substack}
Both $\sH_g$ and $\sR_g$ are  closed substacks of $\sX_g$. 
\end{prop}

\begin{proof}
From Definition \ref{def:hyp-rib-stacks}, we know that both $\sH_g$ and $\sR_g$ admit monomorphisms to $\Curves_g^{\Gor}$. Since $[R_{\hyp}]$ is contained in the closure of every point in $\sH_g$ and $\sR_g$ by Proposition \ref{prop:hyp-rib-closedpt}, we know that the images of $\sH_g$ and $\sR_g$ in $\Curves_g^{\Gor}$ are contained in every open neighborhood of $[R_{\hyp}]$. Thus by Theorem \ref{thm:sX_g-smooth} we have monomorphisms $\iota_{\hyp}:\sH_g \to \sX_g$ and $\iota_{\rib}:\sR_g \to \sX_g$. Since all three stacks $\sH_g$, $\sR_g$ and $\sX_g$ are of finite type over $\bk$, we know that both $\iota_{\hyp}$ and $\iota_{\rib}$ are quasi-finite. Since monomorphisms of algebraic spaces are separated \cite[\href{https://stacks.math.columbia.edu/tag/042N}{Tag 042N}]{stacksproject}, we know that both $\iota_{\hyp}$ and $\iota_{\rib}$ are representable by schemes by \cite[\href{https://stacks.math.columbia.edu/tag/0418}{Tag 0418}]{stacksproject}. Therefore, to show $\iota_{\hyp}$ and $\iota_{\rib}$ to be closed immersions, it suffices to show their properness by \cite[\href{https://stacks.math.columbia.edu/tag/04XV}{Tag 04XV}]{stacksproject}, which further reduces to checking the existence part of the valuative criterion for properness given their separatedness.

Let $f:X\to \Spec\, R$ be an object in $\sX_g(R)$ where $R$ is a DVR with generic point $\eta$ and closed point $\kappa$. By the semicontinuity theorem, we know that $f_* \omega_{X/R}$ is locally free of rank $g$. Moreover, from the definition of $\sX_g$ we know that $\omega_{X/R}$ is $f$-ample and $f$-globally generated. Thus the relative canonical map gives a finite morphism $\varphi: X \to \bfP:=\Proj_R \Sym f_* \omega_{X/R}$. 

Firstly, we assume that $X_{\eta}\in \sH_g(\eta)$ and aim to show $f\in \sH_g(R)$. 
Denote by $P$ the scheme theoretic image of $\varphi$. Then we have an injection $\cO_P \hookrightarrow \varphi_*\cO_X$. Since $X$ is flat over $R$, we know that $\varphi_*\cO_X$ is a torsion free $R$-module, which implies that $\cO_P$ is a torsion free $R$-module. Hence $P$ is also flat over $R$. 
Since $X_{\bar{\eta}}$ is a hyperelliptic curve, we know that $P_{\bar{\eta}}$ is a rational normal curve in $\bfP_{\bar{\eta}}$. Hence the $1$-cycle $\varphi_* [X_{\bar{\eta}}]$ is Chow polystable by Lemma \ref{lem:sm-chow}. This implies that the $1$-cycle $\varphi_* [X_{\bar{\kappa}}]$ is also a doubled rational normal curve, which implies that $P_{\bar{\kappa}}$ is also a rational normal curve by flatness of $P/R$. As a result, $P\to R$ is a $\bP^1$-fibration. Moreover, we have $\varphi_{\bar{\kappa},*} \cO_{X_{\bar{\kappa}}}$ has rank $2$, which implies that $\varphi_{\bar{\kappa},*} \cO_{X_{\bar{\kappa}}}$ is locally free of rank $2$ by Theorem \ref{thm:hyp=Weierstrass}. Since $\varphi$ is affine, we know that $\varphi_{\bar{\kappa},*} \cO_{X_{\bar{\kappa}}}$ is the pull-back of $\varphi_*\cO_X$ to the geometric fiber $P_{\bar{\kappa}}$ by \cite[\href{https://stacks.math.columbia.edu/tag/02KG}{Tag 02KG}]{stacksproject}. Thus $\varphi_* \cO_X$ has constant rank $2$ and hence is locally free, which implies that $f\in \sH_g(R)$.

Next, we assume that $X_{\eta}\in \sR_g(\eta)$ and aim to show $f\in \sR_g(R)$. Since $X$ is flat over $R$ with $X_{\eta}$ irreducible, we know that $X$ is also irreducible. Let $P:= X_{\red}$ with the natural closed immersion $\psi: P\hookrightarrow X$. Since $P$ is integral, we know that $P$ is flat over $R$. From the previous discussion, we know that $\varphi\circ\psi: P\to \bfP$ is a finite morphism that maps the geometric general fiber $P_{\bar{\eta}}$ isomorphically to a rational normal curve in $\bfP_{\bar{\eta}}$, and the image of $P_{\bar{\kappa}}$ is supported on a rational normal curve in $\bfP_{\bar{\kappa}}$. By flatness of $(\varphi\circ\psi)_* \cO_P$ over $R$, we conclude that $\varphi\circ\psi$ is a closed immersion which makes $P\to \Spec\, R$ a $\bP^1$-fibration. Let $\cN$ be the nilradical  sheaf $X$. Then we have $\cN_{\eta}^2=0$ and $\cN_{\eta}$ is a locally free $\cO_{P_{\eta}}$-module of rank $1$ by assumption. Thus $\cN^2$ is supported over $\kappa$ which implies $\cN^2=0$ by flatness of $X/R$.  Moreover, the flatness of $X$ and $P$ over $R$ implies the flatness of $\cN$ over $R$. Thus $\cN$ is a locally free $\cO_P$-module of rank $1$ as $P/R$ is a $\bP^1$-fibration, which implies that $f\in \sR_g(R)$. The proof is finished. 
\end{proof}

\begin{prop}\label{prop:hyp-rib-intersection}
Th point $[R_{\hyp}]$ is the only point in the intersection $\sH_g\cap \sR_g$ and the only closed point in $\sH_g\cup \sR_g$. Moreover, the stack $\sX_g$ is smooth in a neighborhood of $\sH_g\cup \sR_g$, and
\[
T_{\sX_g, [R_{\hyp}]} = T_{\sH_g, [R_{\hyp}]} \oplus T_{\sR_g, [R_{\hyp}]}.
\]
\end{prop}

\begin{proof}
Notice that the canonical map of a hyperelliptic curve is never a closed embedding by Proposition \ref{prop:hyp-canonical-map}, while the canonical map of a non-hyperelliptic ribbon is always a closed embedding by Proposition \ref{prop:ribbon-canonical-map}. Thus $[R_{\hyp}]$ is the only point in $\sH_g\cap \sR_g$. It is the only closed point of $\sH_g\cup \sR_g$ by Proposition \ref{prop:hyp-rib-closedpt}. The last statement follows from Corollary \ref{cor:hyp-rib-unobstructed} and Proposition \ref{prop:hyp-rib-transversal}.
\end{proof}

\subsection{Normality of the Chow scheme}

In this subsection, we show that $\Chow_{g,1}$ is normal in an open neighborhood of the locus parameterizing doubled rational normal curves. 

Let $\boldsymbol{H}_{\chi_g}(\bP^{g-1})$ be the Hilbert scheme of $\bP^{g-1}$ for the Hilbert polynomial $\chi_g(t):=(2g-2)t - g+1$.  Let $\Hilb_{g,1}$ be the closure of the locus parameterizing canonical curves of genus $g$ in $\boldsymbol{H}_{\chi_g}(\bP^{g-1})$ with reduced scheme structure. Then we have a Hilbert--Chow morphism $\Phi_{g,1}: \Hilb_{g,1}^{\rm sn}\to \Chow_{g,1}$, where $\nu_{g,1}:\Hilb_{g,1}^{\rm sn}\to \Hilb_{g,1}$ denotes the seminormalization. From the construction, we know that $\Phi_{g,1}$ is proper birational. 

\begin{lem}\label{lem:ribbon-cycle}
Let $[C]\in \boldsymbol{H}_{\chi_g}(\bP^{g-1})$ be a closed subscheme of $\bP^{g-1}$ whose corresponding $1$-cycle is a doubled rational normal curve. Then $C$ is a canonical ribbon of genus $g$.
\end{lem}

\begin{proof}
Let $C \subset \bP^{g-1}$ be a non-reduced curve with associated Chow cycle $2[C_0]$ where $C_0\subset\bP^{g-1}$ is a rational normal curve.  Let $\cJ$ be the largest finite length ideal sheaf on $C$ and define $C_{\rib}$ as the corresponding closed subscheme of $C$, i.e. $\cO_{C_{\rib}} = \cO_C / \cJ$.  By construction, $C_{\rib}$ is $S_1$ and $(C_{\rib})_{\red} = C_{\red} = C_0$ and the generic points of $C_{\rib}$ and $C$ agree.  We first show that $C_{\rib}$ is a ribbon.  Let $\cN_{\rib}$ be the ideal of nilpotents of $\cO_{C_{\rib}}$.  As $\cO_{C_{\rib}}$ has length $2$ at its generic point and $\cO_{C_{\rib}}$ has no finite length ideal sheaves, we conclude $\cN_{\rib}^2 = 0$. Thus $\cN_{\rib}$ has the structure of an $\cO_{C_0}$-module. Since $\cO_{C_{\rib}}$ has no nonzero finite length ideal sheaves, we know that $\cN_{\rib}$ is a torsion free $\cO_{C_0}$-module of rank $1$. 
Therefore, by smoothness of $C_0$ we conclude that $\cN_{\rib}$ is a rank 1 locally free $\cO_{C_0}$-module, and $C_{\rib}$ is indeed a ribbon.  

From the exact sequence \[ 0 \to \cJ \to \cO_C \to  \cO_{C_{\rib}} \to 0\] and additivity of Euler characteristics, we conclude 
\[ \chi(\cO_{C_{\rib}}) = \chi(\cO_C) - \chi(\cJ) \]
and as $\cJ$ is supported on a finite set, this implies 
\[ g(C_{\rib}) \ge g(C) =  g\]
with equality if and only if $\cJ = 0$ and $C_{\rib} = C$.  Let $g_{\rib} = g(C_{\rib})$ so $C_{\rib}$ is a ribbon with conormal bundle $L = \cO_{\bP^1}(-g_{\rib} - 1)$.  From the exact sequence 
\[ 0 \to L \to \cO_{C_{\rib}} \to \cO_{C_0} \to 0\]
there is an induced exact sequence 
\[ 0 \to \cI_{C_{\rib}} \to \cI_{C_0} \to L \to 0 \]
where $\cI_{C_{\rib}}$, $\cI_{C_0}$ denote the ideal sheaves of $C_{\rib}$, $C_0$ in $\bP^{g-1}$.  By \cite[Lemma 5.4]{BE95}, $\cI_{C_0}/\cI_{C_0}^2 \cong \cO_{\bP^1}(-g-1)^{\oplus (g-2)}$ which gives a surjection  
\[ \cO_{\bP^1}(-g-1)^{\oplus (g-2)} \to L \cong \cO_{\bP^1}(-g_{\rib} - 1) \to 0.\]
This implies $-g_{\rib} - 1 \ge -g -1$, i.e. $g \ge g_{\rib}$.  Combining this with the previous inequality $g_{\rib} \ge g$, we conclude $g_{\rib} = g$ and hence $\cJ = 0$ and $C_{\rib} = C$.  Therefore, if $C$ has image $2[C_0]$ in the Hilbert--Chow morphism, then $C$ is a canonical ribbon.  
\end{proof}

\begin{prop}\label{prop:hilb-chow}
There exists a smooth $\PGL(g)$-equivariant open subscheme $\tU_{g,1}$ of $\Hilb_{g,1}$ such that the following hold.
\begin{enumerate}
    \item $\tU_{g,1}$ contains the locus $H_{\rm rib}$ parameterizing canonical ribbons.
    \item The image $U_{g,1}:=\Phi_{g,1}(\tU_{g,1})$ is an open subscheme of $\Chow_{g,1}^{\rm ss}$. Moreover, $U_{g,1}$ contains  $Z_{\rib}$ which is the reduced locally closed subscheme of $\Chow_{g,1}$ parameterizing doubled rational normal curves.
    \item Every subscheme $[C\hookrightarrow \bP^{g-1}]\in \tU_{g,1}$ that is not a canonical ribbon satisfies that $[C]\in \Curves_{g}^{\rm lci +}$, $\omega_C$ is very ample and induces the closed embedding $C\hookrightarrow \bP^{g-1}$. Conversely, every such curve $C$ whose canonical image gives a Chow semistable $1$-cycle belongs to $\tU_{g,1}$.
    In particular,  $U_{g,1}\setminus Z_{\rm rib}$ only contains reduced $1$-cycles.
    \item The Hilbert--Chow morphism $\tU_{g,1} \to U_{g,1}$ is proper and surjective, and induces an isomorphism between $\tU_{g,1}\setminus H_{\rib}$ and $U_{g,1}\setminus Z_{\rib}$. Moreover, $U_{g,1}$ is normal.
\end{enumerate}
\end{prop}

\begin{proof}
    We construct the schemes $U_{g,1}$ and $\tU_{g,1}$ in several steps. Since all constructions are $\PGL(g)$-equivariant, we will omit this adjective in the entire proof whenever we talk about open subschemes in $\Hilb_{g,1}$ or $\Chow_{g,1}$.

    We first show that there is an open neighborhood $U_1$ of $H_{\rib}$ in $\Hilb_{g,1}$ that parameterizes curves $C$ with lci singularities and $\omega_C$ is ample and base point free.
    From previous discussions we know that the Hilbert--Chow morphism $\Phi_{g,1}: \Hilb_{g,1}^{\rm sn}\to \Chow_{g,1}$ is a projective birational morphism between projective  seminormal varieties. Moreover, $\Phi_{g,1}^{-1}(2[C_0])$ only consists of canonical ribbons by Lemma \ref{lem:ribbon-cycle}. 
    Let $f:\cC_{g,1}\to \Hilb_{g,1}$ be the universal family. Since $f$ is flat and proper of relative dimension $\leq 1$, it induces a morphism of algebraic stacks $\xi: [\Hilb_{g,1}/\PGL(g)] \to \Curves$. From Section \ref{sec:stack-curves} and Definition-Proposition \ref{def:curves_g,1}, we know that $\Curves_{g,1}\cap \Curves_{g}^{\lci}$ is an open substack of $\Curves$. Since canonical ribbons of genus $g$ belong to $\Curves_{g,1}\cap \Curves_{g}^{\lci}$,  there exists an open neighborhood $U_1$ of $H_{\rib}$ in $\Hilb_{g,1}$ such that $[U_1/\PGL(g)] = \xi^{-1}(\Curves_{g,1}\cap \Curves_{g}^{\lci})$. 

    Next, we show that there exists an open neighborhood $U_3$ of $H_{\rib}$ in $U_1$ that precisely parameterizes curves $[C]\in U_1$ such that $\omega_C$ is very ample and gives the embedding $C\hookrightarrow \bP^{g-1}$. 
    Since curves in $H_{\rib}$ are non-degenerate, by openness of non-degeneracy there exists an open neighborhood $U_2$ of $H_{\rib}$ in $U_1$ that precisely parameterizes non-degenerate curves in $U_1$. 
    Denote by $f_2: \cC_{U_2}\to U_2$ the base change of $f$. Let us consider the invertible sheaf $\cF:=\omega_{\cC_{U_2}/U_2}\otimes \cO_{\cC_{U_2}}(-1)$ on $\cC_{U_2}$. For every $u\in U_2$, we have $C_{u} \in \Curves_{g,1}\cap \Curves_{g}^{\lci}$. Thus by the Riemann--Roch formula for Gorenstein curves \cite[\href{https://stacks.math.columbia.edu/tag/0BS6}{Tag 0BS6}]{stacksproject} we have 
    \[
    h^0(C_u, \omega_{C_u} \otimes \cO_{C_u}(-1)) = h^1(C_u, \cO_{C_u}(1)) = h^0(C_u, \cO_{C_u}(1)) - \chi(C_u ,\cO_{C_u}(1))\geq g - (g-1) = 1.
    \]
    Clearly, the equality holds if $u\in H_{\rib}$ as $\cF$ is relatively trivial over $H_{\rib}$. By the semicontinuity theorem we know that there exists an  open neighborhood $U_2'$ of $H_{\rib}$ in $U_2$ such that $u\in U_2'$ if and only if $h^0(\cF_u) = 1$. In particular, for every $u\in U_3$ there is a non-zero homomorphism $s_u:\cO_{C_u}(1)\to \omega_{C_u}$ that is unique up to scaling. Since $s_u$ is clearly an isomorphism for $u\in H_{\rib}$, by openness of isomorphisms we know that there exists an open neighborhood $U_3$ of $H_{\rib}$ in $U_2'$ such that $u\in U_3$ if and only if $\cO_{C_u}(1)\cong \omega_{C_u}$.

    Since every curve $C_u$ in $U_2$ is  non-degenerate, this implies that $u\in U_3$ if and only if $C_u\hookrightarrow \bP^{g-1}$ is given by the canonical embedding. Thus we verified the condition of $U_3$.

    Next, we show that there exists an open neighborhood $U_5$ of $H_{\rib}$ in $U_3$ such that $u\in U_5\setminus H_{\rib}$ if and only if $C_u$ is reduced and $[C_u]$ is Chow semistable. Let $U_4$ be the open subscheme of $U_3$ parameterizing curves $C$ whose corresponding $1$-cycle is Chow semistable. Then $U_4$ contains $H_{\rib}$. By the openness of reducedness we know that the locus $Z_4$ parameterizing non-reduced curves in $U_4$ is closed in $U_4$. We claim that $H_{\rib}$ is a connected component of $Z_4$. By Lemmas \ref{lem:sm-chow} and \ref{lem:ribbon-cycle} we know that $H_{\rib}$ is closed in $Z_4$. By Theorem \ref{thm:ribbon-ext} and  Proposition \ref{prop:ribbon-canonical-map} we know that $H_{\rib}\to Z_{\rib}$ is a Severi--Brauer scheme with geometric fibers isomorphic to $\bP(\Ext^1(\Omega_{\bP^1}, \cO_{\bP^1}(-g-1)))$. Thus $H_{\rib}$ is connected by the connectedness of $Z_{\rib}$. Assume that there is a flat family of non-reduced curves $\{C_t\}_{t\in T\setminus o}$ in $U_4$ degenerating to $C_o$ in $H_{\rib}$ where $o\in T$ is a smooth pointed curve. Then, the non-reduced $1$-cycles $\Phi_{g,1}(C_t)$ degenerate to the doubled rational normal curve. Thus after shrinking $T$ we know that $\Phi_{g,1}(C_t)$ is the sum of at most two integral curves, which implies that $\Phi_{g,1}(C_t) = 2[(C_t)_{\red}]$ and $(C_t)_{\red}$ is integral. Let $C_o'$ be the flat degeneration of $(C_t)_{\red}$ as $t\to o$. Then we have 
    \[
    \chi(\cO_{C_o'}) = \chi(\cO_{(C_t)_{\red}}) = 1- p_a(\cO_{(C_t)_{\red}})\leq 1.
    \]
    On the other hand, it is clear that $C_o'$ is generically reduced with $(C_o')_{\red} = (C_o)_{\red}\cong \bP^1$. Thus we have 
    \[
     \chi(\cO_{C_o'}) \geq \chi(\cO_{\bP^1})  = 1.
    \]
    Combining the above two inequalities we conclude that $p_a(\cO_{(C_t)_{\red}}) = 0$ which implies that $(C_t)_{\red}$ is a rational normal curve in $\bP^{g-1}$. Thus Lemma \ref{lem:ribbon-cycle} implies that $C_t$ is a canonical ribbon. Thus the claim is proved, and we take $U_5 := U_4 \setminus (Z_4\setminus H_{\rib})$. 

    To summarize, we have found an open subscheme $U_5\subset \Hilb_{g,1}$ containing $H_{\rib}$ such that $[C\hookrightarrow \bP^{g-1}]\in U_5$ if and only if the following conditions hold.
    \begin{itemize}
        \item $C$ has lci singularities;
        \item $\omega_C$ is very ample and induces the closed embedding $C\hookrightarrow \bP^{g-1}$;
        \item the corresponding Chow $1$-cycle $[C]$ is Chow semistable;
        \item  $C$ is either a canonical ribbon (i.e. $C\in H_{\rib}$) or reduced.
    \end{itemize}

    Next, we show that $U_5$ is smooth.
    Indeed, the above arguments actually imply that  $U_5$ is open in $\boldsymbol{H}_{\chi_g}(\bP^{g-1})$. Thus standard arguments of moduli stacks being represented by a quotient stack of subschemes of Hilbert schemes implies that $[U_5/\PGL(g)]\to \Curves_{g,1}\cap \Curves_{g}^{\lci}$ is an open immersion such that every curve $[C]\in U_5\setminus H_{\rib}$ satisfies $[C]\in \Curves_{g}^{\lci +}$. Thus the smoothness of $U_5$ follows from the smoothness of $\Curves_{g}^{\lci +}$ and the unobstructedness of deformations of ribbons (Corollary \ref{cor:hyp-rib-unobstructed}).

    Finally, we define $\tU_{g,1}:= U_5$. Then it remains to verify the properties (1)-(4) from the statement. (1) is automatic from the construction. For (2), it suffices to show that $U_{g,1}$ is open in $\Chow_{g,1}$ as clearly it is contained in $\Chow_{g,1}^{\rm ss}$ by construction. 
    Since $\tU_{g,1}$ is smooth, we identify it with its seminormalization. By the properness of $\Phi_{g,1}: \Hilb_{g,1}^{\rm sn}\to \Chow_{g,1}$, we know that $\Phi_{g,1}(\Hilb_{g,1}^{\rm sn}\setminus \tU_{g,1})$ is closed in $\Chow_{g,1}$. Thus it suffices to show that $\Phi_{g,1}^{-1}(U_{g,1}) = \tU_{g,1}$. Suppose $[C'\hookrightarrow\bP^{g-1}]\in \Phi_{g,1}^{-1}(U_{g,1})$, i.e. there exists a curve $C\in \tU_{g,1}$ such that $[C']=[C]$ as $1$-cycles. If $C$ is reduced, then we know that $C'_{\red} = C$ which implies that $C'=C$ as they have the same Hilbert polynomial. If $C$ is non-reduced, then from the definition of $U_5$ we know that $C$ is a canonical ribbon. Thus by Lemma \ref{lem:ribbon-cycle} we know that $C'$ is also a canonical ribbon which implies that $C'\in \tU_{g,1}$. This verifies (2). 

    The first part of (3) is automatic from the construction. For the second part of (3), notice that $\Curves_{g}^{\lci +}$ contains $\cM_g$ as a dense open scheme. Hence any curve $C$ satisfying the assumptions is smoothable which implies that the canonical image of $C$ belongs to $\Hilb_{g,1}$. Then (3) follows from the construction. The first part of (4) is clear from the properness of $\Phi_{g,1}$. The second part of (4) follows from the fact that $\tU_{g,1}\setminus H_{\rib}$ only contains reduced curves. 
    For the last part of (4),
    since $H_{\rib}\to Z_{\rib}$ has connected fibers, by Zariski's main theorem, seminormality of $U_{g,1}$, and smoothness of $\tU_{g,1}$, we conclude that $U_{g,1}$ is normal. This also implies that $\tU_{g,1}\setminus H_{\rib}\to U_{g,1}\setminus Z_{\rib}$ is an isomorphism. The proof is finished.
\end{proof}

\subsection{A specific open substack of $\sX_g$}

In this subsection, we introduce an open substack $\sX_g^{\circ}$ of $\sX_g$. Later in the next subsection we will show that $\sX_g^{\circ}$ admits a good moduli space.

\begin{defn-prop}\label{dp:Xgcirc}

    There exists an open substack $\sX_g^{\circ}\hookrightarrow\sX_g$ parameterizing  curves $C$ of genus $g$ satisfying the following conditions.
    \begin{enumerate}
        \item $C$ is smoothable and has lci singularities;
        \item $\omega_C$ is ample and basepoint free and induces the canonical map $\varphi: C\to \bP^{g-1}$ with $\varphi_*[C]$ Chow semistable;
        \item the Chow polystable degeneration of the $1$-cycle $\varphi_*[C]$ is either a doubled rational normal curve or a reduced $1$-cycle $[C_0]$ as the cycle-theoretic canonical image of a curve $C_0\in \Curves_g^{\lci +}$ with $\omega_{C_0}$ very ample.
    \end{enumerate}

\end{defn-prop}

\begin{proof}
Recall that $\phi_g^{\rm c}: \sX_g^{\rm c} = [\Chow_{g,1}^{\rm ss}/\PGL(g)]\to\fX_g^{\rm c}$ is a good moduli space morphism. By universal closedness of good moduli space morphisms \cite[Theorem 4.16(ii)]{alper}, we know that $\phi_g^{\rm c}([(\Chow_{g,1}^{\rm ss}\setminus U_{g,1})/\PGL(g)])$ is closed in $\fX_g^{\rm c}$, where $U_{g,1}$ is obtained from Proposition \ref{prop:hilb-chow}. We define 
\[
\sX_g^{{\rm c}, \circ}:= (\phi_g^{\rm c})^{-1}(\fX_g^{\rm c}\setminus \phi_g^{\rm c}([(\Chow_{g,1}^{\rm ss}\setminus U_{g,1})/\PGL(g)])).
\]
We also let $V_{g,1}$ be the open subscheme of $\Chow_{g,1}^{\rm ss}$ such that $\sX_g^{{\rm c}, \circ} = [V_{g,1}/\PGL(g)]$. 
The above definition is equivalent to saying that a Chow semistable $1$-cycle belongs to $V_{g,1}$ if and only if its Chow polystable  degeneration belongs to $U_{g,1}$. 
Then from the definition we know that $\sX_g^{{\rm c}, \circ}$ is the largest saturated open substack of $\sX_g^{\rm c}$ contained in $[U_{g,1}/\PGL(g)]$. To conclude our definition, let $\sX_g^{\circ}\hookrightarrow \sX_g$ be the open substack such that 
\[
\sX_g^{\circ, {\rm sn}}:= (\Phi_g^{\rm sn})^{-1}(\sX_g^{{\rm c}, \circ}).
\]

Next, we verify that $\sX_g^{\circ}$ satisfies the statement.
We start from showing that every curve $C$ satisfying conditions (1)(2)(3) belongs to $\sX_g^{\circ}$. Since (1) and (2) guarantee that $C\in \sX_g$, it suffices to show that the Chow $1$-cycle $\varphi_*[C]\in \sX_g^{{\rm c}, \circ}$. This is equivalent to showing that the Chow polystable degeneration of $\varphi_*[C]$ belongs to $U_{g,1}$, which follows directly from condition (3) and Proposition \ref{prop:hilb-chow}(3).

Finally, we show that every curve in $\sX_g^{\circ}$ satisfies conditions (1)(2)(3). Note that every curve in $\sX_g$ is smoothable and satisfies (2) by definition. Thus it suffices to show that every curve  $C$ in $\sX_g^{\circ}$ has lci singularities and satisfies (3).  If  $\varphi_* [C]$ is a doubled rational normal curve (which implies (3)), then by Lemma \ref{lem:2RNC-hyp-rib} we know that $C$ is either a hyperelliptic curve or a ribbon. Thus $C$ has lci singularities. 
If $\varphi_*[C]$ is not a doubled rational normal curve, then we claim that $\varphi$ is a closed embedding. Given the claim, we have $\varphi(C) \in \Hilb_{g,1}$ whose corresponding $1$-cycle $\varphi_*[C]\in V_{g,1}\subset U_{g,1}$. Then by Proposition \ref{prop:hilb-chow} we know that $[C]\in \Curves_g^{\lci +}$. From the definition of $\sX_g^{{\rm c}, \circ}$ and $V_{g,1}$, we know that the Chow polystable degeneration $[C_0]$ of $\varphi_*[C]$ belongs to $U_{g,1}$ as well. Then by Proposition \ref{prop:hilb-chow}(3) we know that either $[C_0]$ is a doubled rational normal curve or  $C_0\hookrightarrow \bP^{g-1}$ is the canonical embedding of a  curve $C_0\in \Curves_{g}^{\lci +}$ with $\omega_{C_0}$ very ample. Thus (3) is verified.

It remains to verify the claim. 
By Proposition \ref{prop:hilb-chow}(3) we know that $\varphi_*[C]\in U_{g,1}\setminus Z_{\rib}$ is reduced. Denote by $\tV_{g,1}:=\tU_{g,1}\times_{U_{g,1}} V_{g,1}$. By surjectivity of $\tU_{g,1}\to U_{g,1}$ from Proposition \ref{prop:hilb-chow}(4) there exists a curve $C'$ in $\tV_{g,1}$ such that $[C']=\varphi_*[C]$ as $1$-cycles. Thus both $C$ and $C'$ are reduced, and $\varphi: C\to C'$ is finite and birational.  As a result,  we have $\cO_{C'}\subset \varphi_* \cO_C$ whose cokernel $\cG$ has finite support, and hence 
\[
1-g =\chi(\cO_{C'}) =  \chi (\varphi_* \cO_C) - h^0(\cG) = \chi(\cO_C) - h^0(\cG) = 1-g-h^0(\cG). 
\]
Thus $h^0(\cG) =0$, i.e. $\cG=0$, which implies that $\varphi:C\to C'$ is an isomorphism. Thus $\varphi$ is a closed embedding and the claim is proved.
\end{proof}

The following lemma was used in the above proof.

\begin{lem}\label{lem:2RNC-hyp-rib}
Let $C\in \sX_g(\bk)$ be a curve. Let $\varphi: C \to \bP^{g-1}$ be the finite morphism induced by the basepoint free linear system $|\omega_C|$. Then the $1$-cycle $\varphi_*[C]$ is a doubled rational normal curve in $\bP^{g-1}$ if and only if $C$ is  either a hyperelliptic curve or a ribbon.
\end{lem}

\begin{proof}
The backward direction is easy, see Propositions \ref{prop:hyp-canonical-map} and \ref{prop:ribbon-canonical-map}. We shall focus on the forward direction. 
We first treat the case where $C$ is reduced. Denote by $2[C_0]$ the $1$-cycle of doubled rational normal curve.
Then $\varphi$ induces a finite morphism $\varphi_{\red}: C\to C'_{\red} =C_0$. Moreover, $\varphi_*[C] = 2[C_0]$ implies that $(\varphi_{\red})_* \cO_C$ has rank $2$ on $C_0$. Thus we know that $C$ is hyperelliptic by Definition \ref{def:hyperelliptic}.

Next we assume that $C$ is non-reduced. Since $C$ is Gorenstein and hence Cohen--Macaulay, it has no embedded points. Thus $C$ has a generically non-reduced irreducible component. On the other hand, $\varphi_* [C] = 2[C_0]$ implies that $C$ is irreducible, has multiplicity $2$ at its generic point, and $\varphi_{\red} : C_{\red} \to C_0$ is birational. Thus we have $C_{\red} \cong C_0\cong \bP^1$. We will show that $C$ is a ribbon. Let $\cI$ be the nilradical sheaf of $\cO_C$. Since $C$ has multiplicity $2$ at its generic point, we know that $\cI$ has rank $1$ as an $\cO_{C_{\red}}$-module and that $\cI^2$ has finite support. Thus $h^1(C, \cI^2) = 0$ and from the standard long exact sequence we have 
\[
1 = h^0(C,\cO_C) = h^0(C,\cI^2) + h^0(C,\cO_C/\cI^2) \geq h^0(C,\cI^2) + 1,
\]
which implies $\cI^2 = 0$. Moreover, since $H^0(C, \cO_C) \cong H^0(C_{\red}, \cO_{C_{\red}})\cong \bk$, we know that the map $H^0(C, \cO_C) \to H^0(C_{\red}, \cO_{C_{\red}})$ is an isomorphism. This implies that $H^0(C, \cI) = 0$, in particular $\cI$ is torsion free as an $\cO_{C_{\red}}$-module. Therefore,  we conclude that $\cI$ is locally free of rank $1$ as an  $\cO_{C_{\red}}$-module which implies that $C$ is a ribbon.
\end{proof}

\begin{thm}\label{thm:hilb-chow}
The stack $\sX_g^{\circ}$ is smooth and  contains $\sH_g\cup \sR_g$. Moreover,  the following hold. 
\begin{enumerate}
    \item The Hilbert--Chow morphism $\Phi_g^{\circ}: \sX_g^{\circ} \to \sX_g^{\rm c}$ has an open image $\sX_g^{\ccirc}$.
    \item The open substack $\sX_g^{\ccirc}\hookrightarrow \sX_g^{\rm c}$ is saturated and normal.
    \item $\Phi_g^{\circ}:\sX_g^{\circ}\to \sX_g^{\ccirc}$ is isomorphic away from $\sH_g\cup \sR_g$, and $\Phi_g^{\circ}(\sH_g\cup \sR_g) $ is the point representing a doubled rational normal curve.
\end{enumerate}
\end{thm}

\begin{proof}

We first show that $\sX_g^{\circ}$ is smooth. Let $[C]\in \sX_g^{\circ}(\bk)$ be a curve. Since $\sX_g$ is an open substack of an irreducible component of $\Curves_{g,1}$, it suffices to show that $\Curves_{g,1}$ is smooth at $[C]$.  Let $\varphi: C\to \bP^{g-1}$ be the finite morphism induced by the basepoint free linear system $|\omega_C|$. If  $\varphi_* [C]$ is a doubled rational normal curve, then by Lemma \ref{lem:2RNC-hyp-rib} we know that $C$ is either a hyperelliptic curve or a ribbon. By Corollary \ref{cor:hyp-rib-unobstructed} we know that the stack $\Curves_{g,1}$ is smooth at $[C]$. If $\varphi_*[C]$ is not a doubled rational normal curve, then from the proof of Definition-Proposition \ref{dp:Xgcirc} we have $[C]\in \Curves_g^{\lci +}$ which implies the smoothness of  $\Curves_{g,1}$ at $[C]$.

Next, we show (1). Since $\sX_g^{\circ}$ is smooth, it can be identified with its seminormalization $\sX_g^{\circ, \rm sn}$ where $\Phi_g^{\rm sn}$ restricts and gives $\Phi_g^{\circ}: \sX_g^{\circ} \to \sX_g^{\rm c, \circ}$. It suffices to show that $\Phi_g^{\circ}$ is surjective onto $\sX_g^{\rm c, \circ}$.  Let $z\in V_{g,1}$ be a $1$-cycle. If $z$ is non-reduced, then by Proposition \ref{prop:hilb-chow}(3) we know that $z$ is a doubled rational normal curve. Thus there exists a canonical ribbon $C$ whose associated $1$-cycle is $z$. Then it follows from Definition-Proposition \ref{dp:Xgcirc} that $C\in \sX_g^{\circ}$. If $z$ is reduced, by Proposition \ref{prop:hilb-chow}(3) we know that there exists a curve $C\in \Curves_{g}^{\lci +}$ with $\omega_C$ very ample such that $z$ is the cycle-theoretic canonical image of $C$. Then we know that $C\in \sX_g$ whose Hilbert--Chow image is $z$, which implies that $C\in \sX_g^{\circ}$. Thus (1) is proved.

Since normality of $\sX_g^{\rm c, \circ}$ follows from normality of $U_{g,1}$ from Proposition \ref{prop:hilb-chow}, and saturatedness has been shown earlier in the proof of Definition-Proposition \ref{dp:Xgcirc}, (2) is proved. 

Finally, we show (3). By Lemma \ref{lem:2RNC-hyp-rib} we know that $\sH_g\cup \sR_g$ is contained in $\sX_g^{\circ}$ and $\Phi_g^{\circ}(\sH_g\cup \sR_g) = \{z_{\rib}\}$, where $z_{\rib}$ represents a doubled rational normal curve. It suffices to show that $\Phi_g^{\circ}$ induces an isomorphism between $\sX_g^{\circ}\setminus (\sH_g\cup \sR_g)$ and $\sX_g^{\rm c,\circ}\setminus \{z_{\rib}\}$. By Proposition \ref{prop:hilb-chow}(4) we have
\[
\sX_g^{\rm c,\circ}\setminus \{z_{\rib}\} = [(V_{g,1}\setminus Z_{\rib})/\PGL_g] \cong [(\tV_{g,1}\setminus H_{\rib})/\PGL_g].
\]
Thus it suffices to show that $\sX_g^{\circ}\setminus (\sH_g\cup \sR_g)$ is isomorphic to $[(\tV_{g,1}\setminus H_{\rib})/\PGL_g]$. From the proof of Definition-Proposition \ref{dp:Xgcirc} we know that there is an open immersion $\iota: [\tV_{g,1}/\PGL_g] \hookrightarrow \sX_g^{\circ}$. Since $\tV_{g,1}\subset \tU_{g,1}$, by Proposition \ref{prop:hilb-chow}(3) we know that every curve $C$ in $\tV_{g,1}\setminus  H_{\rib}$ satisfies that $C$ is reduced with very ample $\omega_C$. Thus by Lemma \ref{lem:2RNC-hyp-rib} we have $\iota([(\tV_{g,1}\setminus H_{\rib})/\PGL_g])\subset \sX_g^{\circ}\setminus (\sH_g\cup \sR_g)$. Conversely, let $[C]\in \sX_g^{\circ}\setminus (\sH_g\cup \sR_g)$ be a point. Then $\varphi_*[C]$ is not a doubled rational normal curve by Lemma \ref{lem:2RNC-hyp-rib}. From the claim in the proof of Definition-Proposition \ref{dp:Xgcirc}, we conclude that $\varphi$ is a closed embedding and $\varphi(C)\in \tV_{g,1}$. Clearly $C$ is not a canonical ribbon, which implies that $\varphi(C)\in \tV_{g,1}\setminus H_{\rib}$.  The proof is finished.
\end{proof}

\subsection{Existence of good moduli spaces}

Let $\sX_g^{\circ}$ be a smooth open substack of $\sX_g$ from Definition-Proposition \ref{dp:Xgcirc} whose Hilbert--Chow image is denoted by $\sX_g^{\rm c, \circ}$. 
Then $\sX_g^{\rm c, \circ} = [V_{g,1}/\PGL(g)]$ admits a good moduli space $\fX_g^{\rm c, \circ}= V_{g,1}\sslash \PGL(g)$ as an open subscheme of $\fX_g^{\rm c}=\Chow_{g,1}^{\rm ss}\sslash \PGL(g)$ by saturatedness, where $V_{g,1}$ is the $\PGL(g)$-equivariant open subscheme of $\Chow_{g,1}^{\rm ss}$ from the proof of Definition-Proposition \ref{dp:Xgcirc}. Denote by $\phi_g^{\rm c, \circ}: \sX_g^{\rm c, \circ}\to \fX_g^{\rm c, \circ}$ the good moduli space morphism as the restriction of $\phi_g^{\rm c}: \sX_g^{\rm c} \to \fX_g^{\rm c}$. Let $\phi_g^{\circ}= \phi_g^{\rm c, \circ}\circ \Phi_g^{\circ}: \sX_g^{\circ}\to \fX_g^{\rm c, \circ}$ be the composition.

Our goal in this subsection is to prove the following result.

\begin{thm}\label{thm:sX_g-gms}
Notation as above. Then $\phi_g^{\circ}: \sX_g^{\circ}\to \fX_g^{\rm c, \circ}$ is a good moduli space morphism.
\end{thm}

Let $\pi:\cC_{\sX_g^{\circ}}\to \sX_g^{\circ}$ be the universal family of curves. We have a vector bundle $\pi_* \omega_{\cC_{\sX_g^{\circ}}/\sX_g^{\circ}}$ of rank $g$ on $\sX_g^{\circ}$. Let $\sY_g\to \sX_g^{\circ}$ be the $\PGL(g)$-torsor associated to $\pi_*\omega_{\cC_{\sX_g^{\circ}}/\sX_g^{\circ}}$. Then we have a Hilbert--Chow morphism $\tPhi_g: \sY_g \to V_{g,1}$ defined by $\tPhi(C, [s]) = (\varphi_s)_*[C]$, where $s=(s_1, \cdots, s_g)$ is a basis of $H^0(C,\omega_C)$, and $\varphi_s: C\to \bP^{g-1}$ is the finite morphism induced by $s$. Clearly $\tPhi_g$ induces the morphism $\Phi_g^{\circ}: \sX_g^{\circ} \to \sX_g^{\rm c, \circ} = [V_{g,1}/\PGL(g)]$ by taking the $\PGL(g)$-quotient. Let $\tsH_g:=\sH_g \times_{\sX_g^{\circ}} \sY_g$ and $\tsR_g:=\sR_g \times_{\sX_g^{\circ}} \sY_g$  denote the $\PGL(g)$-torsors over the closed substacks  $\sH_g$ and $\sR_g$ of $\sX_g^{\circ}$ respectively by Proposition \ref{prop:hyp-rib-closed-substack}. Recall that  $Z_{\rib}$ is  the closed $\PGL(g)$-orbit in $V_{g,1}$ parameterizing doubled rational normal curves.

\begin{lem}\label{lem:gms-ribbon}
Notation as above. Then $\tPhi_{\rib}:=\tPhi_g|_{\tsH_g\cup \tsR_g}: \tsH_g\cup \tsR_g \to Z_{\rib}$ is a good moduli space morphism.
\end{lem}

\begin{proof}
We shall use the quotient presentations of stacks $\sH_g$ and $\sR_g$ from Theorems \ref{thm:hyp-stack} and \ref{thm:rib-stack}:
\[
\sH_g \cong [\bA(H^0(\bP^1, L^{\otimes -2}) )/ G] \quad \textrm{and}\quad \sR_g \cong [\bA(\Ext^1(\Omega_{\bP^1}, L))/ G],
\]
where $G := \GL(2)/\bmu_{g+1}$ and $L := \cO_{\bP^1}(-g-1)$. By Proposition \ref{prop:hyp-rib-intersection} we know that $\sH_g$ intersects $\sR_g$ at a unique point $[R_{\hyp}]$ which is the origin in both quotient presentations, where the intersection is transversal. Thus we have that 
\[
\sH_g\cup \sR_g \cong [(\bA(H^0(\bP^1, L^{\otimes -2}))\cup \bA(\Ext^1(\Omega_{\bP^1}, L)))/G],
\]
where the union of the two affine spaces is transversal at the origin and hence an affine scheme. As a result, we know that $\sH_g \cup \sR_g \to \rmB G$ is an affine morphism. Composing with the natural morphism $\rmB G \to \rmB \PGL(2)$ where $G\to \PGL(2)$ is the quotient group homomorphism, we see that $\Phi_{\rib}: \sH_g \cup \sR_g \to \rm B \PGL(2)$ is cohomologically affine. Moreover, $\Phi_{\rib}$ is the descent of the good moduli space morphism 
\[
[(\bA(H^0(\bP^1, L^{\otimes -2}))\cup \bA(\Ext^1(\Omega_{\bP^1}, L)))/(\bG_m/\bmu_{g+1})]\to \Spec \bk
\]
under a $\PGL(2)$-quotient. 
Thus we conclude that $\Phi_{\rib}: \sH_g \cup \sR_g \to \rm B \PGL(2)$ is a good moduli space morphism. By taking $\PGL(g)$-torsors of $\Phi_{\rib}$, we conclude that $\tPhi_{\rib}$ is also a good moduli space morphism.
\end{proof}

\begin{prop}\label{prop:gms-PGLtorsor}
Notation as above. Then $\tPhi_g: \sY_g \to V_{g,1}$ is a good moduli space morphism.
\end{prop}

\begin{proof}
We first show that $\sY_g$ admits a good moduli space using Theorem \ref{thm:AFS-gms}.

We start from checking condition  (1) of Theorem \ref{thm:AFS-gms}. Let $y\in \sY_g(\bk)$ be a closed point. If $y\not\in \sH_g\cup \sR_g$, then by the last paragraph of the proof of Theorem \ref{thm:hilb-chow} there is a stack isomorphism $\sX_g^{\circ}\setminus (\sH_g \cup \sR_g) \cong [(\tV_{g,1}\setminus H_{\rib})/\PGL(g)]$. Thus by taking $\PGL(g)$-torsors we have an isomorphism $\sY_g \setminus (\tsH_g\cup \tsR_g) \cong \tV_{g,1}\setminus H_{\rib}$. Moreover, by Proposition \ref{prop:hilb-chow}(4) we know that $\tPhi_g' = \tPhi_g|_{\sY_g \setminus (\tsH_g \cup \tsR_g)}: \sY_g \setminus (\tsH_g\cup \tsR_g) \to V_{g,1}\setminus Z_{\rib}$ is an isomorphism, which implies that $\sY_g \setminus (\tsH_g\cup \tsR_g)$ is represented by a scheme. Since $[(V_{g,1}\setminus Z_{\rib})/\PGL(g)]$ is saturated open in $\sX_g^{\rm c, \circ}$, every $\bk$-point in $\sY_g \setminus (\tsH_g\cup \tsR_g)$ is closed. Thus (1) holds at $y\not\in \sH_g\cup \sR_g$. 

If $y\in \tsH_g \cup \tsR_g$ is a closed point, then by Proposition \ref{prop:hyp-rib-intersection} we know that $y$ parameterizes the hyperelliptic ribbon $R_{\hyp}$ with a projective basis $[s] = [s_1, \cdots , s_g]$ of $H^0(R_{\hyp}, \omega_{R_{\hyp}})$. By Proposition \ref{prop:hyp-rib-closedpt} we have $\Aut(R_{\hyp}) \cong  \GL(2)/\bmu_{g+1}.$
Thus we have 
\[
\Stab_y \cong \Aut(R_{\hyp}, [s]) \cong \bG_m /\bmu_{g+1}\cong \bG_m. 
\]
By Propositions \ref{prop:hyp-rib-transversal} and \ref{prop:hyp-rib-intersection} we have
\[
T_{\sY_g, y}\cong H^0(\bP^1, L^{\otimes -2}) \oplus \Ext^1(\Omega_{\bP^1}, L) \oplus T_{\rm id} \PGL(g+1)/\PGL(2).
\]
Here $L = \cO_{\bP^1} (-g-1)$. Hence $\Stab_y \cong \bG_m$ acts on $T_{\sY_g, y}$ with weight $(2, -1, 0)$ in the above direct sum presentation. Thus we see that the locus of vectors in $T_{\sY, y}$ with non-trivial stabilizers are precisely the union of $(T_{\rm id} \PGL(g+1)/\PGL(2)) \setminus \{0\}$ (with stabilizer $\bG_m$) and $H^0(\bP^1, L^{\otimes -2})\setminus \{0\}$ (with stabilizer $\bmu_2$). By Theorem \ref{thm:ahr}, there exists a local quotient presentation $f: [\Spec(A)/\bG_m] \to \sY_g$ around $y$. Denote by $[V(I)/\bG_m]$ the preimage of $\tsH_g\cup \tsR_g$ under $f$. Then from the local description we know that the stabilizer is trivial in $\Spec(A) \setminus V(I)$ after possibly shrinking $\Spec(A)$. Moreover, since $\tPhi_{\rib}$ is a good moduli space morphism by Lemma \ref{lem:gms-ribbon}, using \cite[Proposition 4.3]{AHLH18} we know that after possibly shrinking $\Spec(A)$, the restriction  $[V(I)/\bG_m]\to \tsH_g\cup \tsR_g$ of $f$ satisfies condition (1). Thus condition (1) holds at closed points $y\in \tsH_g\cup \tsR_g$.

Next, we check condition (2) of Theorem \ref{thm:AFS-gms}. If $y\in \sY_g\setminus (\tsH_g\cup \tsR_g)$, then $y$ is a closed point with trivial stabilizer, thus (2) holds. If $y\in \tsH_g\cup \tsR_g$, then by Lemma \ref{lem:gms-ribbon} and \cite[Remark 2.11]{ABHLX19} we know that (2) holds. Thus by Theorem \ref{thm:AFS-gms} we conclude that $\sY_g$ admits a good moduli space $\fY_g$.

Finally we show that $\fY_g \cong V_{g,1}$. By universality of good moduli spaces (see Theorem \ref{thm:gms-universal}), we have a morphism $g:\fY_g \to V_{g,1}$. From the above arguments and Lemma \ref{lem:gms-ribbon}, we know that $g$ is bijective and $\fY_g$ is normal. Thus $g$ is separated. By Zariski's main theorem of algebraic spaces \cite[\href{https://stacks.math.columbia.edu/tag/082K}{Tag 082K}]{stacksproject}, we conclude that $g$ is an isomorphism. 
\end{proof}

\begin{proof}[Proof of Theorem \ref{thm:sX_g-gms}]
By Proposition \ref{prop:gms-PGLtorsor}, we have that $\tPhi_g: \sY_g\to V_{g,1}$ is cohomologically affine, and $(\tPhi_g)_* \cO_{\sY_g} = \cO_{V_{g,1}}$. This implies that $\Phi_g^{\circ}:\sX_g^{\circ}\to \sX_g^{\rm c, \circ}$ is cohomologically affine and $(\Phi_g^{\circ})_* \cO_{\sX_g^{\circ}} = \cO_{\sX_g^{\rm c, \circ}}$ as $\tPhi_g$ is a $\PGL(g)$-torsor over $\Phi_g^{\circ}$. Thus $\Phi_g^{\circ}$ is a good moduli space morphism. Since $\phi_g^{\rm c, \circ}: \sX_g^{\rm c, \circ} \to \fX_g^{\rm c,\circ}$ is a good moduli space morphism, we conclude that the composition $\phi_g^{\circ}: \sX_g^{\circ} \to \fX_g^{\rm c, \circ}$ is a good moduli space morphism. This proves the theorem.
\end{proof}

\section{The hyperelliptic flip}\label{sec:hyp-flip}

\subsection{Good moduli spaces for genus $4$ curves with Chow semistable cycles}

From now on, we shall focus on the case of genus $g=4$. For simplicity, we omit the subscript $4$ from $\sX_4$, $\sX_4^{\rm c}$, $\fX_4^{\rm c}$, etc.~to denote the corresponding stack or space by $\sX$, $\sX^{\rm c}$, $\fX^{\rm c}$, etc. In particular, $\sX$ parameterizes Gorenstein smoothable curves of genus $4$ with ample and basepoint free canonical bundle whose cycle-theoretic canonical image is Chow semistable.

The following is the main result of this subsection which strengthens Theorem \ref{thm:sX_g-gms} in genus $4$. 

\begin{thm}\label{thm:sX-gms}
The stack $\sX$ is  a smooth irreducible algebraic stack of finite type over $\bk$ with affine diagonal that contains $\cM_4$ as a dense open substack. We have an open immersion $\sX\hookrightarrow \Curves_{4}^{\lci}\cap \Curves_{4,1}$.
Moreover, $\sX$ admits a projective good moduli space isomorphic to $\fX^{\rm c}$.
\end{thm}

\begin{lem}\label{lem:Xcirc=X}
We have $\sX^{\circ} = \sX$ and $\sX^{{\rm c}, \circ} = \sX^{\rm c}$. 
\end{lem}

\begin{proof}
We only need to show the second equality as the first one follows by Definition-Proposition \ref{dp:Xgcirc}. Since $\sX^{{\rm c}, \circ}$ is the largest saturated open substack of $\sX^{\rm c}$ contained in $[U_{4,1}/\PGL(4)]$, it suffices to show that $U_{4,1} = \Chow_{4,1}^{\rm ss}$. By Theorem \ref{thm:CMJL-Chow-GIT} we know that every $1$-cycle in $\Chow_{g,1}^{\rm ss}$ is either a reduced  $(2,3)$-complete intersection curve in $\bP^3$ or a doubled rational normal curve. Thus it follows from Proposition \ref{prop:hilb-chow} that $U_{4,1} = \Chow_{4,1}^{\rm ss}$ and the proof is finished.
\end{proof}

\begin{proof}[Proof of Theorem \ref{thm:sX-gms}]
By Lemma \ref{lem:Xcirc=X} we can identify $\sX$ with $\sX^{\circ}$. Thus the first statement follows from 
Theorems \ref{thm:sX_g-smooth} and \ref{thm:hilb-chow}. The second statement follows from Definition-Proposition \ref{dp:Xgcirc}. The last statement follows from Theorem \ref{thm:sX_g-gms}.
\end{proof}

\subsection{The hyperelliptic flip in genus $4$}

In this subsection, we verify that $\sX$ contains $\ocM_4(\frac{5}{9}, \frac{23}{44}) $ as an open substack and we introduce 
an open substack $\sX^+$ of $\sX$ and show that it admits a proper good moduli space which replaces ribbons by hyperelliptic curves in genus $4$.  If we let $\sX^-\hookrightarrow \sX$ be a suitable open substack that provides a model for the Hassett--Keel wall crossing near $\Gamma$ (see Remark \ref{rmk:sX-andVGIT}), this will give the wall crossing $\sX^- \hookrightarrow \sX \hookleftarrow \sX^+$ that provides the flip over the locus $\Gamma$ of $C_{A,B}$ curves, illustrated in Figure \ref{fig:M59}.  

\begin{prop}\label{prop:59-ewall}
    There is an open immersion $\ocM_4(\frac{5}{9}, \frac{23}{44}) \hookrightarrow \sX$, where $\ocM_4(\frac{5}{9}, \frac{23}{44}) $ is the last VGIT quotient stack from Theorem \ref{thm:lastVGITstack}.  Consequently, there is a diagram of stacks and good moduli spaces
\[
\begin{tikzcd}
\ocM_4(\frac{5}{9}, \frac{23}{44}) \arrow[d] \arrow[r,hook]  &
\sX \arrow[d]   \\
\oM_4(\frac{5}{9}, \frac{23}{44}) \arrow[r]  & 
\fX^{\rm c} \cong \oM_4(\frac{5}{9}) 
\end{tikzcd}
\]
where the bottom arrow is a projective morphism.
\end{prop}

\begin{proof}
We use notation as in Section \ref{sec:chow-vgit}.  Denote by $\sX_t = \ocM_4(\frac{5}{9}, \frac{23}{44})$ the VGIT quotient stack from \cite{CMJL12, CMJL14} for $t \in (\frac{6}{11}, \frac{2}{3})$ (see also Theorem \ref{thm:cmjl-vgit}).  By Definition \ref{def:23VGIT}, $\sX_t = [W^{\rm{ss}}(N_t)/\PGL(4)]$, where $W^{\rm{ss}}(N(t)) \subset W$.  By Theorems \ref{thm:CMJL-Chow-GIT} and \ref{thm:lastVGITstack}, every point of $W^{\rm{ss}}(N(t))$ represents a reduced $(2,3)$ complete intersection curve lying on a uniquely determined quadric.  By \cite[Corollary I.6.6.1]{Kol96}, the Hilbert--Chow morphism is an isomorphism along this locus, and thus the induced morphism  $W^{\rm{ss}}(N(t)) \hookrightarrow W \to \Chow_{4,1}$ is birational and quasi-finite.  Furthermore, the locus $\Chow^{\rm{ss}}_{4,1} \subset \Chow_{4,1}$ is open, and by Theorem \ref{thm:lastVGITstack}, the image of $W^{\rm{ss}}(N(t)) \to \Chow_{4,1}$ is contained in $\Chow^{\rm{ss}}_{4,1}$.

Therefore, by Zariski's Main Theorem and Lemma \ref{lem:Chow-normal}, $W^{\rm{ss}}(N(t))\to \Chow^{\rm{ss}}_{4,1}$ is an open immersion.  
This yields an open immersion of stacks $\sX_t = [W^{\rm{ss}}(N(t))/\PGL(4)] \to \sX^{\rm{c}} =[ \Chow^{\rm{ss}}_{4,1}/\PGL(4)]$ such that the image does not contain the point representing a doubled rational normal curve.  Therefore, by Theorem \ref{thm:hilb-chow}(3) and Lemma \ref{lem:Xcirc=X}, there is an open immersion $\sX_t \to \sX$, as desired. 

The last sentence follows from Theorem \ref{thm:sX-gms} and universality of the good moduli space morphism $\ocM_4(\frac{5}{9}, \frac{23}{44}) \to \oM_4(\frac{5}{9}, \frac{23}{44})$ among maps to algebraic spaces (see Theorem \ref{thm:gms-universal}).
\end{proof}

\begin{defn}\label{def:sUsV}
Recall that $\Phi: \sX\to \sX^{\rm c}$ is the Hilbert--Chow morphism, and 
\[
\phi=\phi^{\rm c}\circ\Phi: \sX\to \fX^{\rm c} = \Chow_{4,1}^{\rm ss}\sslash \PGL(4)
\] is a good moduli space morphism by Theorem \ref{thm:sX-gms}. 
We denote by $\sU := \phi^{-1}(\fX^{\rm c} \setminus \{[C_{2A_5}], [C_D]\})$ and $\sV:=\phi^{-1} (\fX^{\rm c}\setminus \Gamma)$ two open substacks of $\sX$. Then we have $\sX = \sU \cup \sV$. Moreover, $\sU\cap \sV = \phi^{-1}(\fX^{\rm s}) =\Phi^{-1}(\sX^{\rm s})$ is isomorphic to the Chow stable stack $\sX^{\rm s}$ by Theorem \ref{thm:hilb-chow} and Lemma \ref{lem:Xcirc=X}. By abuse of notation, we simply write $\sU\cap \sV = \sX^{\rm s}$.
\end{defn}

\begin{defn}\label{def:sX+}
Let $\sX^+$ be the substack of $\sX$ parametrizing curves $C$ satisfying either of the following conditions.
\begin{enumerate}
    \item $C$ has $A_{\leq 4}$-singularities;

    \item $C$ belongs to $\sV$, i.e. there are no special degenerations of the $1$-cycle $\Phi(C)$  to  $[C_{A,B}]$ for any $(A,B) \in \bk^2\setminus \{(0,0)\}$.
\end{enumerate}
By deformation theory of $A_k$-singularities, we know that condition (1) defines an open substack of $\sX$. By Definition \ref{def:sUsV} we know that $\sV$ is an open substack $\sX$. Thus $\sX^+$ is an open substack of $\sX$. We denote by $\sU^+:=\sU\cap \sX^+$. 
\end{defn}

The main result of this subsection is the following.

\begin{thm}\label{thm:X+gms}
Notation as above. The stack $\sX^+$ admits a good moduli space $\fX^+$ as a proper algebraic space.
\end{thm}

In order to prove Theorem \ref{thm:X+gms}, we first focus on the open substack $\sU^{+}$ of $\sX^+$.

\begin{prop}\label{prop:U+open}
The stack $\sU^{+}$ precisely parameterizes curves $C$ in $\sU$ with at worst $A_4$-singularities. Moreover, $\sU^+$ is an open substack of $\ocM_4(\frac{2}{3}-\epsilon)$.
\end{prop}

\begin{proof}
 For the first statement, notice that a curve in $\sU$ with $A_{\leq 4}$-singularities belongs to $\sX^+$ and hence $\sU^+$ by Definition \ref{def:sX+}. Conversely, a curve $[C]$ in $\sU^+$ satisfies that either $C$ has $A_{\leq 4}$-singularities, or $[C]\in \sU \cap \sV $. Since $\sU\cap \sV = \sX^{\rm s}$ which only contains curves with $A_{\leq 4}$-singularities by Theorem \ref{thm:CMJL-Chow-GIT}, we know that $[C]$ always has $A_{\leq 4}$-singularities.  

 Next, we shall focus on the second statement.  From Theorem \ref{thm:sX-gms}, Definitions \ref{def:HKsing} and \ref{def:sX+} we know that both $\sU^+$ and $\ocM_4(\frac{2}{3}-\epsilon)$ are open substacks of $\Curves_4^{\lci}$. Thus it suffices to show that $\sU^+(\bk) \subset \ocM_4(\frac{2}{3}-\epsilon)(\bk)$. We have shown that every curve $C\in\sU^+(\bk)$ has $A_{\leq 4}$-singularities. Thus we only need to show that the three cases of elliptic tails, elliptic chains and Weierstrass chains do not appear in $\sU^+$. Since $\sU^+$ is open, it suffices to show that these three cases do not appear when each of the elliptic tail/elliptic chain/Weierstrass chain is smooth.
\begin{enumerate}
    \item elliptic tails $(E,p)$:
    \begin{itemize}
        \item $A_1$-attached. Then $\omega_C|_E \cong \omega_E(p)\cong \cO_E(p)$ is not base point free, a contradiction.
        \item $A_2$-attached. Then $\omega_C|_E \cong \omega_E(2p)\cong \cO_E(2p)$ is base point free but not very ample. Thus the canonical  image of $C$ is non-reduced and contains a line as the image of $E$. This is a contradiction.
        \item $A_4$-attached. Then $E$ is the normalization of $C$, and   $p_a(C)= p_a(E) + 2 = 3\neq 4$, a contradiction.
    \end{itemize}
    \item elliptic chains $(E, p_1, p_2)$:
    \begin{itemize}
        \item length is at least $3$. Then $p_a(C) \geq 3+2 =5 >4$, a contradiction.
        \item length is $2$. If $p_1$ or $p_2$ is $A_4$-attached, then $p_a(C) \geq 2+1+2 =5 >4$, a contradiction.
        \item length is $1$. If both $p_1$ and $p_2$ are $A_4$-attached, then $p_a(C)\geq 1+2+2 =5 >4$, a contradiction.
        If $p_1$ is $A_1$-attached and $p_2$ is $A_4$-attached, then $p_a$ of the image of $E$ in $C$ is $3$, which implies that the remaining curve is an $A_1$-attached elliptic tail, a contradiction.
        If both $p_1$ and $p_2$ are $A_1$-attached, then $\omega_C|_E\cong \cO_E(p_1+p_2)$. Thus the canonical  image of $C$ is non-reduced and contains a line as the image of $E$, a contradiction.
    \end{itemize}
    \item Weierstrass chain $(E, p)$:
    \begin{itemize}
        \item length is at least $2$. Let $(E',p')$ be the genus $2$ component with Weierstrass point. Then $\omega_C|_{E'}\cong \omega_{E'}(2p')\cong \pi^*\cO_{\bP^1}(2)$ where $\pi:E'\to\bP^1$ is the double cover map. Hence $\pi_* \omega_C|_{E'}\cong \cO_{\bP^1}(2)\oplus \cO_{\bP^1}(-1)$ which implies that $\omega_C|_{E'}$ is base point free but not very ample. Thus the canonical  image of $C$ is non-reduced and contains a conic as the image of $E'$, a contradiction.
        \item length is $1$. Then $\omega_C|_{E}\cong \omega_{E}(p)$ is not base point free, a contradiction.
    \end{itemize}
\end{enumerate}
\end{proof}

\begin{prop}\label{prop:U+closed}
The closed points $[C]$ of $\sU^+$ are one of the following.
\begin{enumerate}
    \item $C$ is a GIT polystable hyperelliptic curve;
    \item $C\in \sX\setminus \sH$ is a canonical curve on a normal quadric $Q$ with at worst $A_3$-singularities at a singularity of $Q$ and at worst $A_4$-singularities elsewhere.
\end{enumerate}
Moreover, every $\bk$-point of $\sU^+$ admits a $\bG_m$-equivariant degeneration to a closed point.
\end{prop}

\begin{proof}
We first prove that (1) and (2) are both closed points. 

For (1), let $\sH^+\hookrightarrow \sH$ be the open substack parameterizing GIT semistable hyperelliptic curves. From the GIT of binary forms we know that $\sH^+$ precisely parameterizes hyperelliptic curves with $A_{\leq 4}$-singularities.  Thus we have $\sH^+ = \sH \cap \sU^+$. By Proposition \ref{prop:hyp-rib-closed-substack} we know that $\sH^+$ is a closed substack of $\sU^+$. Thus every GIT polystable hyperelliptic curve being a closed point of $\sH^+$ by GIT is also a closed point of $\sU^+$. 

For (2), by Theorem \ref{thm:CMJL-Chow-GIT} we know that if $C$ is a canonical curve with at worst $A_2$-singularities at singularity of $Q$ and at worst $A_4$-singularities elsewhere, then $[C]\in \sX^{\rm s}$. Thus $C$ is a closed point of $\sX$ and hence of $\sU^+$. 

Next we assume that $C$ has an $A_3$-singularity at the singularity of $Q$, and at worst $A_4$-singularities elsewhere. By \cite[Proof of Theorem 3.1]{CMJL12} we know that $C$ admits a $\bG_m$-equivariant degeneration to $C_{A,B}$ with $4A\neq B^2$ which has an $A_3$-singularity at the singularity of $Q$ and an $A_5$-singularity in the smooth locus of $Q$ (and another $A_1$-singularity if $B=0$). Since $\Aut^0(C_{A,B})\cong \bG_m$, we know that this is the only other point in the closure of $[C]$ in $\sX^{\rm c}$. Since $[C_{A,B}]\notin \sU^+$, we know that $[C]$ is a closed point. 

Next we show that every closed point $[C]$ of $\sU^+$ is either (1) or (2). The hyperelliptic case follows from GIT. In the canonical case, we have that $Q$ is normal as otherwise $C$ degenerates to $C_D$ which implies $[C]\not\in \sU$, a contradiction. Assume to the contrary that $C$ has an $A_4$-singularity at the singularity of $Q$. Then Proposition \ref{prop:A4-cone} implies that $C$ admits a $\bG_m$-equivariant degeneration to the unique hyperelliptic curve with two $A_4$-singularities, i.e. $(x_2^2=x_0^5y_1^5)\subset\bP(1,1,5)_{[x_0,x_1,x_2]}$, which is a closed point of $\sU^+$ from case (1), a contradiction.

Finally, the last statement follows from the discussion above. 
\end{proof}

\begin{prop}\label{prop:U+closed-preserve}
Every closed point of $\sU^+$ is also closed in $\ocM_4(\frac{2}{3}-\epsilon)$.
\end{prop}

\begin{proof}
Let $[C]\in \sU^+$ be a closed point. If $C$ has at worst $A_3$-singularities, then by Definition \ref{def:HKsing} we know that $[C]\in \ocM_4(\frac{2}{3}+\epsilon)$. Thus Lemma \ref{lem:2/3+DM} implies that $[C]$ is a closed point of $\ocM_4(\frac{2}{3}+\epsilon)$. Since $[C]\in \ocM_4(\frac{2}{3}-\epsilon)$, it is also a closed point in $\ocM_4(\frac{2}{3}-\epsilon)$ by local VGIT (see \cite[Section 3]{AFSvdW} and \cite[Theorem 1.3]{AFS17}).

From now on, we may assume that $C$ has at least an $A_4$-singularity by Proposition \ref{prop:U+closed}. Let $[C_0]$ (resp. $[C_0']$) be the unique closed point of $\ocM_4(\frac{2}{3})$ (resp. $\ocM_4(\frac{2}{3}-\epsilon)$) in the closure of $[C]$. By \cite[Theorem 2.22]{AFSvdW}, we know that $C_0 = K \cup E_1\cup \cdots \cup E_r$, such that each $E_i$ is an $\alpha_{\frac{2}{3}}$-atom and $[K]$ is a closed point in the stack of $(\frac{2}{3}+\epsilon)$-stable (pointed) curves with no $A_1$-attached Weierstrass tails. Since every $\alpha_{\frac{2}{3}}$-atom has arithmetic genus $2$, we know that $r\leq 2$. Below we split to three cases based on the value of $r$. 

If $r=0$, then $[C_0]=[K]$ belongs to the stack $\ocM_{4}(\frac{2}{3}+\epsilon)$ which implies that $C_0$ has at worst $A_3$-singularities. This is a contradiction to the assumption that $C$ has at least an $A_4$-singularity. 

If $r=1$, then $C_0= K\cup E_1$ where $[(K,q)]$ is a closed point in $\ocM_{2,1}(\frac{2}{3}+\epsilon)$ with $q=K\cap E_1$. If moreover $\Aut(K,q)$ is infinite, then by \cite[Remark 2.28]{AFSvdW} we know that $K=K_0\cup R_1$ where $R_1$ is an $A_1/A_1$-attached rosary of length $3$. Then we must have $p_a(K_0)=0$ and $K_0$ is a one pointed curve, which is a contradiction to the ampleness of $\omega_K (q)|_{K_0}$. Thus $\Aut(K,q)$ must be finite which implies that $\Aut^0(C_0)\cong \bG_m$. Since $C_0$ admits an $A_1$-attached Weierstrass tail, we know that $C_0\not\in \ocM_{4}(\frac{2}{3}-\epsilon)$. Thus $C_0'$ admits a non-trivial $\bG_m$-equivariant degeneration to $C_0$ which implies that $\Aut(C_0')$ is finite. Hence $C\cong C_0'$ is closed in $\ocM_4(\frac{2}{3}-\epsilon)$.

If $r=2$, then $C_0= E_1\cup E_2$ and $\Aut(C_0)\cong \bG_m^2 \rtimes \bmu_2$. Similar to the $r=1$ case, we have $C_0\not\in \ocM_{4}(\frac{2}{3}-\epsilon)$. Thus we have $\dim\Aut(C_0')<\dim\Aut(C_0)=2$. If moreover $\Aut(C_0')$ is finite, then $C\cong C_0'$ is closed in $\ocM_4(\frac{2}{3}-\epsilon)$.  Hence we may assume $\Aut^0(C_0')\cong \bG_m$. Denote by $q=E_1\cap E_2$. By \cite[Lemma 3.18 and Proposition 3.20, Type C]{AFSvdW}, we have
\[
\rmT^1 (C_0) = \rmT^1(E_1) \oplus \rmT^1(E_2) \oplus \rmT^1(\widehat{\cO}_{C_0, q}) \cong \bk^5\oplus\bk^5 \oplus \bk.
\]
Denote by $\sigma_1, \sigma_2$ the two standard $1$-PS's of $\bG_m^2$, i.e. $\sigma_1(t) = (t,1)$ and  $\sigma_2(t) = (1,t)$. Then from loc. cit. we know that $\sigma_i$ acts on $\rmT^1(E_i)$ with weights $(1,-10,-8,-6,-4)$, on $\rmT^1(E_{3-i})$ trivially, and on $\rmT^1(\widehat{\cO}_{C_0, q})$ with weight $1$ for $i\in\{1,2\}$. In addition, the $\bmu_2$-action swaps $\rmT^1(E_1)$ and $\rmT^1(E_2)$ and acts trivially on $\rmT^1(\widehat{\cO}_{C_0, q})$.  Thus by analyzing the $\bG_m^2 \rtimes \bmu_2$ orbits, we see that $C_0'$ is either of the form $K_i\cup E_{3-i}$, or  an irreducible curve with two $A_4$-singularities. The former is not possible as it contains a $A_1$-attached Weierstrass tail which implies $[C_0']\not\in  \ocM_{4}(\frac{2}{3}-\epsilon)$. The latter is precisely the strictly GIT polystable hyperelliptic curve $(z^2 = x^5 y^5)\subset \bP(1,1,5)$ which is a closed point of $\sU^+$ by Proposition \ref{prop:U+closed}. Thus $C$ admits a $\bG_m$-equivariant degeneration to $C_0'$ in $\sU^+$ and  both points are closed. This implies $C\cong C_0'$ is closed in $\ocM_{4}(\frac{2}{3}-\epsilon)$. The proof is finished.
\end{proof}

\begin{remark}\label{rmk:M21-is-DM}
    In the $r = 1$ case of the previous proof, it is shown that any closed point of $\ocM_{2,1}(\frac{2}{3}+\epsilon)$ has finite stabilizer.  In particular, this implies that $\ocM_{2,1}(\frac{2}{3}+\epsilon)$ is a DM stack. 
\end{remark}

\begin{cor}\label{cor:U+gms}
$\sU^+$ is a saturated open substack of $\ocM_4(\frac{2}{3}-\epsilon)$. In particular, $\sU^+$ admits a separated good moduli space.
\end{cor}

\begin{proof}
It suffices to show that if $[C_1]\in \sU^+$ and $[C_2]\in \ocM_4(\frac{2}{3}-\epsilon)$ admit a common $\bG_m$-equivariant degeneration to a closed point $[C_0]\in \ocM_4(\frac{2}{3}-\epsilon)$, then $[C_2]\in \sU^+$. By Proposition \ref{prop:U+closed} we know that  $[C_1]$ admits a $\bG_m$-equivariant degeneration to a closed point $[C_0']\in \sU^+$. By Proposition \ref{prop:U+closed-preserve}, we know that $[C_0']$ is also closed in  $\ocM_4(\frac{2}{3}-\epsilon)$. Thus $[C_0]=[C_0']\in \sU^+$. By openness of $\sU^+$ from Proposition \ref{prop:U+open} we know that $[C_2]\in \sU^+$ as well. The last statement follows from \cite[Remark 6.2]{alper}. The proof is finished.
\end{proof}

\begin{proof}[Proof of Theorem \ref{thm:X+gms}]
We first show that $\sX^+$ admits a good moduli space.  By Theorem \ref{thm:CMJL-Chow-GIT}, a curve $[C]\in \sV$ either belongs to the Chow stable locus $\sX^{\rm s}$ (hence has at worst $A_4$-singularities), or degenerates to $[C_{2A_5}]$ or $[C_D]$. Thus $\sV$ is an open substack of $\sX^+$. By Corollary \ref{cor:U+gms} $\sU^+$ admits a separated good moduli space which we denote by $\fU^+$. By definition we know that $\sV\cong \Phi(\sV) = (\phi^{\rm c})^{-1}(\fX^{\rm c}\setminus \Gamma)$ admits a separated good moduli space $\fV\cong \fX^{\rm c}\setminus \Gamma$. 

Next we show that $\sU^+\cap \sV$ is saturated in both $\sU^+$ and $\sV$.  Since $\sV\subset \sX^+$, we know that $\sU^+\cap \sV = \sU\cap \sX^+ \cap \sV=\sU\cap \sV =\sX^{\rm s}$. In particular, every $\bk$-point of $\sU^+\cap \sV$ is closed with finite inertia in $\sX$. Thus $\sU^+\cap \sV$ is saturated in both $\sU^+$ and $\sV$. Thus Theorem \ref{thm:gms-glue} implies that $\sX^+=\sU^+\cup \sV$ admits a good moduli space which we denote by $\fX^+ = \fU^+ \cup \fV$. 

Next we show that the good moduli space $\fX^+$ of $\sX^+$ is separated. By \cite[Theorem 5.4]{AHLH18}, it suffices to show that $\sX^+$ is S-complete with respect to DVRs $(R,\fm, \kappa)$ that are essentially of finite type over $\bk$. Let $f:\STR \setminus 0 \to \sX^+$ be a map. Then $f$ is obtained by gluing two maps $f_1, f_2: \Spec R \to \sX^+$ such that $f_1|_{\Spec K} = f_2|_{\Spec K}$ where $K= \mathrm{Frac}(R)$. If the image of $f$ lies in $\sU^+$ (resp. in  $\sV$), then there exists an extension $\bar{f}: \STR \to \sU^+$ (resp. $\bar{f}: \STR\to \sV$) by the S-completeness of $\sU^+$ (resp. of $\sV$). Moreover, such an extension is unique by \cite[\S 3.2]{ABHLX19} (see also \cite[\S 5.2]{ABB+}). If the image of $f$ is not contained in $\sU^+$ or $\sV$, after possibly switching $t$ and $s$ we may assume that $f_1([\fm])\in \sU^+\setminus \sV$ and $f_2([\fm])\in \sV\setminus\sU^+$. This implies that $\phi(f_1([\fm]))\in \Gamma$ and $\phi(f_2([\fm]))\in \{C_{2A_5}, C_D\}$. By the S-completeness of $\sX^{\rm c}$, we know that $\Phi\circ f: \STR\setminus 0\to \sX^{\rm c}$ admits an extension to $\STR$. In particular, we have $\phi(f_1([\fm]))=\phi(f_2([\fm]))$. This contradicts the fact that $\Gamma$ and $\{C_{2A_5}, C_D\}$ are disjoint in $\fX^{\rm c}$. Thus the separatedness of $\fX^+$ is proved.

Finally, we show that $\fX^+$ is proper. By Theorem \ref{thm:gms-universal}, the composition $\sX^+ \hookrightarrow \sX \to \sX^{\rm c}$ induces a morphism $\psi: \fX^+\to \fX^{\rm c}$ between separated good moduli spaces. By Lemma \ref{lem:Chow-normal} and Theorem \ref{thm:sX-gms} we know that both $\fX^+$ and $\fX^{\rm c}$ are normal and irreducible. 
Moreover, from the construction we know that $\psi$ is birational. Since $\fX^{\rm c}$ is normal, irreducible and proper, by Lemma \ref{lem:proper-algspace} it suffices to show that every fiber of $\psi$ over $\bk$-points is proper and non-empty. Since $\sV$ is saturated in $\sX^+$, we know that $\psi$ is an isomorphism over $\fX^{\rm c}\setminus \Gamma$. Thus it suffices to show that $\psi^{-1}([C_{A,B}])$ is proper and non-empty for every $(A,B)\in \bk^2\setminus\{(0,0)\}$.  By Proposition \ref{prop:U+closed} we know that $\psi^{-1}(\Gamma)$ precisely consists of closed points of the stack $\sU^+\setminus \sV$, which are either GIT polystable hyperelliptic curves, or curves of degree $6$ on $\bP(1,1,2)$ with an $A_3$-singularity at the cone vertex and at worst $A_4$-singularities elsewhere.

If $4A=B^2$, then $C_{A,B}$ is the canonical ribbon which we denote by $C_R$. By Proposition \ref{prop:A4-cone}, 
$\psi^{-1}([C_{R}])$ precisely consists of GIT polystable hyperelliptic curves. Thus we have an isomorphism $\fH^+\cong \psi^{-1}([C_{R}])$ which implies the properness of $\psi^{-1}([C_{R}])$, where $\fH^+=\fH_4^+$ is as in Definition \ref{def:hyp-GIT}. 

If $4A\neq B^2$, then by Proposition \ref{prop:A3-cone},  $\psi^{-1}([C_{A,B}])$ precisely consists curves of the form $(y^2z^2 + Bx^3 yz + Ax^6 +h_4(x,y)=0)$ where $h_4\neq 0$. Such curves are parameterized by the weighted  projective stack $[\bA^5\setminus \{0\}/\bG_m]$ where the $\bG_m$ acts on $\bA^5 \cong H^0(\bP^1_{[x,y]}, \cO(4))$ with weight $(1,3)$ on $(x,y)$. Thus we have a morphism 
$[(\bA^5\setminus \{0\})/\bG_m]\to \sX^+$
that descends to a morphism $(\bA^5\setminus \{0\})\sslash\bG_m \to \fX^+$ on good moduli spaces whose image equals  $\psi^{-1}([C_{A,B}])$. Thus the properness of $\psi^{-1}([C_{A,B}])$ follows from the properness of $(\bA^5\setminus \{0\})\sslash\bG_m$. The proof is finished.
\end{proof}

The following lemma was used in the previous proof. 

\begin{lem}\label{lem:proper-algspace}
Let $f:X\to Y$ be a birational morphism between normal irreducible separated algebraic spaces of finite type over $\bk$.
Assume that $f$ is surjective. If every fiber of $f$ over $\bk$-points of $Y$ is proper, then $f$ is proper.
\end{lem}

\begin{proof}
By the Nagata compactification for algebraic spaces \cite{Rao74, CLO12}, there exists an open immersion $j: X\to \oX$ to an algebraic space $\oX$ and a proper morphism $\bar{f}: \oX \to Y$ such that $f= \bar{f}\circ j$. By replacing $\oX$ with the normalization of the irreducible component containing $X$, we may assume that $\oX$ is normal irreducible as well. By Stein factorization theorem for algebraic spaces \cite[\href{https://stacks.math.columbia.edu/tag/0AYI}{Tag 0AYI}]{stacksproject}, we know that $\bar{f}^{-1}(y)$ is connected for every $y\in Y(\bk)$. Since $f^{-1}(y)$ is proper, non-empty, and open in $\bar{f}^{-1}(y)$, we have $f^{-1}(y) = \bar{f}^{-1}(y)$ for every $y\in Y(\bk)$. Thus we have $X(\bk) = \oX(\bk)$ which implies that $X = \oX$. Thus $f=\bar{f}$ is proper.
\end{proof}

\begin{corollary}\label{cor:U+toU}
    The open immersion $\sU^+ \to \sU$ descends to a proper birational morphism of good moduli spaces $\fU^+ \to \fU$ with exceptional locus consisting of the curves with $A_3$ singularities at the singularity of $Q$ mapping to the associated point of $C_{A,B}$ in Proposition \ref{prop:A3-cone} and GIT polystable hyperelliptic curves mapping to $C_R$, the canonical ribbon.
\end{corollary}

\begin{proof}
    This follows from the proof of Theorem \ref{thm:X+gms}.
\end{proof}

\begin{remark}\label{rmk:sX-andVGIT}
    Analogously to $\sX^+$, we can define a companion open substack $\sX^-\hookrightarrow \sX$ of $(2,3)$ complete intersection curves $C = V(f,q)$ in $\sX$ where, if $Q=V(q)$ is a singular quadric surface, $C$ has at worst $A_2$ singularities at the vertex of $Q$. Moreover, one can show that there is a wall crossing diagram as follows.
\[
\begin{tikzcd}
\sX^- \arrow[d] \arrow[r,hook]  &
\sX \arrow[d] & 
\sX^+ \arrow[d] \arrow[l,hook']  \\
\fX^- \arrow[r]  & 
\fX^{\rm c}  & 
\fX^+ \arrow[l]
\end{tikzcd}
\]
Here the top arrows are open immersions of stacks, the bottom arrows are birational morphisms between proper varieties, and the the vertical arrows are good moduli space morphisms. One can show that $\sX^-$ and $\ocM_4(\frac{5}{9}, \frac{23}{44})$ are isomorphic near the locus of curves whose Chow polystable degeneration lies in $\Gamma$. 
\end{remark}

Combining Proposition \ref{prop:59-ewall} and the diagram from Remark \ref{rmk:sX-andVGIT}, we illustrate the wall crossing for $\alpha = \frac{5}{9}$ near the curve $\Gamma \subset \fX^{\rm c}$ in Figure \ref{fig:M59}.  The description of the preimage on the left side follows from Theorem \ref{thm:lastVGITwall} and the description of the preimage on the right from Corollary \ref{cor:U+toU}. For a complete description of the wall crossing at $\alpha = \frac{5}{9}$, see Theorem \ref{thm:new-models}.

\begin{center}
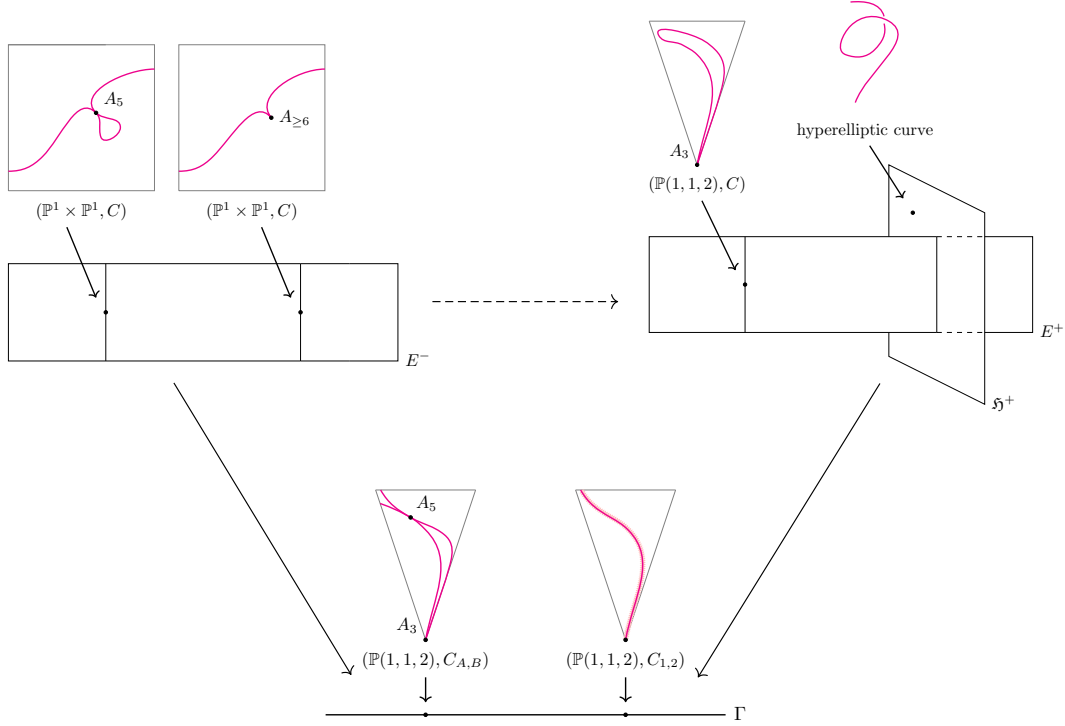
\begin{figure}
\begin{tikzcd}
        \resizebox{.4\textwidth}{!}{
    \begin{tabular}{c}
    {\scalebox{4}{\begin{tabular}{c}
	\begin{tikzpicture}[gren0/.style = {draw, circle,fill=greener!80,scale=.7},gren/.style ={draw, circle, fill=greener!80,scale=.4},blk/.style ={draw, circle, fill=black!,scale=.08},plc/.style ={draw, circle, color=white!100,fill=white!100,scale=0.02},smt/.style ={draw, circle, color=gray!100,fill=gray!100,scale=0.02},lbl/.style ={scale=.2}] 


		\draw [-,color=gray] (-1-1, 3) to (1-1,3);
            \draw [-,color=gray] (-2,0) to (-2,3) to (1,3) to (1,0) to (-2,0);

        \draw [-,color=magenta,thick] (-2,.4) to [out=0, in=-135] (.7-2,.7) to [out=45,in=135] (-.2,1.6) to [out=-45, in=110] (-0.1,1.1) to [out=-70,in=-90] (.3,1.3) to [out=90,in=-45] (-.2,1.6) to [out=135, in=180] (3-2,2.5);
    \filldraw (-.2,1.6) circle (1pt);

		\node[below, node font=\small] at (0-.5,-.1) {$(\bP^1 \times \bP^1, C)$};

		\node[above right, node font=\small] at (-.2,1.6) {$A_5$};

        \draw [-,color=gray] (0+1.5,0) to (0+1.5,3) to (3+1.5,3) to (3+1.5,0) to (0+1.5,0);

    \draw [-,color=magenta,thick] (1.5,.4) to [out=0, in=-135] (.7+1.5,.7) to [out=45,in=130] (1.8+1.6,1.5) to [out=130, in=180] (3+1.5,2.5);
    \filldraw (3.4,1.5) circle (1pt);

    \node[below, node font=\small] at (1.5+1.5,-.1) {$(\bP^1 \times \bP^1, C)$};
    \node[right, node font=\small] at (3.4,1.5) {$A_{\ge 6}$};

    \draw[-]  (4,-3.5) to (5,-3.5);
    \draw[-]  (4,-1.5) to (5,-1.5);
    \draw[-]  (5,-1.5) to (6,-1.5) to (6, -3.5) to (5, -3.5);
    \draw[-]  (4, -3.5) to (-2, -3.5) to (-2,-1.5) to (4,-1.5);
    \draw[-] (0,-1.5) to (0,-3.5); 
    \filldraw (0, -2.5) circle (1pt);
    \draw[-] (4,-1.5) to (4,-3.5); 
    \filldraw (4, -2.5) circle (1pt);
    \node[right, node font=\small] at (6,-3.5) {$E^-$};

    \draw[->,thick] (-.8,-.75) to (-.2, -2.2);
    \draw[->,thick] (3.2,-.75) to (3.8, -2.2);
	\end{tikzpicture}
    
\end{tabular}}}
    \end{tabular}
    } \arrow[rd, start anchor={[xshift=-20ex]},
end anchor={[xshift=-15ex,yshift=-15ex]}] \arrow[rr,dashed,start anchor={[xshift=-3ex,yshift=-8ex]},
end anchor={[xshift=-19ex,yshift=-8ex]}] & & 
    \hspace{-1.25in} \resizebox{.4\textwidth}{!}{
    \begin{tabular}{c}
    {\scalebox{4}{\begin{tabular}{c}
	\begin{tikzpicture}[gren0/.style = {draw, circle,fill=greener!80,scale=.7},gren/.style ={draw, circle, fill=greener!80,scale=.4},blk/.style ={draw, circle, fill=black!,scale=.08},plc/.style ={draw, circle, color=white!100,fill=white!100,scale=0.02},smt/.style ={draw, circle, color=gray!100,fill=gray!100,scale=0.02},lbl/.style ={scale=.2}] 

		\draw [-,color=gray] (0-1,0) to (1-1, 3);
		\draw [-,color=gray] (0-1,0) to (-1-1, 3);

		\draw [-,color=gray] (-1-1, 3) to (1-1,3);

    \draw [-,color=magenta,thick] (0-1,0) to [out=80,in=-45] (0-1,2.3) to [out=135,in=-30] (-.3-1,2.5) to [out=150,in=-60] (-1.8,2.7) to [out = 120, in=160] (-1.1, 2.7) to [out=-20,in=135] (.38-1,2.4)  to [in=75,out=-45] (.4-1,1.2) to (-1,0);

		\filldraw (0-1,0) circle (1pt);

		\node[below, node font=\small] at (0-1,-.1) {$(\bP(1,1,2), C)$};
		\node[above left, node font=\small] at (0-1,0) {$A_{3}$};

    \draw [-,color=magenta,thick] (2.3, 1.3) to [in=-30,out=70] (3, 3) to [out=150,in=45] (2.2, 2.9) to [out=-135,in=160] (2.1,2.3) to [out=-20,in=-135] (2.7, 2.4) to [out=45,in=-75] (2.9, 2.95);
    \draw [-,color=magenta,thick](2.9, 3.1) to [out=105,in=5] (2.2, 3.4);

    \node[below, node font=\small] at (1.5+1,1) {hyperelliptic curve};

    \draw[-,dashed]  (4,-3.5) to (5,-3.5);
    \draw[-,dashed]  (4,-1.5) to (5,-1.5);
    \draw[-]  (5,-1.5) to (6,-1.5) to (6, -3.5) to (5, -3.5);
    \draw[-]  (4, -3.5) to (-2, -3.5) to (-2,-1.5) to (4,-1.5);
    \draw[-] (0,-1.5) to (0,-3.5); 
    \filldraw (0, -2.5) circle (1pt);
    \draw[-] (3,-1.5) to (3,0) to (5, -1) to (5, -5) to (3,-4) to (3,-3.5);
    \draw[-] (4,-1.5) to (4,-3.5);
    \filldraw (3.5,-1) circle (1pt);
    \node[right, node font=\small] at (6,-3.5) {$E^+$};
    \node[right, node font=\small] at (5,-5) {$\frak{H}^+$};

    \draw[->,thick] (-.9,-.75) to (-.2, -2.2);
    \draw[->,thick] (2.6,.4) to (3.3, -.7);
	\end{tikzpicture}
    
\end{tabular}}}
    \end{tabular}
    } \arrow[ld,start anchor={[xshift=2ex,yshift=-2ex]},
end anchor={[xshift=-5ex,yshift=-16ex]}] \\
    & \hspace{-1.25in}    \resizebox{.4\textwidth}{!}{
    \begin{tabular}{c}
    {\scalebox{4}{\begin{tabular}{c}
	\begin{tikzpicture}[gren0/.style = {draw, circle,fill=greener!80,scale=.7},gren/.style ={draw, circle, fill=greener!80,scale=.4},blk/.style ={draw, circle, fill=black!,scale=.08},plc/.style ={draw, circle, color=white!100,fill=white!100,scale=0.02},smt/.style ={draw, circle, color=gray!100,fill=gray!100,scale=0.02},lbl/.style ={scale=.2}] 
		
		\draw [-,color=gray] (0,0) to (1, 3);
		\draw [-,color=gray] (0,0) to (-1, 3);

		\draw [-,color=gray] (-1, 3) to (1,3);

        \draw [-,color=pink, decorate,decoration={snake, amplitude=.4mm,segment length=.4mm}] (0,0) to [out=80,in=-45] (.05,2.2) to [out=135,in=-30] (-.3,2.45) to [out=150,in=-60] (-.9,3);
		\draw [-,color=magenta,thick] (0,0) to [out=80,in=-45] (.05,2.2) to [out=135,in=-30] (-.3,2.45) to [out=150,in=-60] (-.9,3);

		\filldraw (0,0) circle (1pt);

		\node[below, node font=\small] at (0,-.1) {$(\bP(1,1,2), C_{1,2})$};

    	\draw [-,color=gray] (0-4,0) to (1-4, 3);
		\draw [-,color=gray] (0-4,0) to (-1-4, 3);

		\draw [-,color=gray] (-1-4, 3) to (1-4,3);

		\draw [-,color=magenta,thick] (0-4,0) to [out=80,in=-45] (0-4,2.2) to [out=135,in=-30] (-.3-4,2.45) to [out=150,in=-60] (-.9-4,3);
		\draw [-,color=magenta,thick] (0-4,0) to (.4-4,1.2) to [out=75,in=-40] (.38-4,2.1) to [out=140,in=-30] (-.3-4,2.45) to [out=150,in=-20] (-.91-4,2.73);

		\filldraw (0-4,0) circle (1pt);
		\filldraw (-.3-4,2.45) circle (1pt);
		
		\node[below, node font=\small] at (0-4,-.1) {$(\bP(1,1,2), C_{A,B})$};
		\node[above left, node font=\small] at (0-4,0) {$A_3$};
		\node[above right, node font=\small] at (-.35-4,2.45) {$A_5$};


        \draw[-,thick]  (-6,-1.5) to (2,-1.5);
        \filldraw (-4,-1.5) circle (1pt);
        \filldraw (0,-1.5) circle (1pt);
        \draw[->, thick] (-4,-.75) to (-4, -1.25);
        \draw[->, thick] (0,-.75) to (0, -1.25);
        \node[right] at (2,-1.5) {$\Gamma$};
	\end{tikzpicture}
\end{tabular}}}
    \end{tabular}
    } & 
\end{tikzcd}

    \caption{A depiction of the wall crossing $\oM_4(\frac{5}{9} -\epsilon) \rightarrow \oM_4(\frac{5}{9}) \leftarrow \oM_4(\frac{5}{9}+\epsilon)$ near the curve $\Gamma$.}
    \label{fig:M59}
\end{figure}
\end{center}

\begin{remark}\label{rmk:highergenus-hyperellipticflip}
We expect the construction of $\sX^{\pm}$ and the wall crossing diagram from Remark \ref{rmk:sX-andVGIT} can be generalized to any genus as open substacks $\sX_g^{\circ, \pm}\hookrightarrow \sX_g^{\circ}$ that admit separated good moduli spaces $\fX_g^{\circ, \pm} \to \fX_g^{\circ}$ fitting into a similar diagram. This would provide a local model for the hyperelliptic flip in every genus $g\geq 4$. The idea is to follow the local VGIT approach of \cite[Theorem 1.3]{AFS17} at the hyperbolic ribbon $[R_{\hyp}]$. By Proposition \ref{prop:hyp-rib-intersection} and \cite{AHR20}, we know that the \'etale local structure of $\sX_g^{\circ}$ at $[R_{\hyp}]$ is given by 
\[
[\bA(T_{\sX_g^{\circ}, [R_{\hyp}]})/\Aut(R_{\hyp})]\cong [\bA(H^0(\bP^1, L^{\otimes -2})\oplus \Ext^1(\Omega_{\bP^1}, L))/(\GL(2)/\bmu_{g+1})], 
\]
where $L = \cO_{\bP^1}(-g-1)$. The center $\bG_m\cong \bG_m/\bmu_{g+1}$ of $\GL(2)/\bmu_{g+1}$ acts on $H^0(\bP^1, L^{\otimes -2})$ and $\Ext^1(\Omega_{\bP^1}, L)$ with weights $2$ and $-1$ respectively. Thus we can have a VGIT wall crossing for the $\GL(2)/\bmu_{g+1}$-action on $\bA(H^0(\bP^1, L^{\otimes -2})\oplus \Ext^1(\Omega_{\bP^1}, L))$ via the non-trivial character $\det:\GL(2) \to \bG_m$ in the sense of \cite{Tha96, DH98}.
We expect that this VGIT on the \'etale Luna slice can be extended to give $\sX_g^{\circ, \pm}\hookrightarrow \sX_g^{\circ}$.
\end{remark}

\section{Log Calabi--Yau wall crossing}\label{sec:logCYwallcrossing}

In this section, we describe the wall crossing for curves with separating $A_5$-singularities or $D_4$-singularities. Both loci can be replaced using the log Calabi--Yau wall crossing framework developed in \cite{ABB+}. Our approach is to glue the good moduli spaces therein with the Chow quotient. 
Note that these loci from the Chow moduli space are described in \cite[Section 3]{CMJL12} (see also \cite[Remark 3.3]{CMJL14}).

\subsection{Curves with separating $A_5$-singularities}

In the Chow moduli space $\fX^{\rm c}$, there is a unique polystable point with separating $A_5$-singularities: $C_{2A_5}= V(x_0x_3-x_1x_2, x_0x_2^2 +x_1^2 x_3)$.  In the associated wall crossing in the Hassett--Keel program, this is replaced by curves on a surface that is two copies of a $(1,3)$-weighted blow-up of $\bP(1,1,3)$, glued them along the ruling with weight $1$ upside down. The exceptional locus after the wall is a divisor as the birational transform of $\delta_2$. As we shall see later, this locus is replaced at $\alpha = \frac{19}{29}$, which matches the prediction from \cite{CMJL14}.

\begin{lemma}\label{lem:C2A5-is-K-polystable}
    The pair $(\bP^1 \times \bP^1, \frac{2}{3} C_{2A_5})$, depicted in Figure \ref{fig:C2A5andCD}, is K-polystable in $\ocM_4^{\K}$.
\end{lemma}

\begin{proof}
    By \cite[Section 2.1]{Fed12}, $C_{2A_5}$ is GIT polystable in the GIT moduli stack of $(3,3)$ curves on $\bP^1 \times \bP^1$.  By \cite[Theorem 1.4]{ADL20}, $(\bP^1 \times \bP^1, \epsilon C_{2A_5})$ is K-polystable for $0<\epsilon \ll 1$.  Because $(\bP^1 \times \bP^1, \frac{2}{3} C_{2A_5})$ is log canonical log Calabi--Yau, it is K-semistable, and hence by interpolation \cite[Proposition 2.13]{ADL19}, $(\bP^1 \times \bP^1, c C_{2A_5})$ is K-polystable for all $c \in (0, \frac{2}{3})$.  In particular, it is K-polystable for  $c = \frac{2}{3} - \epsilon$ so polystable in $\ocM_4^{\K}$. 
\end{proof}

In $\ocM_4^{\CY}$, however, the pair $(\bP^1 \times \bP^1, \frac{2}{3}C_{2A_5})$ does not remain polystable.  It admits a further degeneration to a curve on the non-normal surface $S_{2A_5}$ defined below.

\begin{defn}\label{def:S2A5}
    Let $\pi:S_{A_5}\to \bP(1,1,3)$ be the weighted blow-up of $\bP(1,1,3)_{[x_0,x_1,x_2]}$ at the point $p_1 = [1,0,0]$ with weight $(1,3)$ in $(x_1, x_2)$, depicted in Figure \ref{fig:SA5}. Denote by $\ell_{1}\subset S_{A_5}$ the strict transform of $(x_1=0)\subset \bP(1,1,3)$. Let $[u_0, u_1]$ be the coordinates of $\ell_1$ such that $\pi([u_0, u_1]) = [u_0, 0, u_1^3]$. Let $S_{2A_5}$ be the surface obtained as the union of two isomorphic surfaces  $S_{A_5}$ and $S_{A_5}'$ glued along $\ell_1$ and $\ell_1'$ via the map $[u_0,u_1]\mapsto [u_0', u_1']:=[u_1, u_0]$, depicted in Figure \ref{fig:S2A5}.
\end{defn}

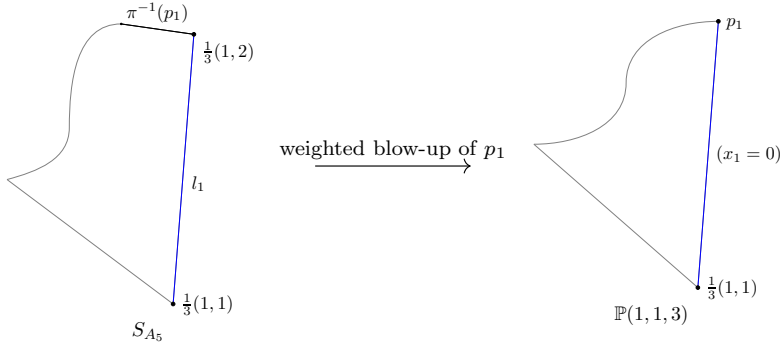
\begin{figure}[h]
    \centering
    \begin{tikzcd}
    \resizebox{.25\textwidth}{!}{
    \begin{tabular}{c}
    {\scalebox{4}{\begin{tabular}{c}
\begin{tikzpicture}[gren0/.style = {draw, circle,fill=greener!80,scale=.7},gren/.style ={draw, circle, fill=greener!80,scale=.4},blk/.style ={draw, circle, fill=black!,scale=.08},plc/.style ={draw, circle, color=white!100,fill=white!100,scale=0.02},smt/.style ={draw, circle, color=gray!100,fill=gray!100,scale=0.02},lbl/.style ={scale=.2}] 

\draw [-,color=gray] (-3.4,-.2) to [out=15, in=-90] (-2.2,.8) to [out=90, in=180] (-1.2,2.8);
\draw [-,color=gray] (-.2,-2.6) to (-3.4,-0.2);
\draw [-,color=gray] (0.2,2.6) to (-0.2,-2.6);
\draw [-] (-1.2,2.8) to (0.2,2.6);

\draw [-,color=blue] (0.2,2.6) to (-.2,-2.6);
\draw [-,color=black] (-1.2,2.8) to (0.2,2.6);

\filldraw (-.2,-2.6) circle (1pt);
\filldraw (0.2,2.6) circle (1pt);
\filldraw[black] (-1.2,2.8) circle (0.4pt);

\node[below left, node font=\normalsize] at (-.2,-2.8) {$S_{A_5}$};
\node[right, node font=\small] at (-.2,-2.6) {$\frac{1}{3}(1,1)$};
\node[below right, node font=\small] at (.2,2.6) {$\frac{1}{3}(1,2)$};
\node[above, node font=\small] at (-.5,2.7) {$\pi^{-1}(p_1)$};
\node[below right, node font=\small] at (0,0) {$l_1$};
\end{tikzpicture}
\end{tabular}}}
    \end{tabular}
    } \qquad \arrow[r, start anchor={[xshift=-3ex]}, end anchor={[xshift=4ex]}, "\text{weighted blow-up of } p_1"] & 
    \qquad \quad \resizebox{.25\textwidth}{!}{
    \begin{tabular}{c}
    {\scalebox{4}{\begin{tabular}{c}
\begin{tikzpicture}[gren0/.style = {draw, circle,fill=greener!80,scale=.7},gren/.style ={draw, circle, fill=greener!80,scale=.4},blk/.style ={draw, circle, fill=black!,scale=.08},plc/.style ={draw, circle, color=white!100,fill=white!100,scale=0.02},smt/.style ={draw, circle, color=gray!100,fill=gray!100,scale=0.02},lbl/.style ={scale=.2}] 

\draw [-,color=gray] (-3.4,.2) to [out=0, in=-90] (-1.6,1.4) to [out=90, in=180] (0.2,2.6);

\draw [-,color=gray] (-.2,-2.6) to (-3.4,.2);
\draw [-,color=gray] (0.2,2.6) to (-.2,-2.6);

\draw [-,color=blue] (0.2,2.6) to (-.2,-2.6);

\filldraw (-.2,-2.6) circle (1pt);
\filldraw (.2,2.6) circle (1pt);

\node[below left, node font=\normalsize] at (-.2,-2.8) {$\bP(1,1,3)$};
\node[right, node font=\small] at (-.2,-2.6) {$\frac{1}{3}(1,1)$};
\node[right, node font=\small] at (0.2,2.6) {$p_1$};
\node[right, node font=\small] at (0,0) {$(x_1 = 0)$};
\end{tikzpicture}
\end{tabular}}}
    \end{tabular}
    } \\
    \end{tikzcd}
    \vspace{-2em} 
    \caption{The construction of $S_{A_5}$.}
    \label{fig:SA5}
\end{figure}

\begin{figure}[h]
\resizebox{.32\textwidth}{!}{
\begin{tabular}{c}
{\scalebox{4}{\begin{tabular}{c}
\begin{tikzpicture}[gren0/.style = {draw, circle,fill=greener!80,scale=.7},gren/.style ={draw, circle, fill=greener!80,scale=.4},blk/.style ={draw, circle, fill=black!,scale=.08},plc/.style ={draw, circle, color=white!100,fill=white!100,scale=0.02},smt/.style ={draw, circle, color=gray!100,fill=gray!100,scale=0.02},lbl/.style ={scale=.2}] 

\draw [-,color=gray] (-3.4,-.2) to [out=15, in=-90] (-2.2,.8) to [out=90, in=180] (-1.2,2.8);
\draw [-,color=gray] (-.2,-2.6) to (-3.4,-0.2);
\draw [-,color=gray] (0.2,2.6) to (-0.2,-2.6);
\draw [-] (-1.2,2.8) to (0.2,2.6);

\draw [-,color=blue] (0.2,2.6) to (-.2,-2.6);
\draw [-,color=black] (-1.2,2.8) to (0.2,2.6);

\filldraw (-.2,-2.6) circle (1pt);
\filldraw (0.2,2.6) circle (1pt);
\filldraw (-1.2,2.8) circle (0.4pt);


\draw [-,color=gray] (3.4,.2) to [out=195, in=90] (2.2,-.8) to [out=-90, in=0] (1.2,-2.8);
\draw [-,color=gray] (.2,2.6) to (3.4,0.2);
\draw [-,color=gray] (-0.2,-2.6) to (0.2,2.6);
\draw[-] (1.2,-2.8) to (-0.2,-2.6);

\draw [-,color=blue] (-0.2,-2.6) to (.2,2.6);
\draw [-,color=black] (1.2,-2.8) to (-0.2,-2.6);

\filldraw (.2,2.6) circle (1pt);
\filldraw (-0.2,-2.6) circle (1pt);
\filldraw (1.2,-2.8) circle (0.4pt);



\node[above left, node font=\small] at (-.2,-2.6) {$\frac{1}{3}(1,1)$};
\node[above right, node font=\small] at (-.2,-2.6) {$\frac{1}{3}(1,2)$};
\node[below left, node font=\small] at (.2,2.6) {$\frac{1}{3}(1,2)$};
\node[below right, node font=\small] at (.2,2.6) {$\frac{1}{3}(1,1)$};
\node[below left, node font=\small] at (0,0) {$l_1$};
\node[above right, node font=\small] at (0,0) {$l_1'$};

\end{tikzpicture}
\end{tabular}}}
\end{tabular}
}
\caption{The surface $S_{2A_5}$.}
\label{fig:S2A5}
\end{figure}
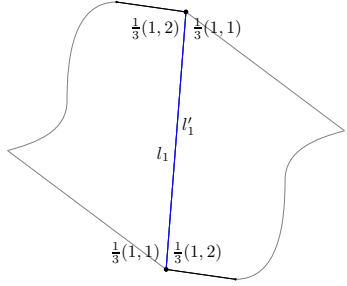

\begin{prop}\label{prop:S2A5polystable}
The polystable degeneration of $(\bP^1\times \bP^1, \frac{2}{3}C_{2A_5})$ in $\ocM_4^{\CY}$ is $(S_{2A_5}, \frac{2}{3}C_{2A_5}^\circ)$ where $C_{2A_5}^\circ$ is defined by the strict transform of the equations $(x_2^2 - x_0^3 x_2 = 0)$ and $(x_2'^2 - x_0'^3 x_2' = 0)$ on the irreducible components $S_{A_5}$ and $S_{A_5}'$, respectively.  This is depicted in the lower right of Figure \ref{fig:C2A50}.
\end{prop}

\begin{proof}
We construct an explicit degeneration of $(\bP^1\times \bP^1, C_{2A_5})$ to $(S_{2A_5}, C_{2A_5}^\circ)$ by birationally modifying the constant family $(\cS, \cC) := (\bP^1\times \bP^1, C_{2A_5}) \times \bA^1$ over $\bA^1$.  We perform several birational transformations of the family, modifying the central fiber to $(S_{2A_5}, C_{2A_5}^\circ)$.  The central fibers of the transformations are depicted in Figure \ref{fig:C2A50}.  Let $p_1$, $p_2$ denote the two $A_5$ singularities in the central fiber $t  =0$.  Let $f_1$ (resp. $f_2$) be the ruling of $\cS_0 = \bP^1 \times \bP^1$ through $p_1$ (resp. $p_2$) contained in the curve $\cC_0 = C_{2A_5}$, and let $C$ denote the remaining component.  Suppose $p_1$ (resp. $p_2$) is given by local coordinates $(y_1, z_1, t)$ with local equation $y_1(y_1-z_1^3) = 0, t = 0$  (resp. $(y_2, z_2, t)$ with equation $y_2(y_2-z_2^3) = 0, t = 0$) and $f_1 = (y_1 = 0)$ (resp. $f_2 = (y_2 = 0)$). Perform a $(3,1,1)$ weighted blow-up of $\cS$ in the coordinates $(y_i, z_i, t)$ of $p_1$ and $p_2$, resulting in exceptional divisors $E_1 \cong \bP(1,1,3)$ and $E_2 \cong \bP(1,1,3)$.  Denote this blow up by $\phi: \hcS \to \cS$ and denote by $\tcS_0$ the strict transform of $\cS_0$, so $\hcS_0 = \tcS_0 \cup E_1 \cup E_2$.  By direct computation, in a suitable choice of coordinates on $E_i \cong \bP(1,1,3)_{[x_0:x_1:x_2]}$, $\tcS_0 \cap E_i = (x_1 = 0)$ and $\cC|_{E_i}$ has equation $x_2(x_2-x_0^3) = 0$.

The surface $\tcS_0$ is the $(3,1)$ weighted blow-up of $\bP^1 \times \bP^1$ in the two points $p_1$ and $p_2$.  Let $\Delta_i = E_i|_{\tcS_0}$ be the exceptional divisor over $p_i$.  By construction, $(\Delta_i)^2 = -\frac{1}{3}$ and $\tcS_0$ has a $\frac{1}{3}(1,2)$ singularity along $\Delta_i$.  If $\tilde{f}_i$ denotes the strict transform of $f_i$, $\tilde{f}_i$ is contained in the smooth locus of $\tcS_0$ and by direct computation $(\phi|_{\tcS_0})^* f_i = \tilde{f}_i + 3 \Delta_i$, so $(\tilde{f}_i)^2 = -3$.  Furthermore, $\phi^*(\cS_0) = \tcS_0 + E_1 + E_2$, so $\tcS_0 \cdot \tilde{f}_i = -1$.  It follows that $\tilde{f}_i$ is extremal in $\overline{NE}(\hcS)$.  

Let $F_i = f_i \times \bA^1 \subset \cS$ and, as above, compute $\phi^* F_i = \tilde{F}_i + 3 \Delta_i$, where $\tilde{F}_i$ is the strict transform of $F_i$.  Then, $\tilde{F}_i \cdot \tilde{f}_i = -3$ and $\phi|_{\tilde{F}_i}: \tilde{F}_i \to F_i$ is the $(1,1)$-weighted blow up of $p_i \in F_i$, hence $(\tilde{f}_i)^2 = -1$ where this intersection number is computed on the surface $\tilde{F}_i$.  As $(\hcS, \tilde{F}_i)$ is a plt pair and $(K_{\hcS} + \tilde{F}_i) \cdot \tilde{f}_i = -1$, $\tilde{f}_i$ is in fact a contractible extremal ray in $\overline{NE}(\hcS)$ and we may construct the flip of $\tilde{f}_i$ for each $i$.  Denote by $\hcS^+$ the (relative) flip (over $\bA^1$) of $\tilde{f}_1$ and $\tilde{f}_2$.  We construct the flip explicitly below.

By the exact sequence 
\[ 0 \to \cN_{\tilde{f}_i/\tilde{F}_i} \to \cN_{\tilde{f_i}/\hcS} \to \cN_{\tilde{F}_i/\hcS}|_{\tilde{f}_i} \to 0,\]
which by the computation above is 
\[ 0 \to \cO(-1) \to \cN_{\tilde{f_i}/\hcS} \to \cO(-3) \to 0, \]
we see that $\cN_{\tilde{f_i}/\hcS} \cong \cO(-1) \oplus \cO(-3)$ (as $\Ext^1_{\bP^1}(\cO(-3), \cO(-1)) = 0$).  

In $\hcS$, blow up each curve $\tilde{f}_i$ in a morphism $\phi_1: \hcS_1 \to \hcS$, which by the normal bundle computation yields exceptional divisors $\Sigma_{i,1} \cong \bF_2$.  By a similar computation, the negative section $\sigma_{i,1}$ of $\Sigma_{i,1}$ has normal bundle $\cN_{\sigma_{i,1}/\hcS_1} \cong \cO(-1) \oplus \cO(-2)$ and by construction is not contained in the strict transform of $\tcS_0$.  Next, blow up each $\sigma_{i,1}$ in a morphism $\phi_2: \hcS_2 \to \hcS_1$ which results in exceptional divisors $\Sigma_{i,2} \cong \bF_1$, with negative sections $\sigma_{i,2}$ not contained in the strict transform of $\Sigma_{i,1}$ and $\cN_{\sigma_{i,2}/\hcS_2} \cong \cO(-1) \oplus \cO(-1)$.  Now, perform the Atiyah flop of each $\sigma_{i,2}$ by blowing up each $\sigma_{i,2}$ to create an exceptional divisor $\Sigma_{i,3} \cong \bP^1 \times \bP^1$ and then contracting the opposite ruling.  Let $\hcS_2 \dashrightarrow \hcS_2^+$ be the Atiyah flop.  The strict transform of $\Sigma_{i,2}$ in $\hcS_2^+$ is isomorphic to $\bP^2$ with normal bundle $\cO(-2)$, so may be contracted to a $\frac{1}{2}(1,1,1)$ singularity in a morphism $\phi_2^+: \hcS_2^+ \to \hcS_1^+$.  Finally, the strict transform of $\Sigma_{i,1}$  in $\hcS_1^+$ is isomorphic to $\bP(1,1,2)$ with normal bundle $\cO(-3)$ (in the weighted coordinates on $\bP(1,1,2)$) and hence may be contracted to a $\frac{1}{3}(1,1,2)$ singularity in a morphism $\phi_1^+: \hcS_1^+ \to \hcS^+$.  

Denote by $\tilde{f}_i^+$ the image of $\Sigma_{i,3}$ on $\hcS^+$.  It is straightforward from the construction to see that the induced map $\hcS \dashrightarrow \hcS^+$ is an isomorphism of $\hcS \setminus \{ \tilde{f}_1 \cup \tilde{f}_2 \}$ with $\hcS^+ \setminus \{ \tilde{f}_1^+ \cup \tilde{f}_2^+ \}$ and that $\tilde{f}_i^+$ is contractible in $\hcS^+$.  Furthermore, if $\tilde{F}_i^+$ is the strict transform of $\tilde{F}_i$, by computation, $(K_{\hcS^+} + \tilde{F}_i^+)\cdot \tilde{f}_i^+ = \frac{1}{3}$ and hence by uniqueness of flips \cite[Corollary 6.4]{KM98}, the rational map $\hcS \dashrightarrow \hcS^+$ is the flip of $\{ \tilde{f}_1, \tilde{f}_2\}$

Now, observe that the rational map $\hcS \dashrightarrow \hcS^+$ restricted to $\tcS_0$ is the contraction of each $\tilde{f}_i \subset \tcS_0$, creating two $\frac{1}{3}(1,1)$ singularities on its strict transform $\tcS_0^+$.  Similarly, the restriction to $E_i$ is the $(1,3)$ weighted blow-up of $E_i$ at the intersection point of $E_i$ and $\tilde{f}_i$ with exceptional divisor denoted $\tilde{f}_i^+$, the flip of $\tilde{f}_i$.  Denote by $E_i^+$ the strict transform of $E_i$.

The surface $\tcS_0^+$ then has Picard rank 2 with two extremal rays, $[\Delta_i^+:=\tcS_0^+\cap E_i^+]$ (both $i = 1, 2$ are equivalent) and $[C^+]$, where $C^+$ is the strict transform of $C$ in $\tcS_0^+$.  Indeed, by computation, $(\Delta_i^+)^2 = 0$ and $(C^+)^2 = 0$, yet $C^+ \cdot \Delta_i^+ = 1$ so both curves are extremal.  Because $\tcS_0^+ \cdot C^+ = -2$, $[C^+]$ is in fact extremal in $\hcS^+$, and as $(\hcS^+,\tcS_0^+)$ is plt and $(K_{\hcS^+} + \tcS_0^+) \cdot C^+= -2$, $[C^+]$ can be contracted over $\bA^1$ in $\hcS^+$.  Denote the image of this contraction by $\oS$.  This contraction glues the surfaces $E_1^+$ and $E_2^+$ together along $\Delta_1^+$ and $\Delta_2^+$.  In summary, the family $\oS \to \bA^1$ has generic fiber $\bP^1 \times \bP^1$ and central fiber $E_1^+ \cup E_2^+ \cong S_{2A_5}$ and the image of the family of curves $\cC$ has equation as claimed.

Finally, it is straightforward to see that $C_{2A_5}^\circ$ has an $A_5$ singularity in the smooth locus of each component of $S_{2A_5}$ and $(S_{2A_5}, \frac{2}{3} C_{2A_5}^\circ)$ is slc so indeed is a point of $\ocM_4^{\CY}$.  To verify the polystability, observe that $\Aut(S_{2A_5}, C_{2A_5}^\circ)$ contains $\mathbb{G}_m^2$: on each component, we may consider the automorphism $[x_0:x_1:x_2] \mapsto [x_0: \lambda x_1: x_2 ]$ of $\bP(1,1,3)$.  This fixes the curves $x_2(x_2 - x_0^3) = 0$ and $x_1 =0$ and thus extends to an automorphism of each component of $S_{2A_5}$ fixing $C_{2A_5}^\circ$ and the intersection of the two components.  In particular, we have an independent $\mathbb{G}_m$ action on each component of $S_{2A_5}$ so  $\mathbb{G}_m^2 \subset \Aut(S_{2A_5}, C_{2A_5}^\circ)$.  Therefore, by \cite[Proposition 8.11]{ABB+}, $\Aut(S_{2A_5}, C_{2A_5}^\circ) = \bG_m^2$ and $(S_{2A_5}, \frac{2}{3} C_{2A_5}^\circ)$ is polystable. 
\end{proof}

\begin{figure}[h]
\begin{adjustwidth}{-5ex}{0ex}
    \centering
    \begin{tikzcd}
    \resizebox{.5\textwidth}{!}{
    \begin{tabular}{c}
    {\scalebox{4}{\begin{tabular}{c}
\begin{tikzpicture}[gren0/.style = {draw, circle,fill=greener!80,scale=.7},gren/.style ={draw, circle, fill=greener!80,scale=.4},blk/.style ={draw, circle, fill=black!,scale=.08},plc/.style ={draw, circle, color=white!100,fill=white!100,scale=0.02},smt/.style ={draw, circle, color=gray!100,fill=gray!100,scale=0.02},lbl/.style ={scale=.2}] 

\draw [-,color=gray] (-.2,.1) to (-.2,3.9) to (4.2,3.9) to (4.2,.1) to (-.2,.1);

\draw [-,color=magenta,thick] (0.1+.8,1.3+1.4) to [out=45, in=-90] (1,3.9);
\draw [-,color=magenta,thick] (-0.1+3.2,-1.3+2.6) to [out=-135, in=90] (3.2,0.1);
\draw [-,color=magenta,thick] (-.05+.8,-.6+1.4) to [out=-15,in=165] (.05+3.2,.6+2.6);

\draw [-,color=gray] (-1.7+.8,.1+1.4) to [out=0, in=-90] (-.8+.8,.7+1.4) to [out=90, in=180] (0.1+.8,1.3+1.4);

\draw [-,color=gray] (-.1+.8,-1.3+1.4) to (-1.7+.8,.1+1.4);
\draw [-,color=gray] (0.1+.8,1.3+1.4) to (-.1+.8,-1.3+1.4);

\draw [-,color=blue] (0.1+.8,1.3+1.4) to (-.1+.8,-1.3+1.4);

\draw [-,color=teal,thick] (-.05+.8,-.6+1.4) -- (-.8+.8,.05+1.4);
\draw [-,color=teal,thick] (-.6+.8,-.7+1.4) to [out=45, in=-45] (-.5+.8,-.21+1.4) to [out=135, in=-135] (0.1+.8,1.3+1.4);
\filldraw (-.1+.8,-1.3+1.4) circle (0.5pt);
\filldraw (.1+.8,1.3+1.4) circle (0.5pt);

\draw [-,color=gray] (1.7+3.2,-.1+2.6) to [out=180, in=90] (.8+3.2,-.7+2.6) to [out=-90, in=0] (-0.1+3.2,-1.3+2.6);

\draw [-,color=gray] (.1+3.2,1.3+2.6) to (1.7+3.2,-.1+2.6);
\draw [-,color=gray] (-0.1+3.2,-1.3+2.6) to (.1+3.2,1.3+2.6);

\draw [-,color=blue] (-0.1+3.2,-1.3+2.6) to (.1+3.2,1.3+2.6);

\draw [-,color=teal,thick] (.05+3.2,.6+2.6) -- (.8+3.2,-.05+2.6);
\draw [-,color=teal,thick] (.6+3.2,.7+2.6) to [out=-135, in=135] (.5+3.2,.21+2.6) to [out=-45, in=45] (-0.1+3.2,-1.3+2.6);
\filldraw (.1+3.2,1.3+2.6) circle (0.5pt);
\filldraw (-.1+3.2,-1.3+2.6) circle (0.5pt);

\node[below left, node font=\normalsize] at (-.5,1.3) {$E_1 = \bP(1,1,3)$};
\node[above right, node font=\normalsize] at (4.5,2.7) {$E_2 = \bP(1,1,3)$};
\node[above, node font=\normalsize] at (2,3.9) {$\tilde{\cS}_0$};
\node[right, node font=\small] at (1.1, 3.4) {$\tilde{f}_1$};
\node[left, node font=\small] at (3, 0.6) {$\tilde{f}_2$};
\end{tikzpicture}
\end{tabular}}}
    \end{tabular}
    } \quad \arrow[r, dashed, start anchor={[xshift=-3ex]}, end anchor={[xshift=4ex]}, "\text{flip of } \tilde{f}_1 \text{ and } \tilde{f}_2"] \arrow[d, near start, start anchor={[xshift=-4ex]}, end anchor={[xshift=-6ex]}, "\text{blow up of } p_1 \text{ and } p_2"] & 
    \quad \resizebox{.45\textwidth}{!}{
    \begin{tabular}{c}
    {\scalebox{4}{\begin{tabular}{c}
\begin{tikzpicture}[gren0/.style = {draw, circle,fill=greener!80,scale=.7},gren/.style ={draw, circle, fill=greener!80,scale=.4},blk/.style ={draw, circle, fill=black!,scale=.08},plc/.style ={draw, circle, color=white!100,fill=white!100,scale=0.02},smt/.style ={draw, circle, color=gray!100,fill=gray!100,scale=0.02},lbl/.style ={scale=.2}] 

\draw [-,color=gray] (-1,-1.3) to (-1,1.3) to (3,1.3) to (3,-1.3) to (-1,-1.3);

\draw [-,color=magenta,thick] (0,0) to [out=-15,in=165] (2,0);

\draw [-,color=gray] (-1.7,-.1) to [out=15, in=-90] (-1.1,.4) to [out=90, in=180] (-.6,1.4);
\draw [-,color=gray] (-.1,-1.3) to (-1.7,-0.1);
\draw [-,color=gray] (0.1,1.3) to (-0.1,-1.3);
\draw [-,color=gray] (-.6,1.4) to (0.1,1.3);

\draw [-,color=blue] (0.1,1.3) to (-.1,-1.3);
\draw [-,color=black] (-.6,1.4) to (0.1,1.3);

\draw [-,color=teal,thick] (0,0) -- (-.85,.65);
\draw [-,color=teal,thick] (-.8,-.5) to [out=30, in=-45] (-.4,.31) to [out=150, in=-135] (-.25,1.35);
\filldraw (-.1,-1.3) circle (0.5pt);
\filldraw (0.1,1.3) circle (0.5pt);
\filldraw[black] (-.6,1.4) circle (0.2pt);

\draw [-,color=gray] (3.7,.1) to [out=195, in=90] (3.1,-.4) to [out=-90, in=0] (2.6,-1.4);
\draw [-,color=gray] (2.1,1.3) to (3.7,0.1);
\draw [-,color=gray] (2-0.1,-1.3) to (2.1,1.3);
\draw [-,color=gray] (2.6,-1.4) to (2-0.1,-1.3);

\draw [-,color=blue] (2-0.1,-1.3) to (2.1,1.3);
\draw [-,color=black] (2.6,-1.4) to (2-0.1,-1.3);

\draw [-,color=teal,thick] (2,0) -- (2.85,-.65);
\draw [-,color=teal,thick] (2.8,.5) to [out=-150, in=135] (2.4,-.31) to [out=-30, in=45] (2.25,-1.35);
\filldraw (2.1,1.3) circle (0.5pt);
\filldraw (2-0.1,-1.3) circle (0.5pt);
\filldraw[black] (2.6,-1.4) circle (0.2pt);

\node[below left, node font=\normalsize] at (-1.5,-.1) {$E_1^+ = S_{A_5}$};
\node[above right, node font=\normalsize] at (3.5,.1) {$E_2^+ = S_{A_5}$};
\node[above, node font=\normalsize] at (1.2,1.3) {$\tilde{\cS}_0^+$};
\node[above, node font=\small] at (-.4, 1.4) {$\tilde{f}_1^+$};
\node[below, node font=\small] at (2.4, -1.4) {$\tilde{f}_2^+$};
\end{tikzpicture}
\end{tabular}}}
    \end{tabular}
    } \arrow[d, swap, near start, start anchor={[xshift=4ex]}, end anchor={[xshift=6ex]}, "\text{contraction of } \tilde{\cS}_0^+"] \\
    \resizebox{.2\textwidth}{!}{
    \begin{tabular}{c}
    {\scalebox{4}{\begin{tabular}{c}
\begin{tikzpicture}[gren0/.style = {draw, circle,fill=greener!80,scale=.7},gren/.style ={draw, circle, fill=greener!80,scale=.4},blk/.style ={draw, circle, fill=black!,scale=.08},plc/.style ={draw, circle, color=white!100,fill=white!100,scale=0.02},smt/.style ={draw, circle, color=gray!100,fill=gray!100,scale=0.02},lbl/.style ={scale=.2}] 

\draw [-,color=gray] (0,0) to (0,4) to (4,4) to (4,0) to (0,0);

\draw [color=magenta,thick] (.8,0) -- (.8,4);
\draw [color=magenta,thick] (3.2,0) -- (3.2,4);
\draw [-,color=magenta,thick] (.2,.8) to [out=0, in=-90] (.8,1.4) to [out=90,in=-90] (3.2,2.6) to [out=90, in=180] (3.8,3.2);
\filldraw[color=teal] (.8,1.4) circle (1pt);
\filldraw[color=teal] (3.2,2.6) circle (1pt);

\node[below, node font=\normalsize] at (2,0) {$\cS_0=\bP^1 \times \bP^1$};
\node[left, node font=\small] at (.8,1.4) {$p_1$};
\node[right, node font=\small] at (3.2,2.6) {$p_2$};
\node[right, node font=\small] at (.8, 3.4) {$f_1$};
\node[left, node font=\small] at (3.2, 0.6) {$f_2$};
\end{tikzpicture}
\end{tabular}}}
    \end{tabular}
    } \qquad \qquad \qquad \qquad & \qquad \qquad \qquad \resizebox{.29\textwidth}{!}{
    \begin{tabular}{c}
    {\scalebox{4}{\begin{tabular}{c}
\begin{tikzpicture}[gren0/.style = {draw, circle,fill=greener!80,scale=.7},gren/.style ={draw, circle, fill=greener!80,scale=.4},blk/.style ={draw, circle, fill=black!,scale=.08},plc/.style ={draw, circle, color=white!100,fill=white!100,scale=0.02},smt/.style ={draw, circle, color=gray!100,fill=gray!100,scale=0.02},lbl/.style ={scale=.2}] 

\draw [-,color=gray] (-3.4,-.2) to [out=15, in=-90] (-2.2,.8) to [out=90, in=180] (-1.2,2.8);
\draw [-,color=gray] (-.2,-2.6) to (-3.4,-0.2);
\draw [-,color=gray] (0.2,2.6) to (-0.2,-2.6);
\draw [-] (-1.2,2.8) to (0.2,2.6);

\draw [-,color=blue] (0.2,2.6) to (-.2,-2.6);
\draw [-,color=black] (-1.2,2.8) to (0.2,2.6);

\draw [-,color=teal,thick] (0,0) -- (-1.7,1.3);
\draw [-,color=teal,thick] (-1.6,-1) to [out=30, in=-45] (-.8,.62) to [out=150, in=-135] (-.5,2.7);
\filldraw (-.2,-2.6) circle (1pt);
\filldraw (0.2,2.6) circle (1pt);
\filldraw (-1.2,2.8) circle (0.4pt);


\draw [-,color=gray] (3.4,.2) to [out=195, in=90] (2.2,-.8) to [out=-90, in=0] (1.2,-2.8);
\draw [-,color=gray] (.2,2.6) to (3.4,0.2);
\draw [-,color=gray] (-0.2,-2.6) to (0.2,2.6);
\draw[-] (1.2,-2.8) to (-0.2,-2.6);

\draw [-,color=blue] (-0.2,-2.6) to (.2,2.6);
\draw [-,color=black] (1.2,-2.8) to (-0.2,-2.6);

\draw [-,color=teal,thick] (0,0) -- (1.7,-1.3);
\draw [-,color=teal,thick] (1.6,1) to [out=-150, in=135] (.8,-.62) to [out=-30, in=45] (.5,-2.7);
\filldraw (.2,2.6) circle (1pt);
\filldraw (-0.2,-2.6) circle (1pt);
\filldraw (1.2,-2.8) circle (0.4pt);

\node[below, node font=\normalsize] at (-.2,-2.8) {$S_{2A_5}$};
\node[left, node font=\normalsize] at (-1.4, -.5) {$C_{2A_5}^\circ$};

\end{tikzpicture}
\end{tabular}}}
    \end{tabular}
    } \\
    \end{tikzcd}
    \vspace{-2em} 
    \caption{The construction of $(S_{2A_5}, C_{2A_5}^\circ$) as a specialization of $(\bP^1 \times \bP^1, C_{2A_5})$.}
    \label{fig:C2A50}
    \end{adjustwidth}
\end{figure}
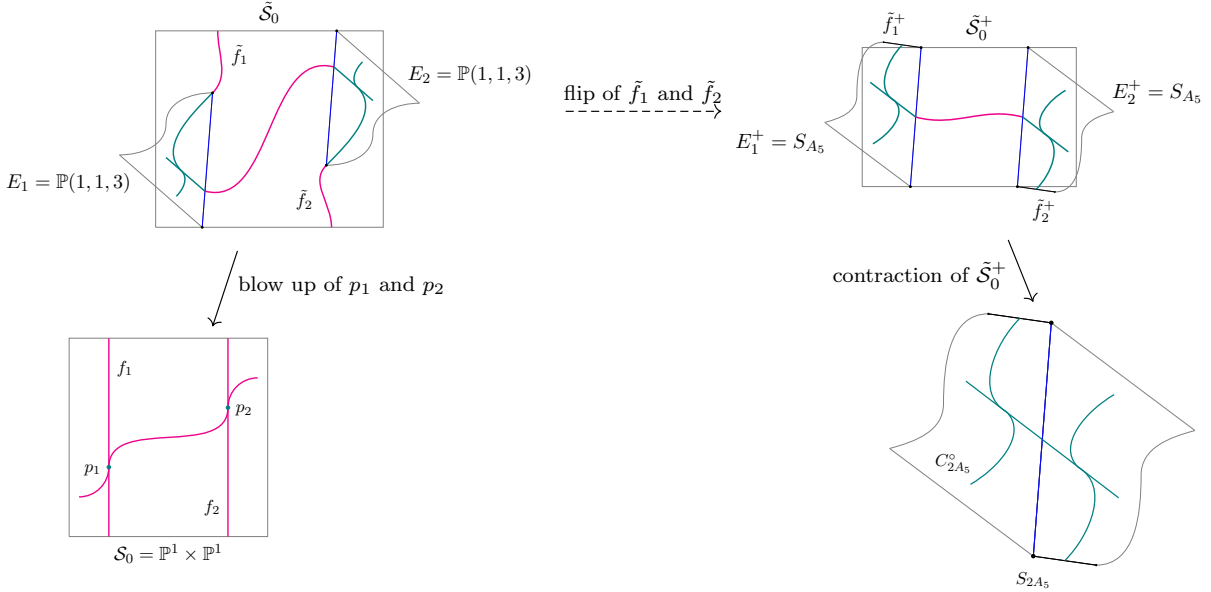

\begin{lem}\label{lem:S2A5-deform}
The $\bQ$-Gorenstein deformation of $S_{2A_5}$ is unobstructed. 
Moreover, every small $\bQ$-Gorenstein deformation of $S_{2A_5}$ is isomorphic to $S_{2A_5}$ or $\bP^1\times \bP^1$.
\end{lem}

\begin{proof}We will show that our surface $S_{2A_5}$ is the same surface that is parameterized by $\xi_{3,3}$ in \cite[Theorem 1]{DH21}. From here, by \cite[Theorem 6.6]{DH21}, the $\bQ$-Gorenstein deformations of the surface $S_{2A_5}$ are unobstructed. Furthermore, by \cite[Theorem 1 (4)]{DH21}, the locus of pairs $(S, C)$ is a divisor, and so the only possible deformations are $S$ or $\bP^1 \times \bP^1$.

By Definition \ref{def:S2A5}, the surface $S_{2A_5}$ is constructed via gluing two copies $S_{A_5}$ where each $S_{A_5}$ is obtained via (1,3)-weighted blowup of $\mathbb{P}(1,1,3)$ at the point $[1,0,0]$. We recall that the fan of $\mathbb{P}(1,1,3)$ is generated in appropriate coordinates by primitive vectors $\mathbf{v}_1 = (0,1), \mathbf{v}_2 =(1,0)$ and $\mathbf{v}_3 =(-1,-3)$. The fan of the weighted blowup is obtained by forming the star subdivision associated to the (1,3)-weighted blowup, i.e. adding the vector $3\mathbf{v}_2 +\mathbf{v}_3$ (see e.g. \cite[10.3]{KM92}). This construction agrees with the description of the surface parameterized by $\xi_{3,3}$ in \cite{DH21}, see \cite[Remark 5.14 and Figure 9]{DH21} for the toric description. \end{proof}

\begin{lemma}\label{lem:D-lci-on-S2A5}
    If $[(S_{2A_5}, \frac{2}{3}D)] \in \ocM_4^{\CY}$, then $D$ is Cartier and has lci singularities.  
\end{lemma}

\begin{proof}
    Because $[(S_{2A_5},\frac{2}{3}D)] \in \ocM_4^{\CY}$, by definition it admits a smoothing to $(\bP^1 \times \bP^1, \frac{2}{3}D')$ where $D'$ is a $(3,3)$ curve.  By \cite[Lemma 2.11]{ABB+}, because $3K_{\bP^1 \times \bP^1} + 2D' \sim 0$, we have $3K_{S_{2A_5}}+2D \sim 0$ and hence $2D \sim -3K_{S_{2A_5}}$ where $\sim$ denotes linear equivalence.  From the explicit description of $S_{2A_5}$, $3K_{S_{2A_5}}$ is Cartier and hence $2D$ is Cartier.  As the local singularities of $S_{2A_5}$ are either normal crossing or of the form $(xy = 0 ) \subset \frac{1}{3}(1,1,2)$, any integral $\bQ$-Cartier divisor $\Delta$ on $S_{2A_5}$ satisfies $3\Delta$ is Cartier and therefore $3D$ is Cartier.  Because $3D$ is Cartier and $2D$ is Cartier, $D$ is Cartier.  Now, we verify that $D$ has lci singularities.  Because $[(S_{2A_5}, \frac{2}{3}D)] \in \ocM_4^{\CY}$, the pair $(S_{2A_5}, \frac{2}{3}D)$ has slc singularities and hence by \cite[Definition-Lemma 5.10]{Kol23}, the normalization $(S_{2A_5}^\nu, \frac{2}{3}D^\nu + \ell)$ has log canonical singularities.  This implies that $D$ does not contain the double locus of $S_{2A_5}$.  Write $S_{2A_5} =S_{A_5} \cup S_{A_5}$ and let $(S_{A_5}, \frac{2}{3}D_1 + \ell)$ be either component of the normalization.  Because $D$ is Cartier, $D_1$ is Cartier.  If $D_1$ passes through the $\frac{1}{3}(1,1)$ singularity, it necessarily has multiplicity at least three at that point, in which case $(S_{A_5}, \frac{2}{3}D_1 + \ell)$ is not log canonical.  Therefore, $D_1$ avoids the $\frac{1}{3}(1,1)$ singularity.  By restricting to one of the components, this implies that $D$ avoids the $\frac{1}{3}(1,1,2)$ singularities on $S_{2A_5}$.  In particular, $D$ is a Cartier divisor that only meets the normal crossing locus of $S_{2A_5}$ and therefore has lci singularities. 
\end{proof}

\begin{prop}\label{prop:eqnofcurveonS2A5}
Suppose $[(S, \frac{2}{3}D)] \in \ocM_4^{\KSBA}\setminus \ocM_4^{\K}$  admits a special degeneration to $(S_{2A_5}, \frac{2}{3}C_{2A_5}^\circ)$. Then $S\cong S_{2A_5}$ and $D$ is defined by the strict transform of the equation
\begin{align*}
(x_2^2 - x_0^3 x_2 + ax_0x_1^2 x_2 + b_0x_1^6 + b_1 x_0 x_1^5 +b_2 x_0^2 x_1^4 + b_3 x_0^3 x_1^3 = 0)\subset \bP(1,1,3);\\
(x_2'^2 - x_0'^3 x_2' + a'x_0'x_1'^2 x_2'+ b_0'x_1'^6 + b_1' x_0' x_1'^5 +b_2' x_0'^2 x_1'^4 + b_3' x_0'^3 x_1'^3 = 0)\subset \bP(1,1,3)
\end{align*}
on each component of $S_{2A_5}$.
Here $([a, b_0, b_1, b_2, b_3], [a', b_0', b_1', b_2', b_3'])\in \bP(2,6,5,4,3)^2$.
\end{prop}

\begin{proof}
By Lemma \ref{lem:S2A5-deform} we know that $S\cong S_{2A_5}$ or $S\cong \bP^1\times \bP^1$. The latter is not possible as otherwise $\bP^1\times \bP^1$ is K-polystable which implies $[(S,\frac{2}{3}D)] \in \ocM_4^{\K}$ by interpolation \cite[Proposition 2.13]{ADL19}, a contradiction. Since there is a symmetry between $S_{A_5}$ and $S_{A_5}'$, we may focus on the curve $D_1:=D|_{S_{A_5}}$. Denote by $C_1:=\pi_* D_1$ and $D_0:=C_{2A_5}^\circ |_{S_{A_5}}$. Then we know that $C_1$ specially degenerates to the curve $C_0:=(x_2^2 - x_0^3 x_2 =0)$. This degeneration is induced by a non-trivial power $\sigma^m$ of the $1$-PS $\sigma$ in $\Aut(S_{A_5}, \ell_1 + \frac{2}{3}D_1)$ such that $\pi_*\sigma$ acts on $\bP(1,1,3)$ as $\pi_*\sigma(t)\cdot [x_0, x_1, x_2] = [tx_0, x_1, t^3 x_2]$. Suppose $C_1 = (f(x_0, x_2) + x_1 g(x_0, x_1, x_2) = 0)$. We claim that $m>0$ and $g\neq 0$. If $g=0$ then clearly $C_1$ is $\pi_*\sigma$-invariant, which implies that $\Aut(S_{2A_5}, D)$ is infinite, a contradiction to KSBA stability.  It is clear that 
\[
C_0 = \lim_{t\to 0}\pi_* \sigma^m (t) \cdot C_1 = \begin{cases}(f(x_0, x_2)=0) & \textrm { if }m>0;\\
(x_1 h(x_0, x_1, x_2)=0) & \textrm{ if }m<0.
\end{cases}
\]
Here $h$ is the sum of terms in $g$ with the highest degree in  $x_1$. 
Since $C_0$ does not contain the line $(x_1=0)$, we must have $m>0$. 
Thus after rescaling we have $f(x_0, x_2) = x_2^2 - x_0^3 x_2$.

Moreover, if we denote by $v$ the monomial valuation centered at $[1,0,0]$ of weight $(1,3)$ in $(x_1, x_2)$, then the condition $D_1 \sim D_0$ implies that  $v(C_1) = v(C_0) = 3$. Thus the terms $x_0^5 x_1$ and $x_0^4 x_1^2$ do not appear in the equation of $C_1$. Moreover, using the affine transformation $[x_0, x_1, x_2]\mapsto [x_0 + rx_1, x_1, x_2+sx_1^3]$ we can eliminate the terms $x_0^2 x_1 x_2$ and $x_1^3 x_2$ in the equation of $C_1$. As a result, we get the desired form of $C_1$. Moreover, by projection to $\bP^1_{[x_0, x_1]}$ we know that $C_1$ is a hyperelliptic curve of arithmetic genus $2$. By taking the discriminant it is not hard to see that $C_1$ has at worst $A_4$-singularities if and only if $a, b_0, b_1,b_2,b_3$ are not all zero. Thus we know that $(S_{2A_5}, (\frac{2}{3}+\epsilon)D)$ is KSBA stable if and only if both vectors $(a, b_0, b_1, b_2, b_3)$ and $(a', b_0', b_1', b_2', b_3')$ are non-zero. Thus the proof is finished.
\end{proof}

We summarize the behavior of the bpCY wall crossing near $[(S_{2A_5}, \frac{2}{3}C^\circ_{2A_5})]$ in the following proposition.

\begin{prop}\label{prop:C2A5-replacement}
Consider the projective morphisms induced by the bpCY wall crossing
\[\oM^{\K}_4 \rightarrow \oM^{\CY}_4 \leftarrow \oM^{\KSBA}_4.\]
Let $[(S_{2A_5}, \frac{2}{3}C^\circ_{2A_5})]$ be a polystable point of $\oM_4^{\CY}$ from Proposition \ref{prop:S2A5polystable}.  The preimage of this pair in $\oM_4^{\K}$ is the pair $[(\mathbb{P}^1 \times \mathbb{P}^1, \frac{2}{3}C_{2A_5})]$ (see Lemma \ref{lem:C2A5-is-K-polystable}), and the morphism $\oM^{\K}_4 \to \oM^{\CY}_4$ is an isomorphism on the preimage of an open neighborhood of $[(S_{2A_5}, \frac{2}{3}C^\circ_{2A_5})]$.
\end{prop}

\begin{proof}
Firstly we claim that $x: = [(\mathbb{P}^1 \times \mathbb{P}^1, \frac{2}{3}C_{2A_5})]$ is a connected component in the preimage of $y:=[(S_{2A_5}, \frac{2}{3}C^\circ_{2A_5})]$. By openness of K-semistability and the Paul--Tian criterion \cite[Theorem 2.22]{ADL19}, there is an isomorphism between some open neighborhood of $x$ in $\oM_4^{\K}$ and some open neighborhood of the point $[C_{2A_5}]$ in the GIT moduli space of $(3,3)$-curves on $\bP^1\times \bP^1$. Then by \cite{Fed12} we know that there exists an open neighborhood $U$ of $x$ in $\oM_4^{\K}$ that $U\setminus \{x\}$ parameterizes  klt bpCY pairs of the form $(\bP^1\times \bP^1, \frac{2}{3}C)$ where  $C$ has at worst $A_4$-singularities. This implies that the image of $U\setminus \{x\}$ in $\oM_4^{\CY}$ does not contain $y$. Hence the claim is proved. 

Next, we consider a sufficiently small open neighborhood $V$ of $y$ in $\oM_4^{\CY}$ whose preimage in $\oM_4^{\K}$ is denoted by $U'$. By Proposition \ref{prop:deformations}, $V$ is normal. Since the induced map $U'\to V$ is birational and proper, by Zariski's Main Theorem we know that the preimage of $y$ in $U'$ is connected. Thus the claim implies that the preimage of $y$ in $U'$ is precisely $x$. Then after  shrinking $V$ we may assume that $U'\subset U$ and the map $U'\to V$ is birational and finite. Thus $U'\to V$ is an isomorphism by Zariski's main theorem. The proof is finished.
\end{proof}

\subsection{Replacing $D_4$} In the Chow moduli space, there is a unique polystable point with $D_4$ singularities by \cite[Remark 3.3]{CMJL12}: $C_D=V(x_0 x_3, x_1^3 + x_2^3)$.  We will consider this curve and its S-equivalent pairs in $\ocM_4^{\CY}$ and to do, we first fix notation.  We denote by $\bP^1 \times \bP^1$ a smooth quadric and by $\bP(1,1,2)$ a singular quadric cone in $\bP^3$ with coordinates $[x_0:x_1:x_2:x_3]$.  We denote by $\PuP := V(x_0x_3) \subset \bP^3$ the non-normal quadric surface that is the union of two planes.  

\begin{lem}\label{lem:quadric-deform}
A small deformation of a  (possibly singular) reduced quadric surface in $\bP^3$ is still a quadric surface.
\end{lem}

\begin{proof}
This follows directly from \cite[Examples 3.2.11]{Sernesi}.
\end{proof}

\begin{lemma}\label{lem:polystable-CD-pair}
    The pair $(\PuP, \frac{2}{3}C_D)$, depicted in Figure \ref{fig:C2A5andCD}, is polystable in $\ocM_4^{\CY}$. 
\end{lemma}

\begin{proof}
    By a straightfoward computation, $(\PuP, \frac{2}{3}C_D)$ is slc.  There is a $\bG_m^2$ action on $(\PuP, C_D)$ given by automorphisms of $\bP^3$ of the form $t \cdot [x_0:x_1:x_2:x_3] = [tx_0: x_1:x_2:x_3]$ and $s \cdot [x_0:x_1:x_2:x_3] = [x_0:x_1:x_2:sx_3]$ and hence by \cite[Proposition 8.11]{ABB+}, the pair is polystable. 
\end{proof}

\begin{lemma}\label{lem:Kpolystable-D4}
    Suppose $[(X,\frac{2}{3}D)] \in \ocM_4^{\K}$ admits a special degeneration to $(\PuP,\frac{2}{3} C_D)$ in $\ocM_4^{\CY}$.  Then, $X$ is a normal quadric surface and $D$ has at least one $D_4$ singularity in the smooth locus of $X$.  The K-polystable points of $\ocM_4^{\K}$ admitting such a degeneration are, up to coordinate change, given by $X =V(x_0x_3 + q(x_1, x_2)) \subset \bP^3$, where $q$ is a nonzero quadratic form, and $D = V(x_1^3 + x_2^3)|_X$.  The locus of such pairs has dimension $2$ in $\oM_4^{\K}$.
\end{lemma}

\begin{proof}
    Normality of $X$ follows as any pair $[(X,\frac{2}{3}D)] \in \ocM_4^{\K}$ satisfies $(X, (\frac{2}{3} - \epsilon)D)$ is klt, and $X$ must be a quadric surface by Lemma \ref{lem:quadric-deform}.  If $(X,\frac{2}{3}D)$ is klt, it is stable in $\ocM_4^{\CY}$ by \cite[Theorem 6.15]{ABB+} and does not admit nontrivial special degenerations, so therefore $(X,\frac{2}{3}D)$ must be strictly log canonical. As $D$ is deformation of $C_D$, this is possible only if $D$ has at least one $D_4$ singularity in the smooth locus of $X$.  Now, suppose $X$ is the smooth quadric surface $\bP^1 \times \bP^1$ with coordinates $[x:y] \times [u:v]$ and suppose $D \subset X$ is a $(3,3)$ curve with a $D_4$ singularity.  Up to coordinate change, we may assume the $D_4$ singularity is at the point $[1:0] \times [1:0]$ and $D$ has equation $y^3f_3(u,v) + xy^2vf_2(u,v) + x^2yv^2f_1(u,v) + x^3 v^3 f_0(u,v)$ where $f_i \in H^0(\bP^1, \cO(i))$.  Given any such $D$, it admits a special degeneration induced by the 1-parameter subgroup $t \cdot ([x,y] \times [u,v]) = [x,ty] \times [t^{-1}u, v]$ to a curve $D_0$ with equation $a_3y^3u^3 + a_2xy^2u^2v + a_1x^2yuv^2 + a_0 x^3v^3$ where $a_i$ is the coefficient of $v^{3-i}$ in $f_i$.  As these are the lowest degree terms with respect to the weight of $y$ and $v$, the condition that $D$ has a $D_4$ singularity is precisely that this equation has three distinct roots.  Therefore, choosing an appropriate embedding of $D_0$ into $\bP^3$, we find that $(X,D_0)$ is defined by the equations in the proposition.  A similar computation holds if $X$ is the singular quadric cone. 
    
    To complete the proof, it suffices to show that the pairs $(X,\frac{2}{3}D_0)$ as above are K-polystable.  By \cite[Section 2.1]{Fed12}, if $X$ is a smooth quadric surface, then $D_0$ is GIT polystable and the K-polystability of $(X,\frac{2}{3}D_0)$ follows from interpolation exactly as in the proof of Lemma \ref{lem:C2A5-is-K-polystable}.  If $X$ is the singular quadric cone, by \cite[Theorem 6.3(4)]{CMJL14}, $D_0$ is polystable in the VGIT moduli stack of $(2,3)$ complete intersections for $t \in (\tfrac{2}{9}, \tfrac{2}{5})$.  By \cite[Theorem 1.4]{ADL20}, $(X,cD_0)$ is therefore K-polystable for $c \in (\tfrac{1}{6}, \tfrac{1}{3})$ and by interpolation as in Lemma \ref{lem:C2A5-is-K-polystable}, we see that $(X,cD_0)$ is K-polystable for all $c \in (\tfrac{1}{6}, \tfrac{2}{3})$.

    Finally, we obtain the dimension count by observing that there is a $2$-dimensional choice of quadratic form $q(x_1,x_2)$ up to scaling.
\end{proof}

\begin{lemma}\label{lem:KSBA-D4}
    Suppose $[(X,\frac{2}{3}D)] \in \ocM_4^{\KSBA}\setminus \ocM_4^{\K}$ admits a special degeneration to $(\PuP, \frac{2}{3}C_D)$.  Then, $X \cong \PuP \subset \bP^3$ and $D = V(x_1^3 + x_2^3 + x_0f(x_0,x_1,x_2) +x_3g(x_1,x_2,x_3))|_X$, where $f,g$ are nonzero polynomials of degree $2$.  These parameterize curves on the union of two planes such that each component of the curve is a cubic plane curve with at worst $A_3$ singularities meeting the intersection line of the planes in three fixed points.  The locus of such pairs has codimension $3$ in $\oM_4^{\KSBA}$.
\end{lemma}

\begin{proof}
    By Lemma \ref{lem:quadric-deform}, $X$ must be a quadric surface.  Furthermore, by Proposition \ref{prop:klt-is-stable} and semicontinuity of log canonical thresholds, $X$ must be $\PuP$, otherwise $(X,\frac{2}{3}D) \in \ocM_4^{\K}$.  On each component of $\PuP$, the curve $D$ must be a cubic meeting the intersection line in three distinct points and it is straightforward to write $D$ in the above form.  Furthermore, the cubic plane curve on each component must have log canonical threshold strictly greater than $\frac{2}{3}$, and by classification of cubic curves can have at worst $A_3$ singularities. 
    
    Finally, up to automorphisms of each plane fixing the line of intersection, there is a 3-dimensional space of curves on each plane and therefore the dimension of such pairs in $\oM_4^{\KSBA}$ is 6 in the 9-dimensional space $\oM_4^{\KSBA}$. 
\end{proof}

Summarizing the content of Lemmas \ref{lem:polystable-CD-pair}, \ref{lem:Kpolystable-D4}, and \ref{lem:KSBA-D4}, we have the following proposition describing the bpCY wall crossing in a neighborhood of $[(\PuP, \frac{2}{3}C_D)]$. 

\begin{prop}\label{prop:D4replacement}
Consider the projective morphisms induced by the bpCY wall crossing
\[\oM^{\K}_4 \rightarrow \oM^{\CY}_4 \leftarrow \oM^{\KSBA}_4.\]
Let $[(\PuP, \frac{2}{3}C_D)]$ be the polystable point of $\oM_4^{\CY}$ from Lemma \ref{lem:polystable-CD-pair}.  The preimage of $[(\PuP, \frac{2}{3}C_D)]$ in $\oM_4^{\K}$ is the $2$-dimensional locus of polystable pairs in Lemma \ref{lem:Kpolystable-D4}. These curves have two $D_4$ singularities. The preimage of $[(\PuP, \frac{2}{3}C_D)]$ in $\oM^{\KSBA}_4$ is the codimension $3$ locus of pairs in Lemma \ref{lem:KSBA-D4}.  These curves are elliptic triboroughs, where each elliptic component has at worst $A_3$ singularities. 
\end{prop}

\subsection{Smoothness of CY moduli stack}

In this section, we prove that the stack $\ocM_4^{\CY}$ of bpCY pairs is smooth.  This will be used repeatedly in verification of the remaining walls in the Hassett--Keel program.  

The previous sections explicitly studies the polystable points $(S_{2A_5}, \frac{2}{3}C_{2A_5}^\circ)$ (Proposition \ref{prop:S2A5polystable}) and $(\PuP, \frac{2}{3}C_D)$ (Lemma \ref{lem:polystable-CD-pair}) of $\ocM_4^{\CY}$ corresponding to the curves $C_{2A_5}$ and $C_D$ of $\sX^{\rm c}$.  There is an additional polystable point of $\ocM_4^{\CY}$ corresponding to the curves $C_{A,B}$. 

Following Definition \ref{def:GIT-strictly-semistable-curves}, let $Q$ denote the singular quadric cone $V(x_2^2 - x_1x_3) \subset \bP^3$ and $C_{A,B}$ the curve $V(x_2^2 - x_1x_3, Ax_1^3+Bx_0x_1x_2+x_0^2x_3) \subset \bP^3$.  Via the identification $Q = \bP(1,1,2)$ and using weighted coordinates $[x:y:z]$ on $\bP(1,1,2)$, $C_{A,B}$ may be written as $V(Ax^6+Bx^3yz+y^2z^2) \subset \bP(1,1,2)$. Recall that, for $4A \ne B^2$, the curve $C_{A,B}$ has an $A_3$ singularity at the vertex of the cone and an $A_5$ singularity at the smooth point $[0:1:0]$ and observe that $\Aut(\bP(1,1,2), C_{A,B})$ contains $\bG_m$ as the automorphism $t \cdot [x:y:z] = [tx: y : t^3 z]$.

\begin{defn}
    Let $\bP(1,2,3) \cup \bP(1,1,3)$ be the surface $V(a_1a_2) \subset \bP(1,1,2,3)$ where the weighted coordinates on $\bP(1,1,2,3)$ are $[a_0:a_1:a_2:a_3]$.  Let $C_{A,B}^\circ$ be the curve $V(a_1a_2, Aa_0^6+Ba_0^2a_3+a_3^2)$.  For any $A,B$ such that $4A \ne B^2$, observe that by coordinate change we have $C_{A,B}^\circ \cong C_{1,0}^\circ = V(a_1a_2, a_0^6+a_3^2)$.
\end{defn}

\begin{lemma}
    If $4A \ne B^2$, the pairs $(Q, C_{A,B})$ admit a weakly special degeneration to the pair $(\bP(1,2,3) \cup \bP(1,1,3), C_{1,0}^\circ)$, where $C_{1,0}^\circ$ is defined above, and the point $[(\bP(1,2,3) \cup \bP(1,1,3), \frac{2}{3} C_{1,0}^\circ)]$ is a polystable point of $\ocM_4^{\CY}$.
\end{lemma}

\begin{proof}
    For fixed $A,B$, let $(\cQ, \cC):=(Q, C_{A,B}) \times \bA^1_t$ be the trivial family of pairs over $\bA^1$.  In the fiber $(t = 0)$, let $x,z$ be local coordinates near the $A_5$ singularity of $C_{A,B}$.  Perform a $(1,3,1)$ weighted blow up in the local coordinates $(x,z,t)$.  This yields a new family $(\cQ', \cC')$ over $\bA^1$ with central fiber $\cQ'_0 = \widetilde{Q}\cup \bP(1,1,3)$, where $\widetilde{Q}$ is the $(1,3)$-weighted blow-up of the point $[0:1:0] \in Q$.  Let $s_z \subset \widetilde{Q}$ be the strict transform of the section $(z = 0) \in Q$.  By direct computation, $s_z$ is a $-1$-curve in $\widetilde{Q}$ and extremal in $\cQ'$ with $K_{\cQ'} \cdot s_z = 0$ and $\cC' \cdot s_z = 0$.  In particular, we may contract $s_z$ in a small contraction to a new family $(\overline{\cQ}, \overline{\cC})$ over $\bA^1$, a weakly special degeneration of $(Q, C_{A,B})$ to $(\overline{\cQ}_0, \overline{\cC}_0)$. As the contraction of $s_z$ in $\widetilde{Q}$ is a birational morphism $\widetilde{Q} \to \bP(1,2,3)$, the central fiber $\overline{\cQ}_0 \cong \bP(1,2,3) \cup \bP(1,1,3)$.  From the equation of the strict transform of $C_{A,B}$, the curve $\overline{\cC}_0$ is precisely $C_{A,B}^\circ$.  As $C_{A,B}^\circ \cong C_{1,0}^\circ$, we have constructed the claimed weakly special degeneration.  Now, observe that $C_{1,0}^\circ$ has an $A_5$ singularity at a smooth point of $\bP(1,1,3)$, an $A_3$ singularity at the $\frac{1}{2}(1,1)$ singularity of $\bP(1,2,3)$, and meets the double locus of $\overline{\cQ}_0$ transversally.  In particular, the pair $(\overline{\cQ}_0, \tfrac{2}{3} C_{1,0}^\circ)$ has type II slc singularities and is a point of $\ocM_4^{\CY}$ and the normalization is non-dlt.  By \cite[Proposition 6.6]{ABB+}, $\dim \Aut(\overline{\cQ}_0, C_{1,0}^\circ) > \dim \Aut(Q, C_{A,B}) = 1$ and by \cite[Proposition 8.11]{ABB+}, $\dim \Aut(\overline{\cQ}_0, C_{1,0}^\circ) \le 2$, so  $\dim \Aut(\overline{\cQ}_0, C_{1,0}^\circ) = 2$ and  $(\overline{\cQ}_0, C_{1,0}^\circ)$ is polystable by \cite[Proposition 8.11]{ABB+}. 
\end{proof}

\begin{remark}\label{rmk:3-polystable-pts}
We have now exhibited three non-isomorphic type II polystable points of $\ocM_4^{\CY}$ with non-dlt normalization: $(S_{2A_5}, \frac{2}{3}C_{2A_5}^\circ)$, $(\PuP, \frac{2}{3}C_D)$, and $(\bP(1,2,3) \cup \bP(1,1,3), \frac{2}{3}C_{1,0}^\circ)$.  
\end{remark}

\begin{thm}\label{thm:bpCY-is-smooth}
The moduli stack $\ocM_4^{\CY}$ is smooth. Moreover, its good moduli space $\oM_4^{\CY}$ is a normal projective variety and isomorphic to the Baily--Borel compactification of Kond\={o}'s ball quotient from \cite{Kondog4}.
\end{thm}

\begin{proof}
Firstly, by \cite[Theorem 7.11]{BL24} we have an isomorphism between the normalization of $\oM_4^{\CY}$ and the Baily--Borel compactification $F^{\rm BB}$, where $F$ denotes Kond\={o}'s ball quotient $\cB/\Gamma$ from \cite{Kondog4}.  Any point of $\ocM_4^{\CY}$ is either of type I or type II by \cite[Proposition 8.18]{ABB+}.  Under this isomorphism, the type I pairs correspond to the interior points of $F^{\rm BB}$ and the type II pairs correspond to the cusps.  By \cite[Theorem 5.9]{CMJL12}, there are precisely three cusps in $F^{\rm BB}$, which then correspond to at most three polystable type II pairs in $\ocM_4^{\CY}$.  By Remark \ref{rmk:3-polystable-pts}, there are at least three such type II polystable points so the normalization $(\oM_4^{\CY})^{\nu} \to \oM_4^{\CY}$ must be bijective over these points.  However, these three pairs have unobstructed deformations by Proposition \ref{prop:deformations}, so $\ocM_4^{\CY}$ is smooth in a neighborhood of each point and hence $\oM_4^{\CY}$ is normal in a neighborhood of each point.  As all other points of $\oM_4^{\CY}$ must correspond to interior points of $F^{\rm BB}$, they are of type I and hence klt pairs, and again by Proposition \ref{prop:deformations} have unobstructed deformations.  Therefore, $\ocM_4^{\CY}$ is smooth at these points and $\oM_4^{\CY}$ is normal.  In particular, $\ocM_4^{\CY}$ is smooth at all points.  Because $\oM_4^{\CY}$ is normal, it is equal to its normalization and therefore isomorphic to $F^{\rm BB}$. 
\end{proof}

\subsection{From Calabi--Yau pairs to curves}

We use the results of the previous subsection to prove that the Hassett--Keel program in a neighborhood of $[C_D]$ and $[C_{2A_5}]$ can be obtained by bpCY wall crossing.  In particular, we will prove that there is a forgetful map $\ocM_4^{\CY} \dashrightarrow\Curves^{\lci+}_4$ that is an isomorphism over a dense open subset containing these points (see Theorem \ref{thm:CY-forget-open}).  

\begin{prop}\label{prop:open-substack-CY}
There exists a smooth open substack $\sW^{\CY}$ of $\ocM_4^{\CY}$ that parameterizes bpCY pairs $(X,\frac{2}{3}D)$ such that either $X\subset \bP^3$ is a quadric surface or $X\cong S_{2A_5}$. Moreover, every $\bk$-point $[(X,\frac{2}{3}D)]\in \sW^{\CY}(\bk)$ satisfies that $D$ is connected, reduced, with lci singularities, and  Cartier in $X$.
\end{prop}

\begin{proof}

To show openness of $\sW^{\CY}$, it suffices to show that the underlying surface being a quadric surface or isomorphic to $S_{2A_5}$ is open in families of bpCY pairs. Clearly this condition is constructible, so we only need to show that it is preserved under generalization. Let $f: (X, \frac{2}{3}D)\to B$ be a family of bpCY pairs in $\ocM_4^{\CY}$ over a pointed smooth curve $0\in B$ such that $X_0$ is either a quadric surface or $X_0\cong S_{2A_5}$. In the latter case, by Lemma \ref{lem:S2A5-deform} we know that after shrinking $B$,  every fiber $X_b$ is either isomorphic to $ S_{2A_5}$ or isomorphic to $\bP^1\times\bP^1$. By Proposition \ref{prop:deformations}, we know that $\sW^{\CY}$ is smooth.

The reducedness of $D$ follows from the fact that $(X, \frac{2}{3}D)$ is slc. The connectedness of $D$ follows from ampleness of $D$ in $X$. If $X$ is isomorphic to a quadric surface, then $D\sim \cO_X(3)$ is Cartier and hence $D$ has lci singularities. If $X\cong S_{2A_5}$, this is the content of Lemma \ref{lem:D-lci-on-S2A5}.
\end{proof}

The main goal of this subsection is to prove the following result on the forgetful map.

\begin{thm}\label{thm:CY-forget-open}
There is an open immersion of algebraic stacks $\Psi: \sW^{\CY} \to \Curves_4^{\lci +}$ induced by the forgetful map $[(X,D)]\mapsto [D]$. 
\end{thm}

\begin{prop}
The map $\Psi$ is well-defined as a morphism of algebraic stacks.
\end{prop}

\begin{proof}
Let $f:(X, \frac{2}{3}D)\to B$ be a family of bpCY pairs over a scheme $B$ from pulling back the universal family over $\sW^{\CY}$. By \cite[Remark 3.2(vii)]{ABB+}, we know that $D$ is flat proper over $B$ of relative pure dimension $1$. By Proposition \ref{prop:open-substack-CY}, we know that $D_b$ has lci singularities hence is Gorenstein for each $b\in B$, and $h^0(D_b, \cO_{D_b})=1$ as $D_b$ is reduced. Hence $b\to \chi(\cO_{D_b})$ is locally constant and equals $-3$ by irreducibility of $\sW^{\CY}$. Thus $h^1(D_b, \cO_{D_b})=4$. By Proposition \ref{prop:open-substack-CY}, we know that every fiber $D_b$ is connected, reduced, with lci singularities, in particular with finitely many singularities. This shows that $D\to B$ is a flat proper family of curves from pulling back the universal family over $\Curves_{4}^{\lci +}$.  The proof is finished.
\end{proof}

\begin{lem}\label{lem:curve-iso-surface}
Let $D$ and $D'$ be two curves on $S_{2A_5}$ such that both $(S_{2A_5}, \frac{2}{3} D)$ and $(S_{2A_5}, \frac{2}{3} D')$ belong to $\ocM_4^{\CY}$. If $D\cong D'$, then there exists $\sigma\in \Aut(S_{2A_5})$ such that $D' = \sigma^* D$.  
\end{lem}

\begin{proof}
Recall that $S_{2A_5}$ is the union of two copies of $S_{A_5}$, each the weighted blow-up of $\bP(1,1,3)$, glued along $\ell$ the strict transform of the ruling through the point blown up.  Denote the map by $\pi: S_{2A_5} \to \bP(1,1,3) \cup \bP(1,1,3)$ and the point blown up by $p$.  By definition of $\ocM_4^{\CY}$, the curve $D$ (respectively $D'$) is the union of two genus 2 hyperelliptic curves meeting at a non-Weierstress point $q$ (respectively, $q'$), noting that the curves $D_i, D_i'$ are indeed hyperelliptic as in Definition \ref{def:hyperelliptic} via the composition of $\pi$ and the projection map on $\bP(1,1,3)$.  Write $D = D_1 \cup D_2$ (respectively, $D' = D_1' \cup D_2'$) as the union of these hyperelliptic components.  

As $q$ (respectively, $q'$) is a non-Weierstrass point of each $D_i$ (respectively, $D_i'$), an isomorphism $D \cong D'$ must fix $q$ and $q'$ and be compatible with the hyperelliptic map $\phi: D_i \to \bP^1$  (respectively, $\phi': D_i' \to \bP^1$) and therefore fix the second point in the fiber $\phi^{-1}(q)$ on each component $D_i$.  Denote these points by $q_1, q_2$ (respectively, $q_1', q_2'$).  Let $\overline{D}$ (respectively $\overline{D'})$ be the image of $D$ (respectively $D'$) in the map $\pi: S_{2A_5} \to \bP(1,1,3) \cup \bP(1,1,3)$, which by construction glues the points $q_1$ and $q_2$ (respectively $q_1'$ and $q_2'$).  As any isomorphism $D \cong D'$ must fix the points $q_1, q_2$ (respectively, $q_1', q_2'$), is it equivalent to an isomorphism $\overline{D} \cong \overline{D'}$.  By Corollary \ref{cor:equiv-hyp}, as any two isomorphic genus 2 hyperelliptic curves differ by an automorphism of $\bP(1,1,3)$, it follows that the isomorphism $\overline{D} \cong \overline{D'}$ extends to an element $ \overline{\sigma} \in \Aut(\bP(1,1,3) \cup \bP(1,1,3) )$ fixing $p$ and $q$ (respectively, $q'$) such that $\overline{\sigma}^* \overline{D} = \overline{D'}$.  Then, if $\sigma = \pi \circ \overline{\sigma}$, the isomorphism $D \cong D'$ is realized by $D' = \sigma^* D$.  
\end{proof}

\begin{cor}\label{prop:CY-stab-preserve}
The morphism $\Psi$ is stabilizer preserving, in particular representable.
\end{cor}

\begin{proof}
It suffices to show that  $\Psi$ induces an isomorphism $\Aut(X,D)\xrightarrow{\cong }\Aut(D)$. If $X$ is isomorphic to a quadric surface $Q$, we know that $D$ is a complete $(2,3)$-intersection curve in $\bP^3$. Thus this follows from the fact that $Q$ is the unique quadric surface containing $D$ and that $H^0(\bP^3, \cO_{\bP^3}(1)) \to H^0(D, \cO_D(1))$ is an isomorphism. If $X$ is isomorphic to $S_{2A_5}$, by Lemma \ref{lem:curve-iso-surface}, an automorphism of $D$ can be lifted to a unique automorphism of $X$ fixing $\Delta_i$, the exceptional divisor of the weighted blow-up on each component $S_{A_5}$.  

By computation, $\Delta_i^2 = - \frac{1}{3}$ and $\ell^2 = 0$ and these curves generate the Mori cone of $S_{A_5}$.  As $\Delta_i$ is the unique curve with negative self-intersection, any automorphism of $S_{2A_5}$ must fix $\Delta_1 \cup \Delta_2$, and hence it follows that $\Aut(S_{2A_5}) = \Aut(S_{2A_5}, \Delta_1 \cup \Delta_2)$ so $\Aut(X,D) \cong \Aut(D)$. 
\end{proof}

\begin{prop}\label{prop:CY-forget-flat}
The morphism $\Psi$ is quasi-finite, flat, and separated.
\end{prop}

\begin{proof}
First, we show quasi-finiteness of $\Psi$. Since $\ocM_4^{\CY}$ is of finite type, we know that $\sW^{\CY}$ is quasi-compact which implies quasi-compactness of $\Psi$. By Proposition \ref{prop:CY-stab-preserve} we know that $\Psi$ is representable. Thus by \cite[\href{https://stacks.math.columbia.edu/tag/0G2M}{Tag 0G2M} and \href{https://stacks.math.columbia.edu/tag/06UA}{Tag 06UA}]{stacksproject} it suffices to show that the induced map $|\Psi|:|\sW^{\CY}|\to |\Curves_{4}^{\lci +}|$ on topological spaces has finite fibers. Indeed, we will show that $|\Psi|$ is injective. 

Let $[(X,\frac{2}{3}D)]\in |\sW^{\CY}|$. Then we will show that $|\Psi^{-1}(D)|$ contains only one point. Let $[(X',\frac{2}{3}D')]\in |\Psi^{-1}(D)|$, i.e. $D\cong D'$. If $\omega_D$ is basepoint free, then Lemma \ref{lem:CY-bpf} implies that both $X$ and $X'$ are isomorphic to quadric surfaces and $D$ is a $(2,3)$-complete intersection curve in $\bP^3$. Since the embedding of $D\hookrightarrow \bP^3$ is induced by $|\omega_D|$, it is unique up to projective transformations. Thus $D$ is contained in a unique quadric surface, which implies that $(X,D)\cong (X',D')$. If $\omega_D$ is not basepoint free, then Lemma \ref{lem:CY-bpf} implies that $X\cong X'\cong S_{2A_5}$. Thus the result follows from Lemma \ref{lem:curve-iso-surface}.

Next, we show flatness of $\Psi$ by applying the miracle flatness criterion. Since flatness is smooth local on the source-and-target, it suffices to show that the base change of $\Psi$ to a smooth cover of $\Curves_4^{\lci +}$ by algebraic spaces is flat, and we may assume that each connected component of the cover is of finite type as $\Curves_4^{\lci +}$ is locally of finite type. Let $U \to \Curves_{4}^{\lci +}$ be a smooth morphism from an algebraic space $U$ of finite type. Let $V:= U\times_{\Curves_{4}^{\lci +}} \sW^{\CY}$. Then $V$ is an algebraic space by representability of $\Psi$ from Proposition \ref{prop:CY-stab-preserve}. Since both algebraic stacks $\sW^{\CY}$ and $\Curves_{4}^{\lci +}$ are smooth by Proposition \ref{prop:open-substack-CY} and \cite[\href{https://stacks.math.columbia.edu/tag/0DZT}{Tag 0DZT}]{stacksproject}, we know that both $U$ and $V$ are smooth algebraic spaces of finite type. For simplicity, we may assume that $U$ is connected and $V$ is non-empty. Thus $U$ dominates  $\Curves_{4}^{\lci +}$ by irreducibility of $\Curves_{4}^{\lci +}$ (see Section \ref{sec:stack-curves}). Since $\Psi: \sW^{\CY}\to \Curves_4^{\lci +}$ is birational, we know that $\Psi_U: V \to U$ is also birational between smooth irreducible algebraic spaces. We know that $\Psi_U$ is quasi-finite as $\Psi$ is quasi-finite. Thus $\Psi_U$ is flat by miracle flatness \cite[\href{https://stacks.math.columbia.edu/tag/00R4}{Tag 00R4}]{stacksproject}. Therefore, the morphism $\Psi$ is flat. 

Finally, we show separatedness of $\Psi$ by the valuative criterion. Suppose $(R,\fm)$ is a DVR essentially of finite type over $\bk$ with residue field $\kappa := R/\fm$ and fraction field $K := \mathrm{Frac}(R)$. Let $f_1, f_2: \Spec R\to \sW^{\rm CY}$ be two morphisms such that $\Psi\circ f_1 = \Psi \circ f_2: \Spec R\to \Curves_4^{\lci +}$. It suffices to show $f_1=f_2$. By pulling back the universal families over $\ocM_4^{\CY}$ through $f_1$ and $f_2$, we obtain families $(X_1, \frac{2}{3}D_1)$ and $(X_2, \frac{2}{3}D_2)$ of bpCY pairs over $\Spec R$ such that $D_1\cong D_2$ as $R$-schemes. Equivalently, it suffices to show $(X_1, D_1)\cong (X_2, D_2)$ over $R$. Below we split into two cases.

If $\omega_{D_{i, \kappa}}$ is basepoint free for every $i\in \{1,2\}$, by Lemma \ref{lem:CY-bpf} we know that $X_{i,\kappa}$ is a quadric surface for every $i$. By Lemma \ref{lem:quadric-deform} we know that $X_{i, K}$ is also a quadric surface for every $i$. In particular, the line bundle $\omega_{X_i}(D_i)$ is very ample over $R$ for each $i$ and induces a closed embedding $h_i: X_i\hookrightarrow \bP^3_R$. Since $D_1\cong D_2$ and $h_i|_{D_i}$ is induced by $\omega_{D_i}$, there exists $g\in \PGL(4,R)$ such that $D_1 = g^* D_2$. Since every $(2,3)$-complete intersection curve in $\bP^3$ is contained in a unique quadric surface, we conclude that $X_1= g^* X_2$, which implies $(X_1, D_1) \cong (X_2, D_2)$ over $R$. 

If $\omega_{D_{i, \kappa}}$ is not basepoint free for every $i\in \{1,2\}$, by Lemma \ref{lem:CY-bpf} we know that $X_{i,\kappa}\cong S_{2A_5}$ for every $i$. Since the stack $\ocM_4^{\CY}$ is S-complete  by Theorem \ref{thm:CY-gms}, we know that there exists $f: \STR \to \ocM_4^{\CY}$ such that $f_1 = f|_{s\neq 0}$ and $f_2 = f|_{t\neq 0}$. Observe that the $S_{2A_5}$-locus in $\ocM_4^{\CY}$ is closed: as $\ocM_4^{\CY}$ contains only type I and II pairs (c.f. \cite[Proposition 8.18]{ABB+}), any pair $(S_{2A_5}, \frac{2}{3}C)$ must be of type II and hence can only admit a specialization to a type II pair $(X,\frac{2}{3}D)$ in $\ocM_4^{\CY}$ by \cite[Proposition 8.8]{ABB+}.  The only type II polystable pairs in $\ocM_4^{\CY}$ are given in Remark \ref{rmk:3-polystable-pts}, and $S_{2A_5}$ cannot degenerate to either $\PuP$ or $\bP(1,2,3) \cup \bP(1,1,3)$ as their double loci are normal crossing away from at most one point.  Therefore, $S_{2A_5}$ can only specialize to itself and hence the $S_{2A_5}$ locus is closed.  This implies that the the image of $f$ is contained in $\sW^{\CY}$. 
By assumption we know that $\Psi\circ f$ factorizes into $\STR \to \Spec R\to \Curves_4^{\lci +}$. Since $\Psi$ is representable by Proposition \ref{prop:CY-stab-preserve}, it is also DM by \cite[\href{https://stacks.math.columbia.edu/tag/050E}{Tag 050E}]{stacksproject} which implies that $f$ also factorizes into $\STR \to \Spec R\to \sW^{\CY}$. Thus we conclude that $f_1=f_2$. This finishes the proof of separatedness of $\Psi$.
\end{proof}

The following lemma was used in the previous proof.

\begin{lem}\label{lem:CY-bpf}
Let $[(X,\frac{2}{3}D)]\in \sW^{\CY}(\bk)$. Then $\omega_D$ is basepoint free if and only if $X$ is isomorphic to a quadric surface in $\bP^3$.
\end{lem}

\begin{proof}
If $X$ is a quadric surface, then $D\sim \cO_X(3)$ which implies that $\omega_D\cong \cO_D(1)$ is very ample and hence basepoint free. If $X\cong S_{2A_5}$, then we claim that $\omega_D$ is not basepoint free. Indeed, a general small deformation $D'$ of $D$ in $S_{2A_5}$ is obtained by gluing two smooth genus $2$ curves $D_1'$ and $D_2'$ at non-Weierstrass points $p_1\in D_1'$ and $p_2\in D_2'$. Thus $\omega_{D'}|_{D_i'}\cong \omega_{D_i'}(p_i)$ is not basepoint-free as $p_i$ is not a Weierstrass point. Since basepoint-freeness is an open condition given that $h^0$ of the canonical bundle is constant in this family, we know that $\omega_D$ is not basepoint free. 
\end{proof}

\begin{proof}[Proof of Theorem \ref{thm:CY-forget-open}] Since open immersion is smooth local on the source-and-target, by similar argument to the proof of Proposition \ref{prop:CY-forget-flat} it suffices to show that $\Psi_U: V\to U$ is an open immersion where $U \to \Curves_4^{\lci +}$ is a smooth morphism from a connected algebraic space $U$ of finite type, where $V = U\times_{\Curves_4^{\lci +}} \sW^{\CY}$. Moreover, by Proposition \ref{prop:CY-forget-flat} we know that both $U$ and $V$ are smooth irreducible algebraic spaces of finite type, and $\Psi_U$ is  quasi-finite, flat, separated, and birational. By Zariski's main theorem for algebraic spaces \cite[\href{https://stacks.math.columbia.edu/tag/082K}{Tag 082K}]{stacksproject}, we conclude that $\Psi_U$ is an open immersion. Thus the proof is finished. 
\end{proof}

\section{New models for the Hassett--Keel program}\label{sec:newmodelsforHK}

\subsection{New stacks and good moduli spaces}
\begin{defn}\label{def:sWK}
Recall that $\sW^{\CY}$ is the open substack of $\ocM^{\CY}_4$ parametrizing bpCY pairs $(X, \frac{2}{3}D)$ such that $X$ is either a quadric surface or $S_{2A_5}$. Let $\sW:=\Psi(\sW^{\CY})$ be the corresponding open substack of $\Curves_4^{\Gor}$ by Theorem \ref{thm:CY-forget-open}. Denote two open substacks of $\sW^{\CY}$ by $\sW^{\K}:= \sW^{\CY}\cap \ocM_4^{\K}$ and $\sW^{\KSBA}:=\sW^{\CY}\cap \ocM_4^{\KSBA}$. Denote by their images under $\Psi$ by $\sW^-:=\Psi(\sW^{\K})$ and $\sW^+:=\Psi(\sW^{\KSBA})$ as open substacks of $\sW$.
\end{defn}

\begin{defn}\label{def:V1minus-V2}
Recall from Definition \ref{def:sUsV} that $\phi: \sX\to \fX^{\rm c}$ is a good moduli space morphism.
Define $\sV_1^-:= \phi^{-1}(\fX^{\rm c} \setminus (\Gamma\cup \{[C_{D}]\}))$ and $\sV_2:=\phi^{-1}(\fX^{\rm c} \setminus (\Gamma\cup \{[C_{2A_5}]\}))$ as open substacks of $\sX$. 
\end{defn}

Since $\sV = \phi^{-1}(\fX^{\rm c} \setminus \Gamma)$, we have $\sV  = \sV_1^- \cup \sV_2$. Moreover, $\sV_1^- \cap \sV_2 =\sU\cap \sV=\sX^{\rm s}$.  By Theorem \ref{thm:CMJL-Chow-GIT} and Remark \ref{rmk:GIT-stable-properties}, we have $\sV \subset \sW$.

\begin{remark}\label{rmk:GITstable-in-K-and-KSBA}
    By Theorem \ref{thm:CMJL-Chow-GIT}, every curve in the Chow stable stack $\sX^{\rm s}$ lies on a normal quadric surface $Q$ such that $(Q, \tfrac{2}{3}C)$ is klt.  In particular, we have $\sX^{\rm s} \subset \sW$ and the preimage $\Psi^{-1}(\sX^{\rm s})$ is contained in each of $\ocM_4^{\K}$, $\ocM_4^{\CY}$, and $\ocM_4^{\KSBA}$ by Proposition \ref{prop:klt-is-stable}.  
\end{remark}

\begin{prop}
$\sV_1^-$ is isomorphic to a saturated open substack of the K-moduli stack $\ocM_4^{\K}$. 
\end{prop}

\begin{proof}
As $\sV_1^-$ is contained in $\sV$, it is contained in $\sW$.  Define $\sV_1^{\K}$ to be the preimage of $\sV_1^-$ under the forgetful map $\Psi: \sW^{\CY} \to \sW$.  By Theorem \ref{thm:CMJL-Chow-GIT}, if $[(X,\frac{2}{3}D)] \in \sV_1^{\K}$, then $X$ is a quadric of rank $\ge 3$ and $(X,\frac{2}{3}D)$ is klt for any $D$ with $\phi([D]) \ne [C_{2A_5}]$.  In particular, $(X,(\frac{2}{3}-\epsilon)D)$ is K-stable if $\phi([D]) \ne [C_{2A_5}]$ by Proposition \ref{prop:klt-is-stable}.  By Lemma \ref{lem:C2A5-is-K-polystable}, the polystable point $[C_{2A_5}] \in \sV^-$ has K-polystable preimage $[(\bP^1 \times \bP^1, \frac{2}{3}C_{2A_5})] \in \sW^{\K}$ and thus any $(X,\frac{2}{3}D)$ with $\phi([D]) = [C_{2A_5}]$ is K-semistable by openness of K-semistablity.  Hence, $\sV_1^{\K} \subset \sW^{\K}$.  As $\Psi$ is an open immersion, $\sV_1^- \cong \sV_1^{\K}$ an open substack of $\ocM_4^{\K}$. Finally, by Lemma \ref{lem:quadric-deform} and the polystability of $(\bP^1 \times \bP^1, \frac{2}{3}C_{2A_5})$ from Lemma \ref{lem:C2A5-is-K-polystable}, this substack is saturated.    
\end{proof}

\begin{definition}\label{def:sV1andsV1+}
Consider the bpCY wall crossing in Theorem \ref{thm:bpCYwallcrossing}: 
        \[
\begin{tikzcd}
\ocM_4^{\K} \arrow[d] \arrow[r,hook]  &
\ocM_{4}^{\CY} \arrow[d] & 
\ocM_{4}^{\KSBA} \arrow[d] \arrow[l,hook']  \\
\oM_4^{\K} \arrow[r]  & 
\oM_4^{\CY}  & 
\oM_4^{\KSBA} \arrow[l]
\end{tikzcd}
,\]
where the top row represents the moduli stacks and the bottom row the corresponding good moduli spaces.  As $\sV_1^{\K}$ is saturated, the induced morphism $\sV_1^{\K}$ to its image in the good moduli space morphism $\ocM_4^{\K} \to \oM_4^{\K}$ is a good moduli space. Denote the image by $V_1^{\K}$.

Define $\sV_1^{\CY} \subset \ocM_4^{\CY}$ as the saturation of $\sV_1^{\K}$, i.e. the preimage under the good moduli space morphism $\ocM_4^{\CY} \to \oM_4^{\CY}$ of the image of $V_1^{\K}$ in $\oM_4^{\CY}$, which is open by Proposition \ref{prop:C2A5-replacement}.  Define $\sV_1^{\KSBA}$ to be $\sV_1^{\CY} \cap \ocM_4^{\KSBA}$.  These are saturated open substacks of $\ocM_4^{\CY}$ and $\ocM_4^{\KSBA}$, respectively. By Proposition \ref{prop:C2A5-replacement} we know that $\sV_1^{\CY}\hookrightarrow \sW^{\CY}$ is an open substack. Finally, define $\sV_1 := \Psi(\sV_1^{\CY})$ and $\sV_1^+ := \Psi(\sV_1^{\KSBA})$ which are both open substacks of $\sW$ by Theorem \ref{thm:CY-forget-open}.
\end{definition}

\begin{prop}\label{prop:gms-for-V1}
We have a diagram $\sV_1^{-}\hookrightarrow \sV_1 \hookleftarrow \sV_1^+$ of open immersions that descends to good moduli spaces
\[
\fV_1^{-}\xrightarrow{\cong} \fV_1 \xleftarrow{\psi_1^+} \fV_1^+,
\]
where $\psi_1^+$ is a proper birational morphism.  Furthermore, the exceptional locus of $\psi_1^+$ is a $\bQ$-Cartier prime divisor birational to $\delta_2$ that parameterizes curves as described in Proposition \ref{prop:eqnofcurveonS2A5} which is contracted to the point $[C_{2A_5}^\circ]$ by $\psi_1^+$.
\end{prop}

\begin{proof}
    From Theorem \ref{thm:bpCYwallcrossing}, there is a diagram of $\ocM_4^{\K}\hookrightarrow \ocM_4^{\CY} \hookleftarrow \ocM_4^{\KSBA}$ of open immersions that descends to good moduli spaces
\[
\oM_4^{\K}\rightarrow \oM_4^{\CY} \leftarrow \oM_4^{\KSBA}.
\]  This induces the corresponding diagram of the saturated open substacks and their good moduli spaces 
        \[
\begin{tikzcd}
\sV_1^{\K} \arrow[d] \arrow[r,hook]  &
\sV_1^{\CY} \arrow[d] & 
\sV_1^{\KSBA} \arrow[d] \arrow[l,hook']  \\
V_1^{\K} \arrow[r]  & 
V_1^{\CY}  & 
V_1^{\KSBA} \arrow[l]
\end{tikzcd}.
\]
By the isomorphisms $\sV_1^{\K} \cong \sV_1^-$, $\sV_1^{\CY} \cong \sV_1$, and  $\sV_1^{\KSBA} \cong \sV_1^+$ induced by the forgetful map $\Psi: \sW^{\CY} \to \sW$, we have the existence of good moduli spaces and the diagram as claimed.  It then suffices to verify the properties of the morphisms on good moduli spaces.  By Theorem \ref{thm:bpCY-is-smooth}, $\sV_1^{-}$, $\sV_1$, and $\sV_1^{+}$ are smooth and hence the good moduli spaces $\fV_1^-$, $\fV_1$, and $\fV_1^+$ are normal.  By Theorem \ref{thm:CMJL-Chow-GIT}, the morphism $\psi_1^-: \fV_1^- \to \fV_1$ is birational and an isomorphism on $\fV_1^- - \{[C_{2A_5}]\}$, hence finite, and therefore an isomorphism by Zariski's Main Theorem. The morphism $\psi_1^+: \fV_1^+ \to \fV_1$ is an isomorphism away from the preimage of the complement of $\{[C_{2A_5}^\circ]\}$, and the exceptional locus of $\psi_1^+$ is as described in Proposition \ref{prop:eqnofcurveonS2A5}.  From the description in the proposition, the curves parameterized by the exceptional locus are unions of two genus two hyperelliptic curves glued at a non-Weierstrass point.  As the forgetful map $\Psi: \sW^{\CY} \to \sW$ is an isomorphism, the exceptional locus is birational to $\delta_2$ and hence of codimension one.  Next, the moduli space $\oM_4^{\KSBA}$ has quotient singularities as it is the coarse space of the smooth Deligne--Mumford stack $\ocM_4^{\KSBA}$ and hence $\bQ$-factorial by \cite[Proposition 5.15]{KM98}.  This implies that the open subset $\fV_1^+ \subset \oM_4^{\KSBA}$ is $\bQ$-factorial and therefore the codimension one exceptional locus of $\psi_1^+$ is a $\bQ$-Cartier divisor.  In fact, the exceptional locus of $\psi_1^+$ is a prime divisor because it is dominated by the irreducible variety $\bP(2,6,5,4,3)^2$ by Proposition \ref{prop:eqnofcurveonS2A5}. 
Finally, the morphism $\fV_1^+ \to \fV_1$ is induced by base change of the proper morphism $\oM_4^{\KSBA} \to \oM_4^{\CY}$ (see Theorem \ref{thm:bpCYwallcrossing}) and is therefore proper. 
\end{proof}

\begin{prop}
$\sV_2$ is isomorphic to a saturated open substack of the bpCY moduli stack $\ocM_4^{\CY}$. Moreover, $\sV_2$ is smooth.
\end{prop}

\begin{proof}
By definition, $\sV_2 \subset \sV$ so $\sV_2 \subset \sW$.  Let $\sV_2^{\CY}$ be its isomorphic preimage under the forgetful map $\Psi: \sW^{\CY} \to \sW$.  By Lemma \ref{lem:quadric-deform}, being contained in a quadric is an open condition, so this is an open substack of $\ocM_4^{\CY}$.  Furthermore, by Remark \ref{rmk:GITstable-in-K-and-KSBA}, all points of $\sV_2$ such that $\phi([D]) \ne [C_D]$ have polystable preimage in $\sV_2^{\CY}$.  If $[D]$ is a point of $\sV_2$ such that $\phi([D]) = [C_D]$, then the corresponding preimage $[(X,\frac{2}{3}D)] \in \sV_2^{\CY}$ admits a special degeneration to the polystable point $[(\PuP, \frac{2}{3}C_D)]$ as described in Proposition \ref{prop:D4replacement}, which is a point of $\sV_2^{\CY}$.  By Lemma \ref{lem:quadric-deform}, $\sV_2^{\CY}$ is saturated in $\ocM_4^{\CY}$.  Smoothness follows from Theorem \ref{thm:bpCY-is-smooth}. 
\end{proof}

\begin{defn}\label{def:V2+}
Let $\sV_2^{\KSBA} := \sV_2^{\CY} \cap \ocM_4^{\KSBA}$ be the corresponding saturated open substack in the KSBA moduli stack $\ocM_4^{\KSBA}$.  Let $\sV_2^+ := \Psi(\sV_2^{\KSBA})$ be its isomorphic image under the forgetful map.
\end{defn}

\begin{prop}\label{prop:gms-for-V2} The open immersion $\sV_2^+\hookrightarrow \sV_2$ descends to a morphism of good moduli spaces $\psi_2^+:\fV_2^+ \to \fV_2$ which is a proper small birational morphism contracting the locus of curves in Lemma \ref{lem:KSBA-D4} to the point $[C_D]$.
\end{prop}

\begin{proof}
    The existence of the good moduli spaces and morphism $\psi_2^+: \fV_2^+ \to \fV_2$ follows from the isomorphisms of $\sV_2$ and $\sV_2^+$ with saturated open substacks of $\ocM_4^{\CY}$ and $\ocM_4^{\KSBA}$.  This morphism is birational because it is an isomorphism on the dense klt locus and proper because it is constructed as the base change of the proper morphism $\oM_4^{\KSBA} \to \oM_4^{\CY}$.  Finally, the exceptional locus of $\psi_2^+$ is as described in Lemma \ref{lem:KSBA-D4} parameterizing elliptic triboroughs which get mapped to the single point $[(V(x_0x_3), \frac{2}{3}C_D)]$ under $\psi_2^+$, which has codimension three.  
\end{proof}

\begin{defn}\label{def:newmodels}
For $\alpha \in [\frac{5}{9}, \frac{2}{3})$, we define 
\[
\ocM_4(\alpha):= \begin{cases}
    \sX & \textrm{ if }\alpha = \frac{5}{9};\\
    \sU^+\cup \sV_1^{-} \cup \sV_2^+ & \textrm{ if }\alpha\in (\frac{5}{9},\frac{19}{29});\\
    \sU^+\cup \sV_1 \cup \sV_2^+ & \textrm{ if }\alpha = \frac{19}{29};\\
    \sU^+\cup \sV_1^{+} \cup \sV_2^+ & \textrm{ if }\alpha\in (\frac{19}{29},\frac{2}{3}).
\end{cases}
\]
A proper curve $C$ over $\bk'$ is called \emph{$\alpha$-stable} if $C$ is isomorphic to a curve represented by a $\bk'$-point in  $\ocM_4(\alpha)$. 
\end{defn}

\begin{thm}\label{thm:HK-stack}
For each $\alpha \in [\frac{5}{9}, \frac{2}{3})$, the stack $\ocM_4(\alpha)$ is smooth, algebraic, of finite type over $\bk$ with affine diagonal. There exists an open immersion $\ocM_4(\alpha) \hookrightarrow \Curves_4^{\lci}$. Moreover, every $\alpha$-stable curve has an ample canonical bundle.
\end{thm}

\begin{proof}
From the definitions above we know that for each $\alpha \in [\frac{5}{9}, \frac{2}{3})$, the stack $\ocM_4(\alpha)$ admits an open immersion into $\sX \cup \sW$. By Theorems \ref{thm:sX-gms} and \ref{thm:CY-forget-open}, we know that both $\sX$ and $\sW$ are smooth open substacks of $\Curves_{4}^{\lci}$ of finite type over $\bk$ with affine diagonal, which implies the first and second statements on $\ocM_4(\alpha)$. For the last statement, notice that if $[C]\in \sX$ then $C$ has ample and basepoint free canonical bundle by definition; if $[C]\in \sW$ then there exists a surface $X$ containing $C$ as a Cartier divisor such that $[(X, \frac{2}{3}C)]\in \sW^{\CY}\subset \ocM_4^{\CY}$, which implies that $\omega_C^{\otimes 2} \sim_{\bQ} 2(K_X + C)|_C \sim -K_X|_C$ is ample by adjunction. 
\end{proof}

\begin{thm}\label{thm:new-models}
For each $\alpha \in [\frac{5}{9}, \frac{2}{3})$, the stack $\ocM_4(\alpha)$ admits a proper good moduli space which we denote by $\ofM_4(\alpha)$. Moreover, we have a  diagram of open immersions of stacks
\[
\ocM_4(\tfrac{5}{9})\hookleftarrow\ocM_4(\tfrac{5}{9},\tfrac{19}{29}) \hookrightarrow \ocM_4(\tfrac{19}{29}) \hookleftarrow \ocM_4(\tfrac{19}{29},\tfrac{2}{3}),
\]
which induces a wall crossing diagram of birational morphisms:
\[
\fX^{\rm c}\cong \ofM_4(\tfrac{5}{9})\xleftarrow{j_5^+}\ofM_4(\tfrac{5}{9},\tfrac{19}{29}) \xrightarrow[\cong]{j_4^-} \ofM_4(\tfrac{19}{29}) \xleftarrow{j_4^+} \ofM_4(\tfrac{19}{29},\tfrac{2}{3}). 
\]
The wall crossing morphisms have the following properties. 
\begin{enumerate}
    \item The exceptional locus of $j_5^+$ parameterizes elliptic triboroughs as described in the second half of Proposition \ref{prop:D4replacement}, which get mapped to the point $[C_D]$ in $\ofM_4(\tfrac{5}{9})$, the hyperelliptic curves parameterized by $\fH_4^{+}$  which get mapped to $[R_{\hyp}]$ in $\ofM_4(\tfrac{5}{9})$, and the curves in Proposition \ref{prop:A3-cone} with $h_4(x,y) \ne 0$ which get mapped to the appropriate point $[C_{A,B}]$ in $\ofM_4(\tfrac{5}{9})$.
    \item The isomorphism $j_4^{-}$ identifies the point  $[C_{2A_5}]$ with the point  $[C_{2A_5}^\circ]$.
    \item The exceptional locus of $j_4^+$ is a $\bQ$-Cartier prime divisor as a birational transform of $\delta_2$ and parameterizes unions of genus 2 hyperelliptic curves glued at a non-Weierstrass point with at worst $A_4$ singularities as described in Proposition \ref{prop:eqnofcurveonS2A5}, which get mapped to the point $[C_{2A_5}^\circ]$ by $j_4^+$.
\end{enumerate}

\end{thm}

\begin{proof}
When $\alpha = \frac{5}{9}$, the existence of a proper good moduli space isomorphic to $\fX^{\rm c}$ follows from Theorem \ref{thm:sX-gms}. Thus we shall focus on showing the existence of proper good moduli spaces for $\alpha \in (\frac{5}{9}, \frac{2}{3})$.
By Propositions \ref{prop:gms-for-V1}, \ref{prop:gms-for-V2} and Corollary \ref{cor:U+gms}, each open substack $\sU^+$, $\sV_1^-$, $\sV_1$, $\sV_1^+$, $\sV_2$, and $\sV_2^+$ of $\Curves_4^{\Gor}$ admits a good moduli space.  Furthermore, their pairwise intersections are the Chow stable locus $\sX^{\rm s}$, which is saturated, and thus the induced good moduli spaces agree on their intersections.  Therefore, we may construct a good moduli space $\ofM_4(\alpha)$ of their union by gluing the associated good moduli space morphisms by Theorem \ref{thm:gms-glue}.  

 The diagram on open immersions of stacks $\ocM_4(\alpha)$ follows directly from Definition \ref{def:newmodels} and definitions of relevant stacks $\sX$, $\sU^+$, $\sV_1$, $\sV_1^{\pm}$, $\sV_2$, and $\sV_2^+$. By Theorem \ref{thm:gms-universal} this descends to a diagram for good moduli spaces $\ofM_4(\alpha)$. 

Next, we show properness of $\ofM_4(\alpha)$, which will follow by properness of $\fX^{\rm c}$ and properness of the wall crossing morphisms in Theorem \ref{thm:X+gms}, Proposition \ref{prop:gms-for-V1}, and Proposition \ref{prop:gms-for-V2} and two general remarks on properness: 

\begin{enumerate}[label=(\roman*)]
    \item If $Y$ is proper and $X \to Y$ is a proper morphism, then $X$ is proper \cite[\href{https://stacks.math.columbia.edu/tag/01W3}{Tag 01W3}]{stacksproject}.
    \item If $X$ is proper and $X \to Y$ is a proper surjective morphism such that $Y$ is locally of finite type, then $Y$ is proper \cite[\href{https://stacks.math.columbia.edu/tag/03GN}{Tag 03GN}, \href{https://stacks.math.columbia.edu/tag/09MQ}{Tag 09MQ}]{stacksproject}. 
\end{enumerate}
 The wall crossing morphisms for $\fV_i^{\pm}$ are proper by Propositions \ref{prop:gms-for-V1}, \ref{prop:gms-for-V2}, and the wall crossing morphism $\fU^+ \to \fU$ is proper as it is the base change of the proper morphism $\fX^+ \to \fX^{\rm c}$ from the proof of Theorem \ref{thm:X+gms}.  These  morphisms are surjective on $\bk$-points by construction, so surjective by \cite[\href{https://stacks.math.columbia.edu/tag/0487}{Tag 0487}]{stacksproject}.  The good moduli spaces $\fV_1^-$, $\fV_1$, $\fV_1^+$, $\fV_2$, and $\fV_2^+$ are  of finite type as they are isomorphic to subspaces of the projective varieties $\oM_4^{\K}$, $\oM_4^{\CY}$, and $\oM_4^{\KSBA}$ and $\fU$, $\fU^+$ are  of finite type as they are contained in the proper good moduli spaces $\fX^{\rm c}$, $\fX^+$.  Therefore, as each $\ofM_4(\alpha)$ is obtained by wall crossing from the proper moduli space $\fX^{\rm c}$, each $\ofM_4(\alpha)$ is proper.  

 The description of the exceptional loci follows from Propositions \ref{prop:gms-for-V1}, \ref{prop:gms-for-V2} and Corollary \ref{cor:U+toU}.
\end{proof}

\begin{thm}\label{thm:AFS-agree}
If $\alpha\in (\frac{19}{29}, \frac{2}{3})$, then the stack $\ocM_4(\alpha)$ (resp. the good moduli space $\ofM_4(\alpha)$) is isomorphic to the stack  of $\ocM_4(\frac{2}{3}-\epsilon)$ (resp. the good moduli space $\oM_4(\frac{2}{3}-\epsilon)$) defined in \cite{AFSvdW}.  In particular, the wall crossing diagram in Theorem \ref{thm:new-models} can be extended to 
\[
\fX^{\rm c}\cong \oM_4(\tfrac{5}{9})\xleftarrow{j_5^+}\ofM_4(\tfrac{5}{9},\tfrac{19}{29}) \xrightarrow[\cong]{j_4^-} \ofM_4(\tfrac{19}{29}) \xleftarrow{j_4^+} \ofM_4(\tfrac{19}{29},\tfrac{2}{3})\cong \oM_4(\tfrac{2}{3}-\epsilon)\xrightarrow{j_3^-} \oM_4(\tfrac{2}{3}). 
\]
\end{thm}

\begin{proof}
Recall that $\ocM_4(\alpha) = \sU^+ \cup \sV_1^+ \cup \sV_2^+$.  We will show that there is an open immersion $\ocM_4(\alpha) \hookrightarrow \ocM_4(\frac{2}{3} - \epsilon)$ that takes closed points to closed points and apply \cite[Proposition 6.4]{alper} to prove the isomorphism of stacks.

First, consider $\sU^+$.  By Propositions \ref{prop:U+open} and \ref{prop:U+closed-preserve}, $\sU^+$ admits an open immersion to $\ocM_4(\frac{2}{3} - \epsilon)$ that takes closed points to closed points.

Now, consider $\sV_1^+ \cup \sV_2^+$.  By definition, $\sV_1^+ \cup \sV_2^+$ is isomorphic to the open substack $\sV_1^{\KSBA} \cup \sV_2^{\KSBA} \subset \sW^{\KSBA}$.  We claim that the forgetful map $\Psi: \sW^{\KSBA} \to \Curves_4^{\lci+}$ has image contained in $\ocM_4(\frac{2}{3}-\epsilon)$.  Indeed, from Definition \ref{def:HKsing}, it suffices to verify that the curves $D$ in the pairs $[(X,D)] \in \sW^{\KSBA}$ have the required properties.  This holds for all $D$ in the Chow stable locus $\sX^{\rm s}$ by Remark \ref{rmk:GIT-stable-properties}.  If $(X,D)$ satisfies that $D\not\in \sX^{\rm s}$, then $X \cong \PuP$ or $S_{2A_5}$ and the pair admits a special degeneration to $(\PuP, C_D)$ or $(S_{2A_5}, C_{2A_5}^\circ)$, respectively.  Such pairs also have images contained in $\ocM_4(\frac{2}{3} - \epsilon)$: the singularity condition holds by Proposition \ref{prop:eqnofcurveonS2A5} and Proposition \ref{prop:D4replacement}, and the conditions on tails and chains hold from the geometry of $\PuP$ and $S_{2A_5}$.  Indeed, on $\PuP$, the curve $D = D_1 \cup D_2$ is a union of two arithmetic genus one components (one on each copy of $\bP^2$) meeting transversally at three distinct points, so by construction cannot contain an irreducible genus two curve and thus can contain no Weierstrass chains.  An elliptic tail is necessarily irreducible, so if it occurred on $\PuP$, would have to be either $D_1$ or $D_2$, but then would meet the rest of the curve at three points and thus not a tail.  Similarly, an elliptic chain could not be equal to $D_1$ or $D_2$, so would necessarily be the union of two rational curves meeting in an $A_3$ singularity, with one rational curve on each component $\bP^2$.  This contradicts the components $D_1$ and $D_2$ meeting transversally.  Therefore, there are no elliptic or Weierstrass chains or tails on $\PuP$.  
On $S_{2A_5}$, a curve $D = C_1 \cup C_2$ where each $C_i$ is an arithmetic genus 2 curve and the intersection $C_1 \cap C_2$ is a non-Weierstrass point of each $C_i$.  By construction, each $C_i$ is the strict transform of a curve of degree 6 on $\bP(1,1,3)$ not passing through the singular point, and this implies that if $C_i$ is reducible, it is the union of at most two arithmetic genus $0$ components (yet their union meets $C_{3-i}$ at only one point).  Therefore, there can be no elliptic or Weierstrass chains or tails.  In particular, the forgetful map $\Psi: \sW^{\KSBA} \to \Curves_4^{\lci+}$ has image contained in $\ocM_4(\frac{2}{3}-\epsilon)$. 

Furthermore, on $\sW^{\KSBA}$, the forgetful map is an open immersion by Theorem \ref{thm:CY-forget-open} and therefore the inclusion $\sV_1^+ \cup \sV_2^+ \to \ocM_4(\frac{2}{3}- \epsilon)$ is an open immersion.  To verify it takes closed points to closed points, observe that $\sV_1^+ \cap \sV_2^+$ is contained in $\sU^+$ and any point of $\sV_1^+ \cap \sV_2^+$ is stable and its closure in $\sV_1^+ \cup \sV_2^+$ is equivalent to its closure in $\sU^+$.  These points are then closed in $\ocM_4(\frac{2}{3}-\epsilon)$ by Proposition \ref{prop:U+closed-preserve}.  The remaining curves in $\sV_1^+ \cup \sV_2^+$ are either curves on $\PuP$ or $S_{2A_5}$.  The curves on $\PuP$ have at worst $A_3$ singularities, hence are closed in the DM stack $\ocM_4(\frac{2}{3} + \epsilon)$ (c.f. Lemma \ref{lem:2/3+DM}).  As these curves are contained in $\ocM_4(\frac{2}{3} - \epsilon)$, by local VGIT they are closed in $\ocM_4(\frac{2}{3} - \epsilon)$.  For curves on $S_{2A_5}$, write $C = C_1 \cup C_2$ as the union of the arithmetic genus 2 curves on each component of $S_{2A_5}$.  By Proposition \ref{prop:eqnofcurveonS2A5}, $C$ has at worst $A_4$ singularities.  If both components have at worst $A_3$ singularities, then the same argument applies to show $[C]$ is closed in $\ocM_4(\frac{2}{3} + \epsilon)$.  Assume then that $C_1$ has an $A_4$ singularity.  By construction, $C_1$ is an $\alpha_{\frac{2}{3}}$ atom.  The second component $C_2$ is either an $\alpha_{\frac{2}{3}}$ atom if it also has an $A_4$ singularity, or is a curve in $\ocM_{2,1}(\frac{2}{3}+\epsilon)$.  In the latter case, by Remark \ref{rmk:M21-is-DM}, $[C_2]$ is closed in $\ocM_{2,1}(\frac{2}{3}+\epsilon)$.  Therefore, in either case, $C$ is closed in $\ocM_4(\frac{2}{3} - \epsilon)$ by \cite[Definition 2.21, Theorem 2.22]{AFSvdW}.  

Finally, by gluing together $\sU^+$ and $\sV_1^+ \cup \sV_2^+$ along the Chow stable locus $\sX^{\rm s}$, we conclude the existence of an open immersion $\ocM_4(\alpha) \to \ocM_4(\frac{2}{3} - \epsilon)$ sending closed points to closed points. This implies that the induced map between the proper good moduli spaces is proper and injective and hence finite, so by \cite[Proposition 6.4]{alper}, $\ocM_4(\alpha) \to \ocM_4(\frac{2}{3} - \epsilon)$ is finite.  As it was an open immersion, we conclude $\ocM_4(\alpha) \cong \ocM_4(\frac{2}{3} - \epsilon)$ by Zariski's Main Theorem. 
\end{proof}

\subsection{Projectivity}

For $\alpha\in (\frac{5}{9}, \frac{2}{3})$, denote by $\varphi_\alpha: \ocM_4 \dashrightarrow \ofM_4(\alpha)$ the birational contraction. Denote by $\lambda_\alpha:=(\varphi_\alpha)_* \lambda$ and $\delta_{i,\alpha}:=(\varphi_\alpha)_* \delta_i$.   From the construction, we know that $\varphi_\alpha$ always contracts $\delta_1$, and it contracts $\delta_2$ if and only if $\alpha\in (\frac{5}{9}, \frac{19}{29}]$. 

\begin{defn}
For each $\alpha\in (\frac{5}{9}, \frac{2}{3})\cap \bQ$, we define a $\bQ$-divisor $\fL_\alpha$ on $\ofM_{4}(\alpha)$ as 
\[
\fL_\alpha:=\begin{cases}
13\lambda_\alpha - (2-\alpha)\delta_{0, \alpha} & \textrm{ if }\alpha \in (\frac{5}{9}, \frac{19}{29}];\\
13\lambda_\alpha - (2-\alpha)(\delta_{0, \alpha}+ \delta_{2, \alpha}) & \textrm{ if }\alpha \in (\frac{19}{29}, \frac{2}{3}).
\end{cases}
\]
\end{defn}

\begin{thm}\label{thm:projectivity}
\begin{enumerate}
    \item For  any $\alpha \in (\frac{5}{9}, \frac{2}{3})\cap \bQ$, $\fL_\alpha$ is $\bQ$-Cartier and ample on $\ofM_4(\alpha)$. 
    \item For any $\alpha \in [\frac{5}{9}, \frac{2}{3})\cap \bQ$, we have
\begin{equation} \label{eq:proj1}
\ofM_4(\alpha) \cong \oM_4(\alpha) := \Proj R(\ocM_4, K_{\ocM_4}+\alpha \delta).
\end{equation}
\end{enumerate} 
\end{thm}

\begin{prop}\label{prop:proj1}
Theorem \ref{thm:projectivity}(1) holds if $\alpha\in (\frac{19}{29}, \frac{2}{3})\cap \bQ$.
\end{prop}

\begin{proof}

By \cite{AFS17b} we know that there exists a sufficiently small $\epsilon_0\in \bQ_{>0}$ such that $\fL_{\frac{2}{3}-\epsilon}$ is $\bQ$-Cartier and ample for every $\epsilon\in (0, \epsilon_0]\cap \bQ$. Thus we may assume that $\alpha\in (\frac{19}{29}, \frac{2}{3}-\epsilon_0)\cap\bQ$.  It suffices to show $\fL_{\frac{19}{29}}^+:=13\lambda_\alpha - (2-\frac{19}{29})(\delta_{0,\alpha} + \delta_{2, \alpha})$ is $\bQ$-Cartier and nef on $\ofM_4(\alpha)$, as $\fL_\alpha$ is a strictly convex combination of $\fL_{\frac{2}{3}-\epsilon_0}$ and $\fL_{\frac{19}{29}}^+$.

Next, we adopt the notation from \cite{ADL20} of $\eta$ and $\xi$ on the VGIT quotients (see also Section \ref{sec:chow-vgit}). By \cite[Proof of Theorem 7.1]{CMJL14} the pull-back of these line bundles to $\ocM_4$ are given by
\begin{equation}\label{eq:CMJL-proj}
\phi_t^*(4s\eta + 4\xi) = (34s-33)\lambda - (4s-4) \delta_0 - (14s-15)\delta_1 - (18s-21) \delta_2. 
\end{equation}
Here $\phi_t: \ocM_4\dashrightarrow \fX_t = W^{\rm ss}(N_t)\sslash \PGL(4)$ is the birational contraction and $s=\frac{1}{t}$. Denote by 
\[\phi^{\rm c}_\alpha:= \phi_{\frac{2}{3}} \circ \varphi_\alpha^{-1} = j_5^+\circ (j_4^-)^{-1} \circ j_4^+ : \ofM_4(\alpha) \to \fX^{\rm c}.
\]
Then by Theorem \ref{thm:new-models} we know that $\phi_\alpha^{\rm c}$ is a birational morphism.
Since $\frac{3}{2}\eta + \xi = \frac{3}{2} (\eta + \frac{2}{3}\xi)$ is a positive multiple of the GIT polarization on $\fX^{\rm c}$, we know that 
\begin{equation}\label{eq:CMJL-proj2}
(\phi_\alpha^{\rm c})^* ( 4\cdot \tfrac{3}{2}\eta + 4\xi) = 18\lambda_\alpha - 2\delta_{0,\alpha} - 6 \delta_{2, \alpha} = 2(9\lambda_\alpha - \delta_{0,\alpha} - 3 \delta_{2, \alpha})
\end{equation}
is $\bQ$-Cartier, big and nef on $\ofM_4(\alpha)$. On the other hand, by \cite{AFS17b} we know that $\fL_{\frac{2}{3}}^{-} := 13\lambda_\alpha - (2-\frac{2}{3}) (\delta_{0,\alpha}+\delta_{2,\alpha})$ is $\bQ$-Cartier, big and nef on $\ofM_4(\alpha)$. Thus
\begin{equation}\label{eq:proj2}
\frac{1}{3}(9\lambda_\alpha - \delta_{0,\alpha} - 3 \delta_{2, \alpha}) + 2 \fL_{\frac{2}{3}}^{-} = 29 \lambda_\alpha - 3\delta_{0,\alpha} - \frac{11}{3}\delta_{2, \alpha} = \frac{29}{13}\fL_{\frac{19}{29}}^+ -\frac{2}{3}\delta_{2,\alpha}
\end{equation}
is $\bQ$-Cartier, big and nef on  $\ofM_4(\alpha)$. Since $\delta_{2,\alpha}$ is $\bQ$-Cartier by Theorem \ref{thm:new-models}(3), so is $\fL_{\frac{19}{29}}^+$. Thus to show nefness of $\fL_{\frac{19}{29}}^+$, it suffices to show nefness of its restriction $\fL_{\frac{19}{29}}^+|_{\delta_{2,\alpha}}$. Indeed, we will show that this restriction is trivial. 

For any $t\in (0, \frac{2}{3})$, consider the birational map $\phi_{t}\circ\varphi_{\alpha}^{-1}:\ofM_4(\alpha) \dashrightarrow \fX_t$ as the composition of the birational morphism $\phi_{\alpha}^{\rm c}$ and the birational map $\fX^{\rm c} \dashrightarrow \fX_t$. Since the latter map is an isomorphism near $[C_{2A_5}]$, we know that $\phi_t\circ \varphi_{\alpha}^{-1}$ is regular near $\delta_{2,\alpha}$ which contracts $\delta_{2,\alpha}$ to a point. Thus by \eqref{eq:CMJL-proj} 
\[
\left.\left((\phi_{t}\circ\varphi_{\alpha}^{-1})^* (4s\eta + 4\xi)\right)\right|_{\delta_{2,\alpha}} = \left.\left( (34s-33)\lambda_\alpha - (4s-4) \delta_{0,\alpha} - (18s-21) \delta_{2,\alpha} \right)\right|_{\delta_{2,\alpha}}
\]
is trivial for any $t\in (0, \frac{2}{3})$. Since the above divisor is linear in $s$, we know that it is trivial for any $s\in \bR$. In particular, we can choose $s=\frac{17}{14}$ which yields
\begin{align*}
0 & = \left.\left( (34\cdot \tfrac{17}{14}-33)\lambda_\alpha - (4\cdot\tfrac{17}{14}-4) \delta_{0,\alpha} - (18\cdot\tfrac{17}{14}-21) \delta_{2,\alpha} \right)\right|_{\delta_{2,\alpha}} \\
& = \left.\left(\tfrac{58}{7}\lambda_\alpha  - \tfrac{6}{7}(\delta_{0,\alpha}+\delta_{2,\alpha})\right)\right|_{\delta_{2,\alpha}} 
  = \tfrac{58}{91} \left.\fL_{\frac{19}{29}}^+\right|_{\delta_{2,\alpha}}. 
\end{align*}
This shows triviality of  $\fL_{\frac{19}{29}}^+|_{\delta_{2,\alpha}}$ and concludes the proof.
\end{proof}

\begin{prop}\label{prop:proj2}
Theorem \ref{thm:projectivity}(1) holds if $\alpha = \frac{19}{29}$. 
\end{prop}

\begin{proof}
We first show $\bQ$-Cartierness of $\fL_{\frac{19}{29}}$.
Pick some $\alpha'\in (\frac{19}{29}, \frac{2}{3})\cap \bQ$. Then we have the birational morphism $j_4^+:\ofM_4(\alpha') \to \ofM_4(\frac{19}{29})$ which contracts $\delta_{2,\alpha'}$ to a point $p$ and is isomorphic elsewhere. From the proof of Proposition \ref{prop:proj1} we know that $\delta_{2,\alpha'}$ and the following two $\bQ$-divisors \[
\fL_{\frac{2}{3}}^{-} = 13\lambda_{\alpha'} - \tfrac{4}{3}(\delta_{0, \alpha'} + \delta_{2, \alpha'})  \quad \textrm{and}\quad\fL_{\frac{19}{29}}^{+} = 13\lambda_{\alpha'} - \tfrac{39}{29}(\delta_{0, \alpha'} + \delta_{2, \alpha'})
\]
are all $\bQ$-Cartier. Thus we have that $\lambda_{\alpha'}$ and $\delta_{0, \alpha'}$ are both $\bQ$-Cartier. This implies that both $\lambda_{\alpha}$ and $\delta_{0, \alpha}$ are  $\bQ$-Cartier away from $p$. 

Similar to the proof of Proposition \ref{prop:proj1}, for each $t\in (0, \frac{2}{3})$ the birational map $\phi_t\circ \varphi_{\alpha}^{-1}: \ofM_4(\frac{19}{29})\dashrightarrow \fX_t$ is an isomorphism near $p$ which sends $p$ to $[C_{2A_5}]$. Thus 
\[
(\phi_{t}\circ\varphi_{\alpha}^{-1})^* (4s\eta + 4\xi) =  (34s-33)\lambda_\alpha - (4s-4) \delta_{0,\alpha} 
\]
is $\bQ$-Cartier near $p$ for every $t\in (0, \frac{2}{3})$. Thus  both $\lambda_{\alpha}$ and $\delta_{0, \alpha}$ are  $\bQ$-Cartier near $p$. This concludes the proof of $\bQ$-Cartierness of $\fL_{\alpha}$.

Next, we show nefness of $\fL_{\frac{19}{29}}$. Clearly, $\fL_{\frac{19}{29}}^+ - (j_4^+)^*\fL_{\frac{19}{29}}$ only supports on $\delta_{2,\alpha'}$ whose restriction to $\delta_{2,\alpha'}$ is trivial. Hence the negativity lemma \cite[Lemma 3.39]{KM98} implies  
\begin{equation}\label{eq:proj4}
    \fL_{\frac{19}{29}}^+ = (j_4^+)^*\fL_{\frac{19}{29}}.
\end{equation} 
Thus the nefness of $\fL_{\frac{19}{29}}$ follows from the nefness of $\fL_{\frac{19}{29}}^+$ as shown in the proof of Proposition \ref{prop:proj1}.

Finally, we show ampleness of $\fL_{\frac{19}{29}}$.  By the Nakai-Moishezon criterion, it suffices to show that for every positive dimensional  closed integral algebraic subspace $Z\subset \ofM_4(\alpha)$, the restriction $\fL_{\frac{19}{29}}|_Z$ is big. Since $Z$ is not contained in the exceptional locus of $j_4^+$ which is a single point $p$, the birational transform $Z'$ of $Z$ under $j_4^+$ is a closed integral subvariety of $\ofM_4(\alpha')$. It suffices to show that $(j_4^+)^* \fL_{\frac{19}{29}}|_{Z'} = \fL_{\frac{19}{29}}^+|_{Z'}$ is big. Similar to \eqref{eq:proj2}, for $0<\epsilon\ll 1$ we can pick positive numbers $a := \frac{1-87\epsilon}{3(1-9\epsilon)}$ and $b :=\frac{2}{1-9\epsilon}$ such that 
\begin{equation}\label{eq:proj3}
    a (9\lambda_{\alpha'} - \delta_{0, \alpha'} - 3\delta_{2,\alpha'}) + b \fL_{\frac{2}{3}-\epsilon} = \tfrac{29}{13}\fL_{\frac{19}{29}}^+ - 2a\delta_{2,\alpha'}.
\end{equation}
From the proof of Proposition \ref{prop:proj1} we know that $9\lambda_{\alpha'} - \delta_{0, \alpha'} - 3\delta_{2,\alpha'}$ (resp. $\fL_{\frac{2}{3}-\epsilon}$) is nef (resp. ample). Thus \eqref{eq:proj3} implies that $\fL_{\frac{19}{29}}^+ - \frac{26}{29}a\delta_{2,\alpha'}$ is ample. As a result, the restriction $\fL_{\frac{19}{29}}^+|_{Z'} - \frac{26}{29}a\delta_{2,\alpha'}|_{Z'}$ is also ample. Since $Z'$ is not contained in $\delta_{2,\alpha'}$ and $a>0$, the $\bQ$-divisor $\frac{26}{29}a\delta_{2,\alpha'}|_{Z'}$ is effective. Therefore, $\fL_{\frac{19}{29}}^+|_{Z'}$ is ample plus effective and hence big. This finishes the proof of ampleness.
\end{proof}

\begin{prop}\label{prop:proj3}
Theorem \ref{thm:projectivity}(1) holds if $\alpha\in (\frac{5}{9},\frac{19}{29})\cap \bQ$.
\end{prop}

\begin{proof}
We define two $\bQ$-divisors on $\ofM_4(\alpha)$ as 
\[
\fL_{\frac{19}{29}}^{-}:=(j_4^{-})^*\fL_{\frac{19}{29}} =13\lambda_{\alpha} - (2-\tfrac{19}{29})\delta_{0, \alpha} \quad \textrm{and} \quad\fL_{\frac{5}{9}}^{+}:=13\lambda_{\alpha} - (2-\tfrac{5}{9})\delta_{0, \alpha}. 
\]
Since $j_4^-$ is an isomorphism by Theorem \ref{thm:new-models}, Proposition \ref{prop:proj2} implies that $\fL_{\frac{19}{29}}^{-}$ is $\bQ$-Cartier and ample. 
Similar to the proofs of Proposition \ref{prop:proj1} and \ref{prop:proj2},  
\[
\fL_{\frac{5}{9}}^+ = \tfrac{13}{18}(j_5^+)^* (4\cdot \tfrac{3}{2}\eta + 4\xi)
\]
is $\bQ$-Cartier, big and nef. Since $\fL_\alpha$ is a strict convex combination of  $\fL_{\frac{19}{29}}^{-}$ and $\fL_{\frac{5}{9}}^+$, we conclude that it is $\bQ$-Cartier and ample. 
\end{proof}

\begin{proof}[Proof of Theorem \ref{thm:projectivity}]
Part (1) follows directly from the combination of Propositions \ref{prop:proj1}, \ref{prop:proj2}, and \ref{prop:proj3}. Thus we focus on part (2) for the rest of the proof. If $\alpha = \frac{5}{9}$, then (2) follows from Theorems \ref{thm:cmjl-vgit} and \ref{thm:new-models}. Thus we may assume $\alpha \in (\frac{5}{9}, \frac{2}{3})\cap \bQ$ from now on.

We start from the case where $\alpha \in (\frac{19}{29}, \frac{2}{3})$.
From \cite{AFSvdW} and Theorem \ref{thm:AFS-agree} we know that the birational map $\oM_4^{\rm ps} \dashrightarrow \oM_4(\frac{2}{3}-\epsilon)\cong \ofM_4(\alpha)$ is an isomorphism in codimension $1$ and at  the center of $\delta_1$ under the divisorial contraction $\oM_4 \to \oM_4^{\rm ps}$. Thus by \cite[Lemma 4.1]{HH09} we have
\[
K_{\ocM_4}+ \alpha \delta = \varphi_\alpha^* \fL_\alpha + (9-11\alpha) \delta_1.
\]
Here we use the canonical bundle formula $K_{\ocM_4} = 13\lambda - 2\delta$ (see e.g. \cite[Section 3.E]{HM98} or \cite[1.45]{Morrison}). Since $\alpha <\frac{9}{11}$ by assumption and $\delta_1$ is $\varphi_\alpha$-exceptional, \eqref{eq:proj1} follows from the ampleness of $\fL_\alpha$. 

Finally we study the case where $\alpha\in (\frac{5}{9}, \frac{19}{29}]$. 
Pick some $\alpha'\in (\frac{19}{29}, \frac{2}{3})$. According to Theorem \ref{thm:new-models}, let $j: \ofM_4(\alpha')\to \ofM_4(\alpha)$ be the birational morphism that only contracts $\delta_{2,\alpha'}$ and is isomorphic elsewhere (in fact, $j$ is $j_4^+$ if $\alpha = \frac{19}{29}$ or $(j_4^{-})^{-1}\circ j_4^{+}$ otherwise). Notice that $\phi_{\alpha}^{\rm c}:\ofM_{4}(\alpha) \to \fX^{\rm c}$ is a birational morphism by Theorem \ref{thm:new-models}. Thus \eqref{eq:CMJL-proj2} implies that 
\begin{equation}\label{eq:proj-pf1}
9\lambda_{\alpha'} - \delta_{0, \alpha'} - 3\delta_{2, \alpha'} =  (\phi_{\alpha'}^{\rm c})^*(3\lambda+2\xi) = j^* (\phi_\alpha^{\rm c})^* (3\lambda + 2\xi) = 9j^* \lambda_\alpha - j^* \delta_{0,\alpha}.
\end{equation}
Moreover, by \eqref{eq:proj4} we know that 
\begin{equation}\label{eq:proj-pf2}
   29\lambda_{\alpha'} - 3\delta_{0, \alpha'} - 3\delta_{2, \alpha'} = 29j^*\lambda_\alpha  - 3 j^* \delta_{0,\alpha}. 
\end{equation}
Combining \eqref{eq:proj-pf1} and \eqref{eq:proj-pf2} yields
\[
j^* \lambda_\alpha = \lambda_{\alpha'} + 3\delta_{2,\alpha'}, \qquad j^* \delta_{0, \alpha} = \delta_{0, \alpha'} + 30 \delta_{2, \alpha'}.
\]
Therefore, 
\begin{align*}
j^*\fL_{\alpha}&  = j^*(13\lambda_\alpha- (2-\alpha)\delta_{0, \alpha})\\
& = 13\lambda_{\alpha'} - (2-\alpha)\delta_{0, \alpha'} +(39 - 30(2-\alpha))\delta_{2, \alpha'} \\
& = 13\lambda_{\alpha'} - (2-\alpha) (\delta_{0, \alpha'} + \delta_{2, \alpha'}) - (19-29\alpha) \delta_{2, \alpha'}. 
\end{align*}
Since the center of $\delta_1$ on $\ofM_4(\alpha')$ generically parameterizes irreducible curves with an $A_2$-singularity, it is not contained in $\delta_{2,\alpha'}$ which implies $\varphi_{\alpha'}^*\delta_{2, \alpha'} = \delta_2$.  Again by \cite[Lemma 4.1]{HH09} we have 
\begin{align*}
K_{\ocM_4} + \alpha \delta & = \varphi_{\alpha'}^*(13\lambda_{\alpha'} - (2-\alpha) (\delta_{0, \alpha'} + \delta_{2, \alpha'})) + (9-11\alpha) \delta_1\\
& = \varphi_{\alpha'}^*j^*\fL_{\alpha} + (19-29\alpha)\varphi_{\alpha'}^*\delta_{2, \alpha'} + (9-11\alpha) \delta_1\\
& = \varphi_{\alpha}^*\fL_{\alpha} + (19-29\alpha)\delta_{2} + (9-11\alpha) \delta_1.
\end{align*}
Since $\alpha <\frac{19}{29}<\frac{9}{11}$ by assumption and $\delta_1$, $\delta_2$ are both $\varphi_\alpha$-exceptional, \eqref{eq:proj1} follows from the ampleness of $\fL_\alpha$. 
\end{proof}

\subsection{Proofs of Theorems \ref{introthm:59to23} and \ref{introthm:fullHKprogram}}

\begin{proof}[Proof of Theorem \ref{introthm:59to23}]
    The properties of the stacks $\ocM_4(\alpha)$ follows directly from Theorem \ref{thm:HK-stack}.
    The existence of projective good moduli spaces and all parts of the diagram except the right-most square follow from Theorems \ref{thm:new-models} and \ref{thm:projectivity}.  The right-most square follows from Proposition \ref{prop:59-ewall}.
\end{proof}

\begin{proof}[Proof of Theorem \ref{introthm:fullHKprogram}]
    This follows from Theorem \ref{introthm:59to23}, Theorem \ref{thm:AFS-agree}, and the previous works \cite{Fed12, CMJL12, CMJL14} for the Hassett--Keel models with $\alpha < \frac{5}{9}$ and the works \cite{HH09, HH13, AFS17} for the Hassett--Keel models with $\alpha > \frac{2}{3} - \epsilon$.  
\end{proof}

\newpage

\appendix

\section{Table of Stacks}

Throughout the paper, many different stacks are defined.  The following table summarizes the stacks and their relationships to one another. 

\renewcommand{\arraystretch}{1.5}

\begin{center}
\begin{longtable}{| p{.15\textwidth} | p{.5\textwidth}| p{.25\textwidth} |}
    \hline \textbf{Stack} & \textbf{Description of $\bk$-points} & \textbf{Reference}  \\ \hline
    \hline  $\cM_g$ & smooth genus $g$ curves & \\
    \hline $\ocM_g$ & Deligne--Mumford stable genus $g$ curves & \\
    \hline $\ocM_g(\alpha)$ for $\alpha > \frac{2}{3} - \epsilon$ & $\alpha$-stable genus $g$ curves &  Definition \ref{def:HKsing}  \\
    \hline $\Curves$ & proper schemes of dimension $\leq 1$ over $\bk$ & See \cite[\href{https://stacks.math.columbia.edu/tag/0D4Z}{Tag 0D4Y} and \href{https://stacks.math.columbia.edu/tag/0DMJ}{Tag 0DMJ}]{stacksproject}  \\
    \hline $\Curves_g^\Gor$ & Gorenstein proper curves $C$ with (arithmetic) genus $g$ (i.e. $h^0(C, \cO_C)=1$ \& $h^1(C, \cO_C) = g$) & See \cite[\href{https://stacks.math.columbia.edu/tag/0E1L}{Tag 0E1L} and \href{https://stacks.math.columbia.edu/tag/0E6K}{Tag 0E6K}]{stacksproject} \\
    \hline $\Curves_g^{\lci}$ & proper curves with local complete intersection (lci) singularities of genus $g$ & See \cite[\href{https://stacks.math.columbia.edu/tag/0E0J}{Tag 0E0J}]{stacksproject} \\ 
    \hline $\Curves_g^{\lci+}$ & proper lci curves of genus $g$ with finite singular locus & See \cite[\href{https://stacks.math.columbia.edu/tag/0DZT}{Tag 0DZT}]{stacksproject}\\
    \hline $\sX_g^{\rm c}$ & Chow semistable 1-cycles in $\bP^{g-1}$ of degree $2g-2$ in the closure of the locus parameterizing smooth canonical curves of genus $g$ &  Definition \ref{def:chowss}\\
    \hline $\sX_t$ & VGIT quotient stack of $(2,3)$ complete intersections in $\bP^3$ of slope $t$ &  Definition \ref{def:23VGIT} \\
    \hline $\ocM_4^{\CY}$ & slc pairs $(X, \frac{2}{3}D)$ admitting a $\bQ$-Gorenstein smoothing to $(\bP^1 \times \bP^1, \frac{2}{3}C_{3,3})$, for $C_{3,3}$ a smooth $(3,3)$ curve &  Definition \ref{def:CY4stack} \\
    \hline $\ocM_4^{\K}$ & pairs $(X,\frac{2}{3}D) \in \ocM_4^{\CY}(\bk)$ such that $(X, (\frac{2}{3}-\epsilon)D)$ is K-semistable for $0<\epsilon \ll 1$ &  Definition \ref{def:KAndKSBA} \\
    \hline $\ocM_4^{\KSBA}$ & pairs $(X,\frac{2}{3}D) \in \ocM_4^{\CY}(\bk)$ such that $(X, (\frac{2}{3}+\epsilon)D)$ is slc for $0<\epsilon \ll 1$ &  Definition \ref{def:KAndKSBA} \\
    \hline $\sH_g$ & genus $g$ hyperelliptic curves &  Definition \ref{def:hyp-rib-stacks} \& Theorem \ref{thm:hyp-stack} \\
    \hline $\sR_g$ & genus $g$ rational ribbons &  Definition \ref{def:hyp-rib-stacks} \& Theorem \ref{thm:rib-stack} \\
    \hline $\Curves_{g,1}$ & curves $C \in \Curves_g^{\Gor}(\bk)$ such that $\omega_C$ is ample and basepoint free &  Definition \ref{def:curves_g,1} \\
    \hline $\osX_{\!g}$ & curves in the closure of $\cM_g$ in $\Curves_{g,1}$, i.e. smoothable Gorenstein genus $g$ curves with $\omega_C$ ample and basepoint free & Definition \ref{def:sX_g} \\
    \hline $\sX_g$ & smoothable Gorenstein genus $g$ curves with $\omega_C$ ample and basepoint free whose Hilbert--Chow image is a Chow-semistable 1-cycle in $\bP^{g-1}$ & Definition \ref{def:sX_g} \\
    \hline $\sX_g^\circ$ & smoothable lci genus $g$ curves with $\omega_C$ ample and basepoint free whose Hilbert--Chow image is a Chow-semistable 1-cycle in $\bP^{g-1}$ whose Chow polystable degeneration is either a doubled rational normal curve or the cycle theoretic image of a reduced lci curve with very ample canonical divisor & Definition-Proposition \ref{dp:Xgcirc}; a smooth open substack of $\sX_g$ that admits a good moduli space (if $g = 4$, $\sX_4^{\circ}$ will be all of $\sX_4$) \\
    \hline $\sX^{\rm c} := \sX_4^{\rm c}$ & Chow semistable 1-cycles in $\bP^{3}$ of degree $6$ in the closure of the locus parameterizing smooth canonical curves of genus $g$ &  Definition \ref{def:chowss}\\
    \hline $\sX := \sX_4 = \sX_4^\circ$ & curves in $\sX_g$ for $g = 4$, i.e. smoothable Gorenstein genus $4$ curves with $\omega_C$ ample and basepoint free whose Hilbert--Chow image is a Chow-semistable 1-cycle in $\bP^{3}$  & Definition \ref{def:sX_g}, Definition-Proposition \ref{dp:Xgcirc}, Lemma \ref{lem:Xcirc=X} \\
    \hline $\sU$ & curves in $\sX$ whose Chow 1-cycle does not admit special degenerations to $[C_{2A_5}]$ or $[C_D]$ & Definition \ref{def:sUsV}  \\
    \hline $\sV$ & curves in $\sX$ whose Chow 1-cycle does not admit special degenerations to $[C_{A,B}]$ for any $(A,B) \in \bk^2 \setminus \{(0,0)\}$ & Definition \ref{def:sUsV} \\
    \hline $\sX^+$ & curves $C \in \sX(\bk)$ such that either $C$ has $A_{\le 4}$ singularities or $C \in \sV$ & Definition \ref{def:sX+} \\
    \hline $\sU^+$ & curves $C \in \sX^+(\bk) \cap \sU(\bk)$ & Definition \ref{def:sX+} \\    
    \hline $\sW^{\CY}$ & pairs $(X, \frac{2}{3}D) \in \ocM_4^{\CY}(\bk)$ such that $X$ is a quadric or $S_{2A_5}$ & Proposition \ref{prop:open-substack-CY} \\ 
    \hline $\sW^{\K}$ & curves $C \in \sW^{\CY}(\bk) \cap \ocM_4^\K(\bk)$ & Definition \ref{def:sWK} \\
    \hline $\sW^{\KSBA}$ & curves $C \in \sW^{\CY}(\bk) \cap \ocM_4^\KSBA(\bk)$ & Definition \ref{def:sWK} \\
    \hline $\sW$ & curves $C$ in the image of the forgetful map $\sW^{\CY} \to \Curves_4^{\lci+}$ & Definition \ref{def:sWK}; isomorphic to $\sW^{\CY}$ by Theorem \ref{thm:CY-forget-open} and $\sV \subset \sW$ by Theorem \ref{thm:CMJL-Chow-GIT} and Remark \ref{rmk:GIT-stable-properties} \\ 
    \hline $\sW^-$ & curves $C$ in the image of the forgetful map $\sW^{\K} \to \Curves_4^{\lci+}$ & Definition \ref{def:sWK}; isomorphic to $\sW^{\K}$ by Theorem \ref{thm:CY-forget-open} \\ 
    \hline $\sW^+$ & curves $C$ in the image of the forgetful map $\sW^{\KSBA} \to \Curves_4^{\lci+}$ & Definition \ref{def:sWK}; isomorphic to $\sW^{\KSBA}$ by Theorem \ref{thm:CY-forget-open} \\
    \hline 
    $\sV_1^-$ & curves in $\sV$ whose Chow 1-cycle does not admit special degenerations to $[C_D]$, or curves in the pre-image of $\fX^{\mathrm c} \setminus (\Gamma \cup \{[C_D]\})$ the good moduli space morphism $\sX \to \fX^{\mathrm c}$ & Definition \ref{def:V1minus-V2} \\ 
    \hline $\sV_1$ & denoting by $\Psi: \sW^{\CY} \subset \ocM_4^{\CY} \to \sW$ the forgetful map, curves in the image of the saturation of $\Psi^{-1}(\sV_1^-)$ under $\Psi$ & Definition \ref{def:sV1andsV1+} \\ 
    \hline $\sV_1^+$ & curves in $\sV_1(\bk) \cap \sW^+(\bk)$, i.e. curves in $\sV_1$ whose preimage in $\ocM_4^{\CY}$ is a KSBA stable pair  & Definition \ref{def:sV1andsV1+} \\
    \hline $\sV_2$ & curves in $\sV$ whose Chow 1-cycle does not admit special degenerations to $[C_{2A_5}]$, or curves in the pre-image of $\fX^{\mathrm c} \setminus (\Gamma \cup \{[C_{2A_5}]\})$ the good moduli space morphism $\sX \to \fX^{\mathrm c}$ & Definition \ref{def:V1minus-V2} \\ 
    \hline $\sV_2^+$ & curves in $\sV_2(\bk) \cap \sW^+(\bk)$, i.e. curves in $\sV_2$ whose preimage in $\ocM_4^{\CY}$ is a KSBA stable pair   & Definition \ref{def:V2+}\\
    \hline $\ocM_4(\alpha)$ for $\alpha \in [\frac{5}{9}, \frac{2}{3})$ & $\bk$-points of $\begin{cases}
    \sX &\textrm{ if }\alpha=\frac{5}{9} \\
    \sU^+\cup \sV_1^{-} \cup \sV_2^+ & \textrm{ if }\alpha\in (\frac{5}{9},\frac{19}{29})\\
    \sU^+\cup \sV_1 \cup \sV_2^+ & \textrm{ if }\alpha = \frac{19}{29}\\
    \sU^+\cup \sV_1^{+} \cup \sV_2^+ & \textrm{ if }\alpha\in (\frac{19}{29},\frac{2}{3})
    \end{cases}$ &  Definition \ref{def:newmodels}; this stack admits $\oM_4(\alpha)$ as its good moduli space \\ \hline
\end{longtable}
    
\end{center}

\bibliographystyle{alpha}
\bibliography{p1xp1}

\end{document}